\documentclass[a4paper,11pt, reqno]{amsart}

\usepackage{Maclane}

\title{Regular models of hyperelliptic curves}
\author{Simone Muselli}
\address{University of Bristol, Bristol, UK.}
\subjclass[2010]{11G20 (Primary), 14H45, 14M25 (Secondary). 
Keywords: Hyperelliptic curves, models of curves, MacLane valuations.}

\begin{document}

\maketitle

\begin{abstract}
    Let $K$ be a complete discretely valued field of residue characteristic not $2$ and $O_K$ its ring of integers. We explicitly construct a regular model over $O_K$ with strict normal crossings of any hyperelliptic curve $C/K:y^2=f(x)$. For this purpose, we introduce the new notion of \textit{MacLane cluster picture}, that aims to be a link between clusters and MacLane valuations. 
\end{abstract}

\section{Introduction}

In this paper we construct regular models of hyperelliptic curves over discrete valuation rings with residue characteristic different from $2$. The understanding of regular models is essential to describe the arithmetic of curves and for example finds application in the study of the Birch \& Swinnerton-Dyer conjecture over global fields.

\subsection{Overview}
Let $K$ be a complete discretely valued field, with ring of integers $O_K$. 
Given a connected smooth projective curve $C/K$, a \textit{regular model} of $C$ over $O_K$ is an integral regular proper flat scheme $\m C\rightarrow O_K$ of dimension $2$ with generic fibre isomorphic to $C$. The main result of this work can be presented as follows: 

\emph{Suppose that the residue characteristic of $K$ is not $2$. Let $C/K:y^2=f(x)$ be a hyperelliptic curve. From the MacLane clusters for $f$ we determine a regular model of $C$ over $O_K$ with strict normal crossings.}

The \textit{MacLane clusters} for a separable polynomial $f\in K[x]$ are a new notion we introduce in this paper (see \S\ref{subsec:MainResult} for more details). It has connections with other objects used for the study of regular models: clusters \cite{D2M2}, rational clusters \cite{Mus}, Newton polytopes \cite{Dok}, and MacLane valuations \cite{OW}. Like (rational) clusters, MacLane clusters define nice and clear invariants from which one can give a result in a closed form. In fact, one can see that rational clusters are MacLane clusters of degree $1$. On the other side, the construction of our model can be implemented from the algorithmic nature of the approaches based on Newton polytopes and MacLane valuations.

\subsection{Main result}\label{subsec:MainResult}
Let $K$ be a complete discretely valued field, with normalised discrete valuation $v_K$, ring of integers $O_K$, and residue field $k$. Let $\bar K$ be an algebraic closure of $K$, extend $v_K$ to $\bar K$. Assume $\mathrm{char}(k)\neq 2$. Let $C/K$ be a hyperelliptic curve, i.e. a geometrically connected smooth projective curve of genus $\geq 1$, double cover of $\P^1_K$. We can fix a Weierstrass equation $C:y^2=f(x)$ where \[f(x)=c_f\textstyle\prod_{r\in\roots}(x-r)\in K[x],\quad c_f\in K,\]
such that $v_K(r)>0$ for all $r\in\roots$.

\begin{defn}
Let $\hat\Q=\Q\cup\{\infty\}$. Given a monic irreducible polynomial $g\in K[x]$ and an element $\lambda\in\hat\Q$, the \textit{discoid} $D(g,\lambda)$ is the set
\[D=D(g,\lambda)=\{\alpha\in\bar K\mid v_K(g(\alpha))\geq \lambda\}\subset \bar K.\]
For any $r\in\roots$, denote by $D\wedge r$ the smallest discoid containing $D$ and $r$.

Define $\deg D=\min\{d\in\Z_+\mid D=D(g,\lambda), \,\deg g=d\}$.
\end{defn}



To each discoid we can associate a pseudo-valuation (Appendix \ref{appendix:pseudo-valuations}) $v_D:K[x]\rightarrow \hat \Q$ defined by
\[v_D(f)=\textstyle\inf_{\alpha\in D}v_K(f(\alpha)).\]
The map $D\mapsto v_D$ is injective. Therefore if $v=v_D$ denote $D_v=D$ and $d_v=\deg D$.


\begin{defn}
A \textit{MacLane cluster} is a pair $(\s,v)$ where $\s\subseteq\roots$, and $v=v_D$ for some discoid $D$, such that
\begin{enumerate}
    \item $\s=D\cap\roots\neq \varnothing$;
    \item if $\s=D'\cap\roots$ for a discoid $D'\subseteq D$ then $\deg D'>\deg D$.
\end{enumerate}

The \textit{degree} of $(\s,v)$ is the quantity $d_v$.
\end{defn}

\begin{defn}
For any MacLane clusters $(\s,v),(\t,w)$ we say:
\begin{center}
\begin{tabular}{|l@{ if }l|}
\hline
$(\s,v)$ proper, & $|\s|>d_v$\cr
$(\t,w)\subseteq(\s,v)$, & $D_w\subseteq D_v$\cr
$(\t,w)$ is a child of $(\s,v)$, & $(\t,w)\subsetneq(\s,v)$ is a maximal subcluster \cr
$(\s,v)$ degree-minimal, & $(\s,v)$ has no proper children of degree $d_v$\cr
\hline
\end{tabular}
\end{center}
We write $(\t,w)<(\s,v)$
for a child $(\t,w)$ of $(\s,v)$.
\end{defn}

For the remainder of this subsection we also assume $k$ algebraically closed. This additional condition is not necessary for the construction of the model but it simplifies the statement of Theorem \ref{thm:SNCModelIntroductionTheorem}.

Let $\Sigma$ be the set of proper MacLane clusters.


\begin{nt}
Let $\m P\subset K[x]$ be the subset of monic irreducible polynomials. 
For any $d\in\Z_+$, denote $\m P_{\leq d}=\{g\in \m P\mid\deg g\leq d\}$.
\end{nt}


\begin{defn}[\ref{defn:QuantitiesForTheoremsOnModelsDefinition}]
Let $(\s,v)\in\Sigma$. Define
the following quantities:
\begin{center}
\begin{tabular}{|l@{$=\>\,$}l|}
\hline
    $\lambda_v$       & $\max_{g\in\m P_{\leq d_v}}\min_{r\in\s}v_K(g(r))$, called \textit{radius}\cr
    $b_v$            & denominator of $\lambda_vd_v$\cr
    $e_{v}$       & $b_vd_v$\cr
    $\nu_v$     & $v_K(c_f)+\sum_{r\in\roots}\big(\lambda_{v_{D_v\wedge r}}/d_{v_{D_v\wedge r}}\big)$\cr
    $n_v$            & $1$ if $e_v\nu_v$ odd, $2$ if $e_v\nu_v$ even\cr
    $m_v$            & $2e_v/n_v$\cr
    $\i_v$ & $|\s|/d_v$\cr
    $p_v$            & $1$ if $\i_v$ is odd, $2$ if $\i_v$ is even\cr
    $s_v$            & $\frac 12(\i_v\lambda_v+p_v\lambda_v-\nu_v)$\cr
    $\gamma_v$       & $2$ if $\i_v$ is even and $\nu_vd_v\!-\!|\s|\lambda_v$ is odd, $1$ otherwise\cr
    $\delta_v$          & $1$ if $(\s,v)$ is degree-minimal, $0$ otherwise\cr
    $p_v^0$          & $1$ if $\delta_v=1$ and $d_v=\min_{r\in\s}[K(r):K]$, $2$ otherwise\cr
    $s_v^0$          & $-\nu_v/2+\lambda_v$\cr
    $\gamma_v^0$     & $2$ if $p_v^0=2$ and $\nu_vd_v$ is an odd integer, $1$ otherwise\cr
\hline
\end{tabular} 
\end{center}

Let $\ell_v\in\Z$, $0\leq \ell_v<b_v$ such that $\ell_v\lambda_vd_v-\frac{1}{b_v}\in\Z$. Define
\[\tilde v=\big\{(\t,w)\in\Sigma\mid (\t,w)<(\s,v)\text{ and }\tfrac{|\t|}{e_v}-\ell_v\nu_vd_w\notin 2\Z\big\}.\]
Let $c_v^0=1$ if 
$\tfrac{2-p_v^0}{b_v}-\ell_v\nu_vd_v\notin 2\Z$, and $c_v^0=0$ otherwise. Define
\[u_v=\smaller{\dfrac{|\s|-\sum_{(\t,w)<(\s,v)}|\t|-d_v(2-p_v^0)}{e_v}}+|\tilde v|+ \delta_vc_v^0,\]
where the sum runs through the proper children $(\t,w)$ of $(\s,v)$.
The \textit{genus $g(v)$} of a MacLane cluster $(\s,v)\in\Sigma$ is defined as follows:
\begin{itemize}
    \item If $n_v=1$, then $g(v)=0$.
    \item If $n_v=2$, then $g(v)=\max\{\lfloor(u_v-1)/2\rfloor,0\}$.

\end{itemize}
\end{defn}

\begin{nt}
Let $\alpha,a,b\in\Z$, with $\alpha>0$, $a>b$, and fix $\frac{n_i}{d_i}\in\Q$ so that
    \[\alpha a=\sfrac{n_0}{d_0}>\sfrac{n_1}{d_1}>\ldots>\sfrac{n_r}{d_r}>\sfrac{n_{r+1}}{d_{r+1}}=\alpha b,\quad\text{with\scalebox{0.9}{$\quad\begin{vmatrix}n_i\!\!\!&n_{i+1}\cr d_i\!\!\!&d_{i+1}\cr\end{vmatrix}=1$}},\]
    and $r$ minimal. We write $\P^1(\alpha,a,b)$ for a chain of $\P^1$s of length $r$ and multiplicities $\alpha d_1,\dots,\alpha d_r$. Denote by $\P^1(\alpha,a)$ the chain $\P^1(\alpha,a,\lfloor\alpha a-1\rfloor/\alpha)$.
\end{nt}

The following theorem describes the special fibre of the regular model of a hyperelliptic curve $C/K$ with strict normal crossings we construct in \S\ref{sec:ModConstr}, when $k$ algebraically closed and $\mathrm{char}(k)\neq 2$. See Definition \ref{defn:QuantitiesForTheoremsOnModelsDefinition} and Theorem \ref{thm:SNCModel} for a more general statement which does not require $k$ algebraically closed.

\begin{thm}[Regular SNC model]\label{thm:SNCModelIntroductionTheorem}
Assume $\mathrm{char}(k)\neq 2$. Suppose $k$ algebraically closed. Let $C/K$ be a hyperelliptic curve.
Then we can explicitly construct a regular model with strict normal crossings $\mathcal{C}/O_{K}$ of $C$ (\S\ref{sec:ModConstr}). Its special fibre $\mathcal{C}_s/k$ is given as follows.

\begin{enumerate}[label=(\arabic*), leftmargin=1cm]
\item Every $(\s,v)\in\Sigma$ gives a $1$-dimensional closed subscheme $\Gamma_v$ of multiplicity 
$m_v$. If $n_v=2$ and 
$u_v=0$,
then $\Gamma_v$ is the disjoint union of $\Gamma_v^{-}\simeq\P^1$ and $\Gamma_v^{+}\simeq\P^1$, otherwise $\Gamma_v$ is a smooth integral curve of genus $g(v)$ (write $\Gamma_v^{-}=\Gamma_v^{+}=\Gamma_v$ in this case). 

\item Every $(\s,v)\in\Sigma$ with $n_v=1$ gives 
\[\frac{1}{e_v}\bigg(|\s|-\textstyle\sum_{\substack{(\t,w)\in\Sigma\\(\t,w)<(\s,v)}}|\t|+d_v(p_v^0-2)\bigg)\] 
open-ended $\P^1$s of multiplicity $e_v$ from $\Gamma_v$.

\item Finally, for any $(\s,v)\in\Sigma$ draw the following chains of $\P^1$s:
\footnotesize
\cellspacetoplimit4pt
\cellspacebottomlimit4pt
\begin{center}
\begin{tabular}{|Sc|Sc|Sc|Sc|}
\hline
\!\!\!Conditions \!\!\!&\!\!\! Chain \!\!\!&\!\!\! From \!\!\!&\!\!\! To\!\!\!\cr
\Xhline{1 pt}
\!\!\!$\delta_v=1$ \!\!\!&\!\!\! $\P^1(d_v\gamma_v^0,-s_v^0)$ \!\!\!&\!\!\! $\Gamma_v^-$ \!\!\!&\!\!\!\!\!\! open-ended\!\!\!\cr 
\hline
\!\!\!$\delta_v=1$, $p_v^0/\gamma_v^0=2$ \!\!\!&\!\!\! $\P^1(d_v\gamma_v^0,-s_v^0)$ \!\!\!&\!\!\! $\Gamma_v^+$ \!\!\!&\!\!\!\!\!\! open-ended\!\!\!\cr
\hline
\!\!\!$(\s,v)<(\t,w)$ \!\!\!&\!\!\! $\P^1(d_v\gamma_v,s_v,s_v-\tfrac{p_v}{2}(\lambda_v-\tfrac{d_v}{d_w}\lambda_{w}))$ \!\!\!&\!\!\! $\Gamma_v^-$ \!\!\!&\!\!\!\!\!\! $\Gamma_{w}^-$\!\!\!\cr
\hline
\!\!\!$(\s,v)<(\t,w)$, $p_v/\gamma_v=2$ \!\!\!&\!\!\! $\P^1(d_v\gamma_v,s_v,s_v-\tfrac{p_v}{2}(\lambda_v-\tfrac{d_v}{d_w}\lambda_{w}))$ \!\!\!&\!\!\! $\Gamma_v^+$ \!\!\!&\!\!\!\!\!\! $\Gamma_{w}^+$\!\!\!\cr
\hline
\!\!\!$(\s,v)$ maximal \!\!\!&\!\!\! $\P^1(d_v\gamma_v,s_v)$ \!\!\!&\!\!\! $\Gamma_v^-$ \!\!\!&\!\!\!\!\!\! open-ended\!\!\!\cr
\hline
\!\!\!$(\s,v)$ maximal, $p_v/\gamma_v=2$ \!\!\!\!&\!\!\! $\P^1(d_v\gamma_v,s_v)$ \!\!\!&\!\!\! $\Gamma_v^+$ \!\!\!&\!\!\! open-ended \!\!\!\cr
\hline
\end{tabular}
\end{center}\normalsize

\end{enumerate}
\end{thm}

\subsection{Example}
Let $p\neq 2$ be a prime number and let $\Q_p^{nr}$ be the maximal unramified extension of $\Q_p$ in $\bar\Q_p$. Let $f=(x^2-p)^3-p^5\in \Q_p[x]$ and $C/\Q_p^{nr}:y^2=f(x)$ a genus $2$ hyperelliptic curve. We can represent the set of MacLane clusters as

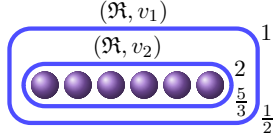
\[
\scalebox{2}{\clusterpicture            
   \Root {1} {first} {r1};
   \Root {} {r1} {r2};
   \Root {} {r2} {r3};
   \Root {} {r3} {r4};
   \Root {} {r4} {r5};
   \Root {} {r5} {r6};
   \ClusterLDName w[2][\frac{5}{3}][(\roots,v_2)] = (r1)(r2)(r3)(r4)(r5)(r6);
   \ClusterLDName v[1][\frac{1}{2}][(\roots,v_1)] = (w)(wn);
\endclusterpicture}
\]
where the bullet points denote the roots of $f$, the circles are the proper MacLane clusters and the superscripts and the subscripts are respectively the degree and the radius of the corresponding cluster. In fact, there are two proper MacLane clusters: 
\begin{enumerate}[label=(\roman*)]
    \item $(\roots,v_1)$, where $D_{v_1}=D(x,1/2)$;
    \item $(\roots,v_2)$, where $D_{v_2}=D(x^2-p,5/3)$.
\end{enumerate}
Note that $\min_{r\in\roots}[K(r):K]=6$ since $f$ is irreducible. We have
\begin{center}
\small
\begin{tabular}{|l
|c|c|c|c|c|c|c|c|c|c|c|c|c|c|}
\hline
    & $b_v$ & $e_{v}$ & $\nu_v$& $n_v$ & $m_v$ & $\i_v$ &$p_v$&$s_v$
    &$\gamma_v$&$\delta_v$&$p_v^0$&$s_v^0$&$\gamma_v^0$&$g(v)$\cr
\hline
    $v_1$ 
    & $2$ & $2$ & $3$ & $2$ & $2$ & $6$ & $2$ & $\nicefrac{1}{2}$ & $1$ & $1$ & $2$ & $-1$ & $2$ & $0$\cr
\hline
    $v_2$ 
    & $3$ & $6$ & $10$ & $2$ & $6$ & $3$ & $1$ & $\nicefrac{-5}{3}$ & $1$ & $1$ & $2$ & $\nicefrac{-10}{3}$ & $1$ & $0$\cr
\hline
\end{tabular} 
\end{center}
By Theorem \ref{thm:SNCModelIntroductionTheorem}, the special fibre of the regular model $\m C$ we construct is


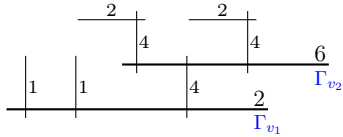
\begin{center}
\pbox[c]{20cm}{
\begin{tikzpicture}[xscale=1,yscale=0.9,
  l1/.style={shorten >=-1.3em,shorten <=-0.5em,thick},
  l2/.style={shorten >=-0.3em,shorten <=-0.3em},
  lfnt/.style={font=\tiny},
  leftl/.style={left=-3pt,lfnt},
  rightl/.style={right=-3pt,lfnt},
  mainl/.style={scale=0.8,above left=-0.17em and -1.5em},
  mainleftl/.style={scale=0.8,above right=-0.17em and -1.5em},
  abovel/.style={above=-2.5pt,lfnt},
  facel/.style={scale=0.7,blue,below right=-0.5pt and 6pt},
  faceleftl/.style={scale=0.7,blue,below left=-0.5pt and 6pt},
  redbull/.style={red,label={[red,scale=0.6,above=-0.17]#1}}]
\draw[l1] (-.22,0.00)--(2.53,0.00) node[mainl] {2} node[facel] {$\Gamma_{v_1}$};
\draw[l2] (0.50,0.00)--node[rightl] {1} (0.50,0.66);
\draw[l1] (1.30,0.66)--(3.33,0.66) node[mainl] {6} node[facel] {$\Gamma_{v_2}$};
\draw[l2] (1.96,0.00)--node[rightl] {4} (1.96,0.66);
\draw[l2] (1.30,0.66)--node[rightl] {4} (1.30,1.32);
\draw[l2] (0.64,1.32)--node[abovel] {2} (1.30,1.32);
\draw[l2] (2.76,0.66)--node[rightl] {4} (2.76,1.32);
\draw[l2] (2.10,1.32)--node[abovel] {2} (2.76,1.32);
\draw[l2] (-.16,0.00)--node[rightl] {1} (-.16,0.66);
\end{tikzpicture}
}
\end{center}
where all irreducible components have genus $0$. In fact, by computing the self-intersections of all irreducible components, we see that $\m C$ is the minimal regular model of $C$ (\cite[Theorem 9.3.8]{Liu}).

\subsection{Related works of other authors}\label{subsec:RelatedWorks}
Let $K$ be a
discretely valued field of residue characteristic not $2$ and let $C/K$ be a hyperelliptic curve.
In this subsection we want to present previous works studying regular models of $C$, possibly under some extra conditions. Note that some of the results cited below may apply to more general curves and fields.

In genus $1$ there is a complete characterisation of (minimal) regular models of $C$ (see for example \cite[IV.8.2]{Sil2} when the residue field of $K$ is perfect).
A description of all special fibre configurations is also given by Namikawa and Ueno \cite{NU} and Liu \cite{Liu94bis} for genus $2$ curves, when $K=\C(t)$. 


If $C$ is semistable over some tamely ramified extension $L/K$, then \cite{FN} describes the special fibre of the minimal regular model of $C$ with strict normal crossings. If, in addition, $L=K$ is a local field, in \cite{D2M2} we can also see an explicit construction of the model itself.

T.\ Dokchitser in \cite{Dok} shows that a certain toric resolution of $C$ gives a regular model in case of $\Delta_v$-regularity (\cite[Definition 3.9]{Dok}). This condition is rephrased in terms of clusters in \cite[Corollary 3.25]{Mus}.

Finally, \cite{Mus} constructs the minimal regular model with normal crossings if $C$ has almost rational cluster picture. One can see that the latter condition is equivalent of requiring that all MacLane clusters have degree $1$.




\section{MacLane valuations}
In this section we summarise definitions and results on MacLane valuations. Our main references are \cite{KW}, \cite{Mac}, \cite{OS} and \cite{Rut}.

Let $K$ be a complete discretely valued field, with normalised discrete valuation $v_K$, ring of integers $O_K$ and residue field $k$. Let $\bar K$ be an algebraic closure of $K$ and let $K^s$ be the separable closure of $K$ in $\bar K$. Let $G_K=\Gal(K^s/K)$ be the absolute Galois group of $K$.

Let $\pV$ denote the set of the discrete pseudo-valuation\footnote{See Appendix \ref{appendix:pseudo-valuations} for more details.} $v: K[x]\rightarrow \hat\Q$ extending $v_K$ and satisfying $v(x)\geq 0$. Let $\V$ be the set of valuations in $\pV$. In other words, $\V$ consists of those pseudo-valuations $v\in\pV$ satisfying $v^{-1}(\infty)={0}$. 
We can equip $\pV$ with a natural partial order:
\[v\geq w\quad\text{ if and only if }\quad v(g)\geq w(g)\text{ for all }g\in K[x].\]
The partially ordered set $\pV$ has a least element $v_0$, called \textit{Gauss valuation}, defined by
\[v_0(a_mx^m+\cdots+a_1x+a_0)=\min_{i}v_K(a_i)\qquad (a_i\in K).\]
Note that $v_0$ is a valuation, i.e.\ $v_0\in \V$.

Every $v\in\V$ can be extended to a valuation $K(x)\rightarrow \hat\Q$, that will also be denoted by $v$.

\begin{defn}
For every $v\in\V$ define
\begin{center}
\begin{tabular}{ll}
     $\Gamma_v$ & the valuation group of $v$\\
     $e_{v}$ & the index $[\Gamma_v:\Z]$\\
     $\Av_v$ & the residue ring of $v$\\
     $\F_v$ & the residue field of $v$
\end{tabular}
\end{center}
\end{defn}

\begin{defn}
Let $v\in\V$. For any $g,h\in K(x)$ we say that 
\begin{itemize}
\item $g$ is \textit{$v$-equivalent} to $h$, denoted $g\veq_v h$, if $v(g-h)>v(g)$.
\item $g$ is \textit{$v$-divisible} by $h$, denoted $h\vdiv_v g$, if there exists $q\in K[x]$ such that $g\veq_v qh$.
\end{itemize}
\end{defn}

Let $v\in \V$. For any $\alpha\in\Gamma_v$, define
\[O_v(\alpha)=\{g\in K[x]\mid v(g)\geq \alpha\}, \qquad O_v^+(\alpha)=\{g\in K[x]\mid v(g)>\alpha\}.\]
The \textit{graded algebra of $v$} is the integral domain
\[Gr(v):=\bigoplus_{\alpha\in\Gamma_v}\Av_v(\alpha),\quad\text{ where } \Av_v(\alpha)=O_v(\alpha)/O_v^+(\alpha).\]
The canonical homomoprhism $k\rightarrow \Av_v$ equips $\Av_v$ and $Gr(v)$ with a $k$-algebra structure. 
There is a natural map $H_v: K[x]\rightarrow Gr(v)$ given by $H_v(0)=0$ and
\[H_v(g)=g+O_v^+(v(g))\in \Av_v(v(g)),\]
when $g\neq 0$. The map $H_v$ satisfies the following properties
\begin{enumerate}
    \item $f\veq_v g$ if and only if $H_v(f)=H_v(g)$,\label{item:Hvprop1}
    \item $H_v(fg)=H_v(f)H_v(g)$,\label{item:Hvprop2}
\end{enumerate}
for $f,g\in K[x]$.
Let $U_v\subseteq K[x]^\ast$ be the multiplicative set 
\[U_v=\{g\in K[x]\mid H_v(g)\text{ is a unit in } Gr(v)\}\]
and let $\loc{v}\subseteq K(x)$ be the localisation of $K[x]$ by $U_v$. We extend $H_v$ to a map $\loc{v}\rightarrow Gr(v)$ by taking $\nicefrac{g}{u}\mapsto H_v(g)H_v(u)^{-1}\in \Av_v(v(\nicefrac{g}{u}))$, for any $g\in K[x], u\in {U_v}$. With a little abuse of notation we denote the extended map again by $H_v$. The properties (\ref{item:Hvprop1}), (\ref{item:Hvprop2}) of $H_v$ hold for all $f,g\in \loc{v}$.

\begin{defn}
We call $H_v: \loc{v}\rightarrow Gr(v)$ the \textit{residue map of $v$}.
\end{defn}

For any $\alpha\in\Gamma_v$, let
\[\loc{v}(\alpha)=\{g\in \loc{v}\mid v(g)\geq \alpha\}, \qquad \loc{v}^+(\alpha)=\{g\in \loc{v}\mid v(g)>\alpha\}.\]
Note that $H_v$ induces a birational map $\loc{v}(\alpha)/\loc{v}^+(\alpha)\rightarrow\Av_v(\alpha)$.

\begin{defn}
Let $v\in\V$. A monic polynomial $\phi\in K[x]$ is a \textit{key polynomial over $v$} if 
\begin{enumerate}[label=(\arabic*)]
    \item $\phi$ is \textit{$v$-irreducible}, i.e.\ if $\phi\vdiv_v ab$ then $\phi\vdiv_v a$ or $\phi\vdiv_v b$, for all $a,b\in K[x]$;
    \item $\phi$ is \textit{$v$-minimal}, i.e.\ if $\phi\vdiv_v a$ then $\deg a\geq\deg\phi$, for all $a\in K[x]$.
\end{enumerate}
Denote by $\KP v$ the set of key polynomials over $v$.
\end{defn}

\begin{rem}
Let $v\in\V$. Then $\KP v\subseteq O_K[x]$ (\cite[Corollary 1.10]{FGMN}).
\end{rem}

\begin{defn}[{\cite[Theorem 4.2]{Mac}}]
Let $v\in\V$. Let $\phi\in \KP v$ and $\lambda\in\hat\Q$, $\lambda> v(\phi)$. 
Define a pseudo-valuation $w\in \pV$, denoted $w=[v,v(\phi)=\lambda]$, by
\[w(a_m\phi^m+\cdots+a_1\phi+a_0)=\min_{i}(v(a_i)+i\lambda)\qquad a_i\in K[x], \,\deg a_i<\deg\phi.\]
We call $w$ the \textit{augmentation} of $v$ with respect to $(\phi,\lambda)$.
\end{defn}

\begin{rem}\label{rem:(i)augmentation>(ii)RadiusDegreeUniqueness}
Let $w=[v,v(\phi)=\lambda]$ be an augmentation of $v$. Then
\begin{enumerate}[label=(\roman*)]
    \item $w> v$ by \cite[Propositions 1.7, 1.9]{FGMN}.\label{item:augmentation>}
    \item $\lambda$ and $\deg\phi$ are uniquely determined by $w$, but not the key polynomial $\phi$ itself in general (see \cite[Remark 2.7]{KW}).\label{item:RadiusDegreeUniqueness}
\end{enumerate}
\end{rem}

\begin{defn}
A pseudo-valuation $v\in\pV$ is \textit{MacLane} if it can be attained after a finite number of augmentations starting with $v_0$. Write
\[v=[v_0,v_1(\phi_1)=\lambda_1,\dots,v_m(\phi_m)=\lambda_m],\quad m\in\N,\]
where  $v_i=[v_{i-1},v_i(\phi_i)=\lambda_i]$ is an augmentation of $v_{i-1}$ for any $i=1,\dots,m$, and $v_m=v$. We will call $\phi_m$ a \textit{centre} of $v$ and $\lambda_m$ the \textit{radius} of $v$.\footnote{By convention, if $v=v_0$, then 
any monic integral polynomial of degree $1$ is a centre 
of $v$ and $0$ is the radius of $v$.}

Let $\pM\subset\pV$ denote the set of MacLane pseudo-valuations and let $\M\subset\V$ denote the set of MacLane valuations.
\end{defn}

\begin{rem}
There are different equivalent characterisations for the sets $\M$ and $\pM$ (see \cite[\S2]{KW}). In fact,
\begin{enumerate}[label=(\roman*)]
\item $\M$ consists of those valuations $v\in \V$ with residue field $\F_v$ of transcendence degree $1$ over $k$;
\item all infinite pseudo-valuations $v\in\pV$ are Maclane.
\end{enumerate}
\end{rem}

\begin{nt}
Let $v\in\pM$. Remark \ref{rem:(i)augmentation>(ii)RadiusDegreeUniqueness}\ref{item:RadiusDegreeUniqueness} implies that the radius of $v$ is uniquely determined by $v$. We will denote it by $\lambda_v$.
\end{nt}

\begin{defn}
Let $v\in\pM$. An \textit{augmentation chain (of length $m$)} for $v$ is a tuple
\begin{equation}\label{eqn:augm_chain}
((\phi_1,\lambda_1),\dots,(\phi_m,\lambda_m)),
\end{equation}
where $v=[v_0,v_1(\phi_1)=\lambda_1,\dots,v_m(\phi_m)=\lambda_m]$.
We say that (\ref{eqn:augm_chain}) is 
\begin{enumerate}
    \item a \textit{MacLane chain} if $\phi_{i+1}\not\veq_{v_i}\phi_i$ for any $i=1,\dots,m-1$.
    \item \textit{minimal} if $\deg\phi_{i+1}>\deg\phi_i$ for any $i=1,\dots,m-1$.
\end{enumerate}
\end{defn}


For any augmentation chain (\ref{eqn:augm_chain}) we have
\[\deg\phi_1\mid\deg\phi_2\mid\dots\mid\deg\phi_m,\]
by \cite[Lemma 2.10]{FGMN}. If it is a MacLane chain, then $v(\phi_i)=\lambda_i$ for any $i=1,\dots,m$. In particular, $\Gamma_v=\lambda_1\Z+\dots+\lambda_m\Z$.

\begin{rem}\label{rem:MacChainToMinimal}
Let $v\in\pM$.
\begin{enumerate}
\item A minimal augmentation chain is a Maclane chain. 
\item From any MacLane chain 
$((\phi_1,\lambda_1),\dots,(\phi_m,\lambda_m))$
for $v$, we can find a minimal augmentation chain for $v$ by removing the pairs $(\phi_i,\lambda_i)$ with $\deg\phi_i=\deg\phi_{i+1}$, for $i=1,\dots,m-1$ (\cite[Lemma 15.1]{Mac}, \cite[Lemma 3.4]{FGMN}).
\end{enumerate}
\end{rem}

\begin{nt}
We will denote an augmentation chain (\ref{eqn:augm_chain}) by
\[[v_0,v_1(\phi_1)=\lambda_1,\dots,v_m(\phi_m)=\lambda_m],\]
where $v_i=[v_{i-1}, v_i(\phi_i)=\lambda_i]$ for all $i=1,\dots,m$.
\end{nt}

\begin{defn}
Let $v\in\pM$ given by a MacLane chain
\begin{equation}\label{eqn:MacLaneChainfordegreedepth}
    [v_0,v_1(\phi_1)=\lambda_1,\dots,v_m(\phi_m)=\lambda_m].
\end{equation}
\begin{enumerate}[label=(\alph*)]
    \item The \textit{degree of $v$}, denoted $\deg v$, is the positive integer $\deg\phi_m$. 
    \item If (\ref{eqn:MacLaneChainfordegreedepth}) is minimal, then $m$ is said the \textit{depth of $v$}.
\end{enumerate}
The degree and the depth of $v$ are independent of the chosen MacLane chain (\ref{eqn:MacLaneChainfordegreedepth}) by \cite[Proposition 3.6]{FGMN}.
\end{defn}

Note that if $v\in\M$ then $\deg v\mid\deg\phi$ for any $\phi\in\KP v$.

\begin{defn}
Let $v\in \M$. A key polynomial $\phi\in\KP v$ is said 
\begin{enumerate}
    \item \textit{proper} if $v$ has a centre $\phi_v\not\veq_v\phi$.
    \item \textit{strong} if $v=v_0$ or $\deg\phi>\deg v$.
\end{enumerate}
\end{defn}

\begin{lem}\label{lem:centre}
Let $w\in\pM$. A polynomial $\phi\in K[x]$ is a centre of $w$ if and only if $\phi\in\KP w$ and $\deg w=\deg\phi$. Furthermore, if $w=[v,w(\phi)=\lambda]$, then any two centres of $w$ are $v$-equivalent.
\proof
Let $v\in\M$ such that $w=[v,w(\phi_w)=\lambda_w]$. If $\phi\in K[x]$ is a centre of $w$ then $\phi\in\KP w$ by \cite[Proposition 1.7(4)]{FGMN} and $\deg w=\deg\phi$ from Remark \ref{rem:(i)augmentation>(ii)RadiusDegreeUniqueness}\ref{item:RadiusDegreeUniqueness}. 
Conversely, suppose $\phi\in\KP w$ and $\deg\phi=\deg w$. From the $w$-minimality of $\phi$ and $\phi_w$, one has $w(\phi)=\lambda_w$.
Hence
\[v(\phi-\phi_w)=w(\phi-\phi_w)\geq \lambda_w>v(\phi_w).\]
Therefore $\phi\veq_v\phi_w$. In particular, $\phi\in\KP v$ as $\deg\phi=\deg\phi_w$, and so $w=[v,w(\phi)=\lambda_w]$.
Thus $\phi$ is a centre of $w$.
\endproof
\end{lem}



\begin{defn}
Given a monic irreducible polynomial $\phi\in K[x]$ and an element $\lambda\in\hat\Q$, the \textit{discoid} of centre $\phi$ and radius $\lambda$ is the set
\[D=D(\phi,\lambda)=\{\alpha\in\bar K\mid v_K(\phi(\alpha))\geq \lambda\}\subset \bar K.\]
Let $\D$ denote the set of discoids.
\end{defn}

\begin{rem}\label{rem:Galois1disjoint2discoids}
Let $D=D(\phi,\lambda)$ be a discoid.
\begin{enumerate} 
    \item $D$ is finite if $\lambda=\infty$, while equals the union of the Galois orbits of a disc centred at a root of $\phi$ if $\lambda<\infty$ (\cite[Lemma 4.43]{Rut}). 
    \item For any $D'\in\D$ such that $D\cap D'\neq \varnothing$ either $D\subseteq D'$ or $D\subseteq D'$ (\cite[Lemma 4.44]{Rut}).
\end{enumerate}
\end{rem}

\begin{defn}
Given a MacLane pseudo-valuation $v$, define
\[D_v=\{\alpha\in\bar K\mid v_K(g(\alpha))\geq v(g)\quad\text{for all }g\in K[x]\}.\]
It is a discoid by the following lemma.
\end{defn}

\begin{lem}
If $v=[v_0,v_1(\phi_1),\dots,v_{m-1}(\phi_{m-1})=\lambda_{m-1}, v_m(\phi_m)=\lambda_m]$ is a MacLane pseudo-valuation, then $D_v=D(\phi_m,\lambda_m)$.
\proof
If $v\in\M$, then the lemma follows from \cite[Lemma 4.55]{Rut}. Suppose $v$ is an infinite MacLane pseudo-valuation. Then $\lambda_m=\infty$. Clearly $D_v\subseteq D(\phi_m,\lambda_m)$. Let $r\in D(\phi_m,\lambda_m)$, i.e. $r$ is a root of $\phi_m$. Let $g\in K[x]$. We want to show that $v_K(g(r))\geq v(g)$. If $\phi_m\mid g$, then $g(r)=0$ and $v(g)=\infty$, so $v_K(g(r))=v(g)$. If $\phi_m\nmid g$, then there is a sufficiently large $\lambda\in\Q$ such that $w(g)=v(g)$, with $w=[v_{m-1},w(\phi_m)=\lambda]$. Since $w\in\M$, we have $D(\phi_m,\lambda)=D_w$. But $r\in D(\phi_m,\lambda)$, and so $v_K(g(r))\geq w(g)=v(g)$.
\endproof
\end{lem}

\begin{thm}\label{thm:valdiscoid}
The map $\pM\rightarrow \D$ taking 
$v\mapsto D_v$
is well-defined, bijective, and inverts partial orders, i.e. for any $v,w\in\pM$ we have
\[w\geq v\quad\text{if an only if}\quad D_w\subseteq D_v.\]
Given a discoid $D$, then $D=D_v$, where $v$ is the MacLane pseudo-valuation given by $v(g)=\inf_{r\in D}v_K(g(r))$ for all $g\in K[x]$.
\proof
The result follows from \cite[Theorem 4.56]{Rut},
\cite[Remark 2.3]{KW}. 
\endproof
\end{thm}

\begin{lem}\label{lem:degdiscoid}
Let $v\in\pM$ and $D_v=D(g,\lambda)$ the associated discoid. Then $\deg v\leq\deg g$ and $v(g)\geq\lambda$.
\proof
Theorem \ref{thm:valdiscoid} implies that $\inf_{r\in D_v}v_K(g(r))=v(g)$. Then $v(g)\geq\lambda$. It follows that
\[D_v\subseteq D(g,v(g))\subseteq D(g,\lambda)=D_v.\]
Then $D_v=D(g,v(g))$. Suppose $\deg v>\deg g$ and let \[[v_0,v_1(\phi_1)=\lambda_1,\dots,v_m(\phi_m)=\lambda_m]\]
be a MacLane chain for $v$. Then $v_{m-1}<v$ but $v_{m-1}(g)=v(g)$. Therefore $D_v\subsetneq D_{v_{m-1}}\subseteq D(g,v(g))$, a contradiction. 
\endproof
\end{lem}

\begin{rem}
Lemma \ref{lem:degdiscoid} shows that $\deg v$ is the lowest positive integer such that $D_v=D(g,\lambda)$ for some monic irreducible polynomial $g\in K[x]$ of degree $\deg g=\deg v$ and some $\lambda\in\hat\Q$.
\end{rem}

\begin{prop}\label{prop:ValuationsInequality}
Let $v,w\in\pM$, with $v_0<w\leq v$. Let
\[[v_0,v_1(\phi_1)=\lambda_1,\dots,v_n(\phi_n)=\lambda_n]\]
be a minimal MacLane chain for $v$. Then there exists $m\leq n$ such that $w=[v_{m-1}, w(\phi_m)=\lambda]$, for some $v_{m-1}(\phi_m)<\lambda\leq\lambda_m$.
\proof
Let $[v_0,w_1(\psi_1)=\mu_1,\dots, w_m(\psi_m)=\mu_m]$ be a minimal MacLane chain for $w$. Then $n\geq m$ by \cite[Proposition 4.35]{Rut} and $v_{m-1}=w_{m-1}$ by \cite[Corollary 4.37]{Rut}. Then $w=[v_{m-1}, w(\psi_m)=\mu_m]$. Since $v_m\leq v\geq w$, either $v_m< w$ or $w\leq v_m$ from
Remark \ref{rem:Galois1disjoint2discoids}(2). 
Suppose by contradiction that $v_m<w$. Then $m<n$. Furthermore, $v_m=[v_{m-1}, v_m(\psi_m)=\lambda_m]$ and $\lambda_m< \mu_m$ by \cite[Lemma 7.6]{FGMN}. Let $r$ be a root of $\phi_n$. Then $r\in D_v$. Since $m<n$, one has $\deg\psi_m=\deg v_m<\deg v$. Therefore $v_K(\psi_m(r))=v(\psi_m)=\lambda_m$ by \cite[Corollary 2.8]{OS2}. But $D_v\subseteq D_w=D(\psi_m,\mu_m)$, so $\lambda_m\geq\mu_m$, a contradiction. Hence $w\leq v_m$. Thus \cite[Lemma 7.6]{FGMN} implies $w=[v_{m-1}, w(\phi_m)=\mu_m]$ and $\mu_m\leq\lambda_m$, as required.
\endproof
\end{prop}

\begin{lem}\label{lem:centreparent}
Let $v,w\in\pM$. Suppose $w<v$. Then $\lambda_w<\lambda_v$ and $\deg w\leq\deg v$. Moreover, if $\deg w=\deg v$, any centre $\phi$ of $v$ is also a centre of $w$.
\proof
The statement is trivial when $w=v_0$. Suppose $w>v_0$. Let $\phi$ be a centre of $v$. Consider a minimal augmentation chain $[v_0,\dots,v_n(\phi_n)=\lambda_n]$ for $v$, with $\phi_n=\phi$. By Proposition \ref{prop:ValuationsInequality} there exist $m\leq n$ and $\mu_m<\lambda_m$ such that $w=[v_{m-1}, w(\phi_m)=\mu_m]$. Then $\deg w\leq\deg v$ and $\lambda_w<\lambda_v$ by \cite[Lemma 4.21]{Rut}. Furthermore, if $\deg w=\deg v$ then $n=m$, since the key polynomials $\phi_i$ have strictly increasing degrees. This concludes the proof as $\phi_m=\phi_n=\phi$ could be any centre of $v$.
\endproof
\end{lem}

\begin{lem}\label{lem:v(lowerdegree)}
Let $v\in\M$. For any monic non-constant $g\in K[x]$ of degree $\deg g\leq\deg v$ we have $v(g)\leq\lambda_v$, with $v(g)=\lambda_v$ only if $\deg g=\deg v$.
\proof
We prove the lemma by induction on $\deg v$. Let 
\[[v_0,\dots,v_{m-1}(\phi_{m-1})=\lambda_{m-1},v_m(\phi_m)=\lambda_m]\] 
be a minimal MacLane chain for $v$. Recall $\lambda_v=\lambda_m$. If $\deg v=1$, then $\deg g=\deg v$. By definition $v(g)=\min\{v(\phi_m),v(g-\phi_m)\}\leq \lambda_v$.
Suppose $\deg v>1$. If $\deg g=\deg v$ then $v(g)\leq \lambda_v$ as above. If $\deg g<\deg v$ then $v(g)=v_{m-1}(g)\leq \lambda_{m-1}<\lambda_m$. 
%
\endproof
\end{lem}

Recall the following result from \cite{FGMN}.
\begin{thm}[{\cite[Theorem 3.10]{FGMN}}]\label{thm:minimality}
Let $v\in\M$. For any monic non-constant $g\in K[x]$ one has
\[\frac{v(g)}{\deg g}\leq\frac{\lambda_v}{\deg v},\]
and the equality holds if and only if $g$ is $v$-minimal.
\end{thm}

\begin{lem}\label{lem:vminimalityproduct}
Let $g_1,g_2\in K[x]$ monic and non-constant. Then $g_1\cdot g_2$ is $v$-minimal if and only if both $g_1$ and $g_2$ are $v$-minimal.
\proof
Suppose $g_1$ is not $v$-minimal. Then there exists $a\in K[x]$, $\deg a<\deg g_1$ such that $g_1\vdiv_v a$. Hence $g_1g_2\vdiv_vag_2$ and $\deg(ag_2)<\deg(g_1g_2)$. So $g_1\cdot g_2$ is not $v$-minimal. Similarly for $g_2$. Suppose both $g_1$ and $g_2$ are $v$-minimal. Theorem \ref{thm:minimality} implies that
\[v(g_1\cdot g_2)\deg v=(v(g_1)+v(g_2))\deg v=\lambda_v(\deg g_1+\deg g_2)=\lambda_v\deg(g_1\cdot g_2),\]
and so $g_1\cdot g_2$ is $v$-minimal. 
\endproof
\end{lem}

\begin{lem}\label{lem:vminimal}
Let $v,w\in\M$ satisfying $w\geq v$. Let $g\in O_K[x]$ monic and non-constant. Suppose $g$ is $w$-minimal. Then $g$ is $v$-minimal.
\proof
By \cite[Remark 4.36]{Rut} we can write
\[w=[v_0,v_1(\phi_1)=\lambda_1,\dots,v_m(\phi_m)=\lambda_m,\dots,v_n(\phi_n)=\lambda_n],\]
with $v=v_m$.
Let $i=m,\dots,n-1$. By recursion it suffices to show that $g$ is $v_i$-minimal if it is $v_{i+1}$-minimal. We can suppose $g$ irreducible by Lemma \ref{lem:vminimalityproduct}. Since $\phi_{i+1}$ is $v_i$-minimal, by Theorem \ref{thm:minimality} we have
\[\frac{v_i(g)}{\deg g}\leq \frac{\lambda_i}{\deg \phi_i}=\frac{v_{i}(\phi_{i+1})}{\deg \phi_{i+1}}<\frac{\lambda_{i+1}}{\deg \phi_{i+1}}=\frac{v_{i+1}(g)}{\deg g}.\]
Therefore $v_{i+1}(g)>v_i(g)$ that is equivalent to $\phi_{i+1}\vdiv_{v_i} g$ by \cite[Lemma 4.13]{Rut}. \cite[Theorem 6.2]{FGMN} implies that $v_i(g)=\deg g\cdot \frac{v_i(\phi_{i+1})}{\deg \phi_{i+1}}$. But then Theorem \ref{thm:minimality} shows that $g$ is $v_i$-minimal.
\endproof
\end{lem}

\begin{lem}\label{lem:minimalityequivalence}
Let $v\in\M$ and let $\phi$ be a centre of $v$. Let $g\in K[x]$ monic, non-constant and $v$-minimal. Then \begin{enumerate}[label=(\roman*)]
    \item $\deg v\mid\deg g$.\label{item:minimal3.10}
    \item $g\veq_w\phi^{\deg g/\deg v}$ for any $w\in\M$, $w<v$.\label{item:minimalequivalent}
\end{enumerate}
\proof
\ref{item:minimal3.10} follows from \cite[Lemma 2.10]{FGMN}. For proving \ref{item:minimalequivalent} we can suppose without loss of generality that $\phi\in\KP w$ by Proposition \ref{prop:ValuationsInequality} and Lemma \ref{lem:vminimal}. Equivalently, $v=[w,v(\phi)=\lambda]$ for some $\lambda\in\Q$, $\lambda>w(\phi)$. Let $d=\deg g/\deg \phi$ and expand
\[g=\sum_{j= 0}^d a_j \phi^j,\qquad\text{where }a_j\in K[x],\,\deg a_j<\deg \phi,\]
and $v(a_d)=w(a_d)=0$. Note that $v(g)=v(\phi^d)$ by Theorem \ref{thm:minimality}. Therefore 
\begin{align*}
    w(\phi^d)&=v(g)-d(\lambda-w(\phi))\leq v(a_j\phi^j)-d(\lambda-w(\phi))\\
    &<v(a_j\phi^j)-j(\lambda-w(\phi))=w(a_j \phi^j),
\end{align*}
for all $j<d$. Thus $g\veq_{w}\phi^d$ as required.
\endproof
\end{lem}

The following two results come from \cite{OS2}. 

\begin{prop}[{\cite[Proposition 2.5]{OS2}}]\label{prop:uniqueMacLanevalStrongPoly}
Let $\phi\in O_K[x]$ be a monic irreducible polynomial. There exists a unique MacLane valuation $v_{\phi}$ over which $\phi$ is a strong key polynomial.
\end{prop}

\begin{prop}[{\cite[Proposition 2.7]{OS2}}]\label{prop:2.7OS2}
Let $v\in\M$ and $\phi$ a proper key polynomial over $v$. Let $w=[v,w(\phi)=\lambda]$, for some $\lambda>v(\phi)$ and let
$r\in D_w$. For any $g\in K[x]$ such that $v(g)=w(g)$, we have $v_K(g(r))=v(g)$.
\end{prop}

\begin{lem}\label{lem:epsilonv}
Let $v\in\pM$ given by a MacLane chain
\[[v_0,v_1(\phi_1)=\lambda_1,\dots, v_{m-1}(\phi_{m-1})=\lambda_{m-1}, v_m(\phi_m)=\lambda_m].\]
Suppose $m>0$. The ramification index $e_{v_{m-1}}$ equals $[\Gamma_{\phi_m}(v):\Z]$, where
\[\Gamma_{\phi_m}(v)=\{v(a) \mid a\in K[x],\, a\neq 0,\, \deg a < \deg \phi_m\}.\]
In particular, it is independent of the chosen MacLane chain.
\proof
First note that if we restrict to minimal MacLane chains, the result is trivial. By Remark \ref{rem:MacChainToMinimal}(2) it suffices to prove that if $m>1$ and $\deg\phi_{m-1}=\deg\phi_{m}$, then $e_{v_{m-2}}=e_{v_{m-1}}$. We have \[v_{m-1}(\phi_m-\phi_{m-1})=\lambda_{m-1}.\] 
since $\phi_{m-1}\not\veq_{v_{m-1}}\phi_m$. But 
$\deg(\phi_m-\phi_{m-1})<\deg \phi_{m-1}$,
so \[\lambda_{m-1}=v_{m-1}(\phi_m-\phi_{m-1})=v_{m-2}(\phi_m-\phi_{m-1})\in \Gamma_{v_{m-2}}.\]
Thus $\Gamma_{v_{m-2}}=\Gamma_{v_{m-1}}$, as required.
\endproof
\end{lem}

\begin{defn}
Let $v\in\pM$ given by a MacLane chain 
\[[v_0,\dots, v_{m-1}(\phi_{m-1})=\lambda_{m-1}, v_m(\phi_m)=\lambda_m].\] 
Define $\epsilon_v=e_{v_{m-1}}$ if $m>0$, and $\epsilon_v=1$ otherwise.
\end{defn}

For any monic irreducible polynomial $\phi\in K[x]$, define $K_\phi=K[x]/(\phi)$, finite extension of $K$. Let $O_\phi$ be the ring of integers of $K_\phi$ and $k_\phi$ the residue field. Recall $\deg \phi=e_\phi f_\phi$, where $e_\phi$ and $f_\phi$ are respectively the ramification index and the residual degree of the extension $K_\phi/K$.

Let $v\in\pM$ with centre $\phi$. Then \cite[Proposition 1.9(2)]{FGMN} shows that $e_\phi=[\Gamma_\phi(v):\Z]$, and so $e_\phi=\epsilon_v$ by Lemma \ref{lem:epsilonv}.
It follows that $f_\phi=\deg v/e_\phi$ is independent of the choice of the centre $\phi$.

\begin{nt}
Given $v\in\pM$ with centre $\phi$, denote $f_v=f_\phi$.
\end{nt}



Let $f\in K[x]$, $v\in\M$ and $\phi\in\KP v$. Write
\[f=\sum_{\i=0}^d a_\i\phi^\i,\qquad\text{where }\deg a_\i<\deg \phi.\]
The \textit{Newton polygon}, $N_{v,\phi}(f)$ of $f$ is
\[N_{v,\phi}(f) = \mbox{lower convex hull}\,(\{(\i,v(a_\i)) \mid a_\i\neq0\})\subset \R^2.\]

\begin{nt}\label{nt:NewtonPolygon}
Let $\lambda\in\Q$, $\lambda>v(\phi)$ and $w=[v,w(\phi)=\lambda]$. We denote by $L_w(f)$ the intersection of $N_{v,\phi}(f)$ with the line of slope $-\lambda$ which first touches it from below:
\[L_w(f):=\{(\i,u)\in N_{v,\phi}(f)\mid u+\lambda \i\text{ is minimal}\}.\]
Therefore if $N_{v,\phi}(f)$ has an edge $L$ of slope $-\lambda$ then $L_w(f)=L$, otherwise $L_w(f)$ is one of the vertices of $N_{v,\phi}(f)$. 
\end{nt}

\begin{nt}\label{nt:i}
Let $\lambda\in\hat\Q$, $\lambda>v(\phi)$ and $w=[v,w(\phi)=\lambda]$. If $\lambda<\infty$ denote by $(\io_w,u^0_w),(\i_w,u_w)$ the two endpoints of $L_w(f)$ (equal if $L_w(f)$ is a vertex), where $\io_w\leq \i_w$. If $\lambda=\infty$, set $\io_w=0$, $u^0_w=\infty$, and denote by $(\i_w,u_w)$ the left-most vertex of $N_{v,\phi}(f)$.
\end{nt}

\section{MacLane chains invariants and residual polynomials}\label{sec:Residualpoly}
Part of the current section can be found in \cite[\S3]{FGMN}.

Let $v\in \M$ given by a MacLane chain
\begin{equation}\label{eqn:MacLaneChain}
    [v_0,v_1(\phi_1)=\lambda_1,\dots,v_n(\phi_n)=\lambda_n].
\end{equation}
We attach to (\ref{eqn:MacLaneChain}) the following data.

\begin{defn}
Set $v_{-1}=v_0$, $\pi_{-1}=\pi$, $\phi_0=x$, $\lambda_0=0$ and for all $0\leq i\leq n$
define
\[e_i=e_{v_i}/e_{v_{i-1}},\qquad h_i=e_{v_i}\lambda_i, \qquad f_{i-1}=f_{v_i}/f_{v_{i-1}}
.\]
Fix $\ell_i,\ell_i'$, such that $\ell_ih_i+\ell_i'e_i=1$, with $0\leq\ell_i<e_i$. Then inductively define
 \[\gamma_i=\phi_i^{e_i}\pi_{i-1}^{-h_i},
 \qquad\pi_{i}=\phi_i^{\ell_i}\pi_{i-1}^{\ell_i'}.\]
\end{defn}

\begin{rem}\label{rem:degreeramificationindexChain}
Let $0\leq i<n$. Then $\deg v_{i+1}=e_if_i\deg v_i$ and $v_i(\phi_n)\in\Gamma_{v_{i-1}}$.
\end{rem}

\begin{lem}\label{lem:gammapival}
For any $0\leq i\leq n$ and any $j> i$, we have
\begin{itemize}
    \item $v_{j}(\gamma_i)=v_{i}(\gamma_i)=0$;
    \item $v_{j}(\pi_{i})=v_{i}(\pi_{i})=\frac{1}{e_{v_i}}$. So $\pi_i$ is a uniformiser for $v_i$.
\end{itemize}
\proof
The lemma follows by induction and the equality $v_j(\phi_i)=v_i(\phi_i)$.
\endproof
\end{lem}

\begin{nt}\label{nt:bvhvlv}
We will denote by $b_v,h_v,\ell_v,\ell_v'$ the quantities $e_n,h_n,\ell_n,\ell_n'$ respectively. They are independent of the chosen MacLain chain for $v$.
\end{nt}

\begin{lem}
For any $0\leq i\leq n-1$ there exists a polynomial $\uni i\in K[x]$ such that $\uni i\veq_{v_{i}}\pi_i$.
Furthermore, there exists a
polynomial $\uni i'\in K[x]$ such that $v_{i}(\uni {i}')=-v_{i}(\pi_{i})$ and $(\uni i')^{-1}\veq_{v_{i+1}}\pi_i$.
\proof
First note that if $\uni i$ exists, then $\uni i\veq_{v_{i+1}}\pi_i$ as $v_{i}(\pi_i)=v_{i+1}(\pi_i)$ by Lemma \ref{lem:gammapival}. Now we prove the lemma by induction on $i$. When $i=0$, we can choose $\uni i=\pi=\pi_i$ and $\uni i'=\pi^{-1}$. Suppose $i>0$. Define $\uni i\in K[x]$ by
\[\uni i=\begin{cases}\phi_i^{\ell_i}\uni{i-1}^{\ell_i'}&\text{if }\ell_i'\geq 0,\\
\phi_i^{\ell_i}(\uni{i-1}')^{-\ell_i'}&\text{if }\ell_i'< 0.\end{cases}\]
By inductive hypothesis, $\uni{i-1}\veq_{v_{i}}\pi_{i-1}$ and $(\uni {i-1}')^{-1}\veq_{v_{i}}\pi_{i-1}$. Therefore $\uni i\veq_{v_{i}}\pi_i$. Finally, \cite[Lemma 4.24]{Rut} shows the existence of $\uni i'$.
\endproof
\end{lem}

\begin{lem}\label{lem:inversegammapi}
For any $0\leq i\leq n$, we have $\phi_i=\gamma_i^{\ell_i'}\pi_i^{h_i}$ and $\pi_{i-1}=\gamma_i^{-\ell_i}\pi_{i}^{e_i}$.
\proof
The lemma follows from direct computation.
\endproof
\end{lem}

\begin{lem}\label{lem:gammadecomposition}
For any $0\leq i\leq n$, we have
    \[\pi_{i}=\phi_{i}^{m'_{i}}\cdots \phi_1^{m'_1}\cdot \pi^{m'_0}\quad \text{and}\quad\pi_{i}^{e_{v_{i}}}=\gamma_{i}^{m_{i}}\cdots \gamma_1^{m_1}\cdot \pi^{m_0},\]
where
\[m_j'=\begin{cases}
\ell_1'\cdot\ell_{i}'&\text{if }j=0,\\
\ell_j\ell_{j+1}'\cdots\ell_{i}'&\text{if }j>0.
\end{cases}\quad \text{and}\quad m_j=\begin{cases}
1&\text{if }j=0,\\
e_1\cdots e_{j-1}\ell_j&\text{if }j>0.
\end{cases}\]
Note that
\[\gamma_i=\phi_i^{e_i}\cdot\phi_{i-1}^{-h_{i}m'_{i-1}}\cdots \phi_1^{-h_{i}m'_1}\pi^{-h_{i}m'_0},
\quad 
\phi_i^{e_{v_{i}}}=\gamma_i^{e_{v_{i-1}}}\cdot\gamma_{i-1}^{h_im_{i-1}}\cdots\gamma_1^{h_im_1}\pi^{h_i}
.\]
\proof
The proof follows by mathematical induction and Lemma \ref{lem:inversegammapi}.
\endproof
\end{lem}

Let $i=0,\dots,n$. Recall the definition of the residue map $H_{v_i}$ of $v_i$. From \cite[Lemma 2.9]{FGMN} a polynomial $f\in K[x]$ belongs to $U_{v_i}$ if $v_{i-1}(f)=v_i(f)$. Therefore $\phi_j,\pi_j,\gamma_j$ are units of $\loc{v_i}$, for all $j=0,\dots,i-1$. It follows that $\phi_i,\pi_i,\gamma_i\in \loc{v_i}$, domain of $H_{v_i}$. Denote 
\[x_i=H_{v_i}(\phi_i),\quad p_i=H_{v_i}(\pi_{i-1}),\quad y_i=H_{v_i}(\gamma_i).\]
Note that by \cite[Lemma 2.9]{FGMN} the set of units $\Av_{v_i}^\times$ of $\Av_{v_i}$ coincides with the image of the canonical homomorphism $\Av_{v_{i-1}}\rightarrow \Av_{v_i}$.

We recall the following from \cite[\S4.1.3]{Rut} and \cite[\S3.4]{FGMN}.
There exist a sequence of simple field extensions 
\[k=k_0\subseteq k_1\subseteq k_2\subseteq\dots\subseteq k_n,\]
with $k_i\simeq k_{\phi_i}$, such that for all $i=0,\dots, n$ there are isomorphisms of $k$-algebras $\bar H_i: \Av_{v_i}\rightarrow k_i[X_i]$. One can see that $\bar H_i$ is the unique homomorphism satisfying:
\begin{enumerate}[label=(\roman*)]
    \item $\bar H_i(y_i)=X_i$; \label{item:Ri0onyi}
    \item $\bar H_i(u)=\bar H_{i-1}(u)$ when $i>0$ and $u\in \Av_{v_{i-1}}$, where we canonically see $u\in \Av_{v_i}$ via $\Av_{v_{i-1}}\rightarrow \Av_{v_i}$ and $\bar H_{i-1}(u)\in k_{i}$ via the natural map $k_{i-1}[X_{i-1}]\rightarrow k_i$ taking $X_{i-1}$ to the generator of $k_i$ over $k_{i-1}$. \label{item:Ri0onunits}
\end{enumerate}
By \cite[Proposition 3.9]{FGMN}, the canonical embedding $\Av_{v_i}\hookrightarrow \F_{v_i}$ induces an isomorphism between the field of fractions of $\Av_{v_i}$ and $\F_{v_i}$.
Therefore we can consider the largest subring $\Loc{v_i}\subset K(x)$ such that the isomorphism $\bar H_i$ lifts to a surjective homomorphism 
\[H_i:\Loc{v_i}\rightarrow k_i[X_i^{\pm 1}],\] satisfying $H_i(f)=H_i(g)$ if $f\veq_{v_i} g$.
In particular, $\loc{v_i}(0)\subseteq\Loc{v_i}$ and  $H_i=\bar H_i\circ H_{v_i}$ on $\loc{v_i}(0)$. Furthermore, note that $\gamma_i^{-1}\in \Loc{v_i}$ from \ref{item:Ri0onyi}.

\begin{defn}\label{defn:Loc(i)H(ii)}
Let $\alpha\in\Gamma_{v_i}$. Define
\begin{enumerate}[label=(\roman*)]
    \item $\Loc{v_i}(\alpha)=F_{v_i}\cdot\loc{v_i}(\alpha)\subset K(x)$.
    \item $H_{i,\alpha}:\Loc{v_i}(\alpha)\rightarrow k_i[X_i^{\pm1}]$ given by $H_{i,\alpha}(f)=H_i\big(f/\pi_i^{e_{v_i}\alpha}\big)$.\label{item:Hialpha}
\end{enumerate}
The map $H_{i,\alpha}$ in \ref{item:Hialpha} is well-defined since $\pi_i^{-1}\in \Loc{v_i}(-\alpha)$.
\end{defn}

\begin{defn}
For $0\leq i\leq n$ and $\alpha\in\Gamma_{v_i}$, let $\i_i(\alpha), u_i(\alpha)\in \Z$ such that $u_i(\alpha)e_i+\i_i(\alpha)h_i=e_{v_i}\alpha$, with $0\leq \i_i(\alpha)<e_i$.
Define 
\begin{enumerate}[label=(\roman*)]
    \item $\varphi_i(\alpha)=x_i^{\i_i(\alpha)}p_{i}^{u_i(\alpha)}\in \Av_{v_i}(\alpha);$
    \item $c_i(\alpha)=\ell_i'\i_i(\alpha)-\ell_i u_i(\alpha)\in\Z.$
\end{enumerate}
\end{defn}

Let $\alpha\in\Gamma_{v_i}$. Let $R_{i,\alpha}:O_{v_i}(\alpha)\rightarrow k_i[X_i]$ be the map defined in \cite[Definition 3.13]{FGMN}, where we replaced the variable $y$ with $X_i$.
By \cite[Theorem 4.1]{FGMN}, we have $\Av_{v_i}(\alpha)=\varphi_i(\alpha)\Av_{v_i}$ and 
$R_{i,\alpha}$ is the lift of the map 
\[\bar R_{i,\alpha}:\Av_{v_i}(\alpha)\rightarrow k_i[X_i]\] 
given by $\bar R_{i,\alpha}(\varphi_i(\alpha)\cdot a)=\bar H_i(a)$. 
Since $e_{v_i}\alpha=u_i(\alpha)e_i+\i_i(\alpha)h_i$, by Lemma \ref{lem:inversegammapi}, we have
\[\pi_i^{e_{v_i}\alpha}\gamma_i^{c_i(\alpha)}=\phi_i^{\i_i(\alpha)}\pi_{i-1}^{u_i(\alpha)}.\]
Therefore for any $f\in \Av_{v_i}(\alpha)$ we have
\begin{equation}\label{HandR}
H_{i,\alpha}(f)=X_i^{c_i(\alpha)}\cdot R_{i,\alpha}(f).
\end{equation}
We extend $R_{i,\alpha}$ through (\ref{HandR}).


\begin{defn}\label{defn:redialphamap}
Let $\alpha\in\Gamma_{v_i}$. The \textit{residual polynomial operator} $R_{i,\alpha}$ is the map $\Loc{v_i}(\alpha)\rightarrow k_i[X_i^{\pm 1}]$ given by $R_{i,\alpha}(f)=X_i^{-c_i(\alpha)}\cdot H_{i,\alpha}(f)$.
\end{defn}

\begin{rem}\label{rem:ki}
Let $0\leq i<n$ and $\alpha_i=v_i(\phi_{i+1})=f_ie_i\lambda_i$. The field $k_{i+1}$ is isomorphic to $k_i[X_i]/(R_{i,\alpha_i}(\phi_{i+1}))$ by \cite[Corollary 5.5(2)]{FGMN}. Furthermore, $k_{i+1}\simeq k_i[X_i^{\pm 1}]/(H_{i,\alpha_i}(\phi_{i+1}))$ by definition.
\end{rem}

\begin{nt}
We denote by $k_v$ the field $k_n$. In fact, it does not depend on the radius of $v$. 
\end{nt}

\begin{defn}
Let $\alpha\in\Gamma_v$. For any $f\in \Loc{v}(\alpha)$, define $f|_{v,\alpha}\in k_v[X]$ by $f|_{v,\alpha}(X)=R_{n,\alpha}(f)(X)$. 
\end{defn}


Let $f\in K[x]$. Let $\alpha=v(f)$. Denote by $N_n(f)$ the Newton polygon $N_{v_{n-1},\phi_n}$. If $n>0$, consider the edge $L_v(f)$ of $N_n(f)$. Let $(\io_v,u^0_v), (\i_v,u_v)$ be the two endpoints of $L_v(f)$, with $\io_v\leq \i_v$. Note that $\io_v-\i_n(\alpha)=e_n\cdot\lfloor \io_v/e_n\rfloor$.


\begin{defn}[{\cite[Definition 3.15]{FGMN}}]
The \textit{reduction of $f$ with respect to $v$} is 
\[f|_v=\begin{cases}f|_{v,\alpha}&\mbox{if }n=0,\\f|_{v,\alpha}/X^{\lfloor \io_v/e_n\rfloor}&\mbox{if }n>0.\end{cases}\]
\end{defn}

\begin{rem}
Note that $f|_{v,\alpha}$ and $f|_v$ do depend on the chosen MacLane chain for $v$. 
\end{rem}

Note that 
\begin{equation}\label{eqn:Hvto|v}
H_{n,\alpha}(f)(X)=X^{\lfloor \io_v/e_n\rfloor+c(\alpha)}f|_v=X^{\io_v/e_n-\ell_n e_{v_{n-1}}\alpha}f|_v.
\end{equation}

\begin{lem}\label{lem:reduction}
Expand $f=\sum_{\i} a_\i \phi_n^\i$, $\deg a_\i<\deg \phi_n$. If $n>0$, then
\[f|_v=\sum_{j\geq 0}H_{n-1,\alpha_j}(a_{\i_j})X^j,\]
where $\i_j=\io_v+je_n$ and $\alpha_j=\alpha-\i_j\lambda_n$.
\proof
There exists $f'\in K[x]$ such that $f\veq_v f'$ and $f'=\sum_\i a'_\i \phi_n^\i$, where either $a'_\i=0$ or $a'_\i=a_\i$ and $v(a'_\i)=\alpha-\i\lambda_n$. If $a'_\i\neq 0$, then $(\i, v(a_\i))\in L_v(f)$. Since
\[
L_v(f)\cap\Big(\Z\times\tfrac{1}{e_{v_{n-1}}}\Z\Big)=(\io_v,\alpha-\io_v\lambda_n)+(e_n,-\lambda_n)\Z,
\]
we have $f'=\sum_{j\geq 0} a_{\i_j}'\phi_n^{\i_j}$. It follows that
\[f'=\phi_n^{\i_n(\alpha)}\pi_{n-1}^{u_n(\alpha)}\gamma_n^{\lfloor \io_v/e_n\rfloor}\sum_{j\geq 0} \frac{a_{\i_j}'}{\pi_{n-1}^{e_{v_{n-1}}\alpha_j}}\gamma_n^j.\]
Therefore
\begin{equation}\label{eqn:f|valpha}
    f|_{v,\alpha}=f'|_{v,\alpha}=X^{\lfloor \io_v/e_n\rfloor}\sum_{j\geq 0} H_{n-1,\alpha_j}\big(a_{\i_j}'\big)X^j.
\end{equation}
Finally, note that $a_{\i_j}'=0$ if and only if $v(a_{\i_j})>e_{v_{n-1}}\alpha_j$. Thus in (\ref{eqn:f|valpha}) we can replace $H_{n-1,\alpha_j}\big(a_{\i_j}'\big)$ with $H_{n-1,\alpha_j}(a_{\i_j})$.
\endproof
\end{lem}

\begin{exa}
Let $f=(x^3-2p)^2-px^2(x^3-2p)\in\Q_p[x]$ ($p\neq 2$) and \[v=v_2=[v_0,v_1(x)=1/3,v_2(x^3-2p)=5/3]\]
The Newton polygon $N_2(f)$ is
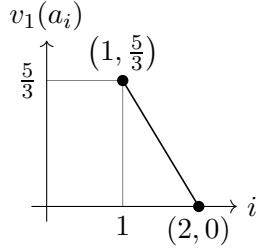
\[\begin{tikzpicture}[scale=1]
     \draw[very thin,color=gray] (0,1.67) -- (1,1.67);
     \draw[very thin,color=gray] (1,0) -- (1,1.67);
    \draw[->] (-0.2,0) -- (2.5,0) node[right] {$i$};
    \draw[->] (0,-0.2) -- (0,2.2) node[above] {$v_1(a_i)$};
    \tkzDefPoint(1,1.67){A}
    \tkzDefPoint(2,0){B}
    \tkzDefPoint(0,1.67){C}
    \tkzDefPoint(1,0){D}
    \tkzLabelPoint[left,above](A){$\left(1,\tfrac{5}{3}\right)$}
    \tkzLabelPoint[right,below](B){$(2,0)$}
    \tkzLabelPoint[left](C){$\frac{5}{3}$}
    \tkzLabelPoint[below](D){$1$}
    \foreach \n in {A,B}
    \node at (\n)[circle,fill,inner sep=1.5pt]{};
    \tkzDrawSegment[black!60!black](A,B)
\end{tikzpicture}\]
Then $\pi_0=p$, $\pi_1=x$, $\pi_2=x$,$\gamma_1=x^3p^{-1}$ and $k_1=k_0=\F_p$. Since $x^3-2p=p^{-1}(\gamma_1-2)$, then $R_{1,1}(x^3-2p)=X_1-2$. It follows that $k_2=\F_p[X_1]/(X_1-2)\simeq \F_p$.
Via Lemma \ref{lem:reduction} compute
\[f|_v=X+H_{1,5/3}(-px^2)=X+H_1\left(\frac{-px^2}{x^5}\right)=X+\bar H_1(-y_1^{-1})=X-2^{-1}.\]

\end{exa}

\begin{prop}[{\cite[Corollary 4.9, Corollary 4.11]{FGMN}}]\label{prop:ReductionDegreeProduct}
Suppose $n>0$. Following the notation above, we have:
\begin{enumerate}[label=(\roman*)]
    \item the $j$-th coefficient of $f|_{v,\alpha}$ is non-zero if and only if $v_{n-1}(a_{\i_j})=\alpha_j$;
    \item $\deg f|_{v,\alpha}=\lfloor \i_v/b_v\rfloor$ and $\mathrm{ord}_{X}(f|_{v,\alpha})=\lfloor \io_v/b_v\rfloor$;
    \item $\deg f|_{v}=(\i_v-\io_v)/b_v$ and $f|_{v}(0)\neq 0$;\label{item:ReductionDegree}
    \item $fh|_{v}=f|_{v}h|_{v}$ for all $h\in K[x]$.\label{item:ReductionProduct}
\end{enumerate}
\end{prop}

\begin{prop}[{\cite[Corollary 4.10]{FGMN}}]\label{prop:v-equivalentResidual}
For non-zero $f,h\in K[x]$, the following conditions are equivalent:
\begin{enumerate}[label=(\roman*)]
    \item $f\veq_v h$,
    \item $v(f)=v(h)$ and $f|_{v}=h|_{v}$,
    \item $L_v(f)=L_v(h)$ and $f|_{v}=h|_{v}$.
\end{enumerate}
\end{prop}

\begin{lem}[{\cite[Lemma 5.1]{FGMN}}]\label{lem:v-irreducreduction}
A polynomial $f\in K[x]$ is $v$-irreducible if and only if either
\begin{itemize}
    \item $\io_v=\i_v=1$ or 
    \item $\io_v=0$ and $f|_{v}$ is irreducible in $k_n[X]$.
\end{itemize}
\end{lem}

\begin{lem}[{\cite[Lemma 5.2]{FGMN}}]\label{lem:keypolyreduction}
Suppose $n>0$. A monic $f\in K[x]$ is a key polynomial over $v$ if and only if one of the two following conditions is satisfied:
\begin{enumerate}[label=(\arabic*)]
    \item $\deg f=\deg v$ and $f\veq_v \phi_n$;
    \item $\io_v=0$, $\deg f=\i_v\deg v$ and $f|_{v}$ is irreducible. \label{item:properkeypoly}
\end{enumerate}
In case \ref{item:properkeypoly}, $\deg f=b_v\deg v\cdot \deg f|_v$, $N_n(f)=L_v(f)$ and $f|_v$ is monic.
\end{lem}

%

\begin{lem}\label{lem:phiunit}
Let $\mu_1,\mu_2\in\M$ such that $\mu_1\geq v\not>\mu_2$. Suppose $\phi_n$ is a centre of $\mu_1$ and let $\phi\in K[x]$ be a centre of $\mu_2$. If $v\neq \mu_1\wedge\mu_2$ then $\phi|_v$ is a unit. 
\proof
Let $w=\mu_1\wedge \mu_2$. Then $\mu_1(\phi)=w(\phi)$ by Proposition \ref{prop:appendix}. Since $w\leq \mu_1$ and $v\leq \mu_1$, either $w<v$ or $w\geq v$ by Theorem \ref{thm:valdiscoid} and Remark \ref{rem:Galois1disjoint2discoids}(2).

Suppose $w<v$. Proposition \ref{prop:ValuationsInequality} implies that there exists $w'\geq w$ such that
$v=[w',v(\phi_n)=\lambda_n]$.
In particular, $w'(\phi)=v(\phi)$.
From \cite[Lemma 2.9]{FGMN} it follows that $\phi$ is $v$-equivalent to some polynomial of degree $<\deg v$. Hence $\phi|_v$ is a unit by Proposition \ref{prop:v-equivalentResidual}. 

Suppose $w>v$. Then the polynomial $\phi_n$ is a centre of $w$ and so
\[\phi\veq_v\phi_n^{\deg \phi/\deg\phi_n}\]
by Lemmas \ref{lem:vminimal} and \ref{lem:minimalityequivalence}. Then $\phi|_v$ is a unit by Proposition \ref{prop:v-equivalentResidual}.
\endproof
\end{lem}

\section{MacLane clusters}\label{sec:MacLaneClusters}
Let $f\in K[x]$ be a separable polynomial and let $c_f\in K$ be its leading term. Assume $f/c_f\in O_K[x]$ and write $\roots$ for the sets of roots of $f$ in $\bar K$. 
If $C/K$ is a hyperelliptic curve, it is always given by an equation $y^2=f(x)$, where $f\in K[x]$ is as above.

\begin{defn}\label{defn:Clusters}
A \textit{MacLane cluster} (for $f$) is a pair $(\s,v)$ where $\s\subseteq\roots$, and $v$ is a MacLane pseudo-valuation such that
\begin{enumerate}
    \item $\s=D_v\cap\roots\neq \varnothing$;
    \item if $\s=D_w\cap\roots$ for a MacLane valuation $w>v$ then $\deg w>\deg v$.\label{item:Clustersdefn(2)}
\end{enumerate}
If $v$ is a MacLane valuation then $(\s,v)$ is said \textit{proper MacLane cluster}. The degree of $(\s,v)$ is $\deg v$. The \textit{degree}, a \textit{centre} and the \textit{radius} of a MacLane cluster $(\s,v)$ are the degree, a centre and the radius of $v$, respectively. 
\end{defn}

\begin{rem}\label{rem:clustersfirst}
Let $(\s,v)$ be a MacLane cluster. Note that by definition
\begin{enumerate}[label=(\roman*)]
    \item $\s$ is $G_K$-invariant, \label{item:sGaloisinvariant}
    \item $v$ determines $\s$.\label{item:vdeterminess}
\end{enumerate}
\end{rem}

\begin{defn}\label{defn:ClusterPicture}
The \textit{MacLane cluster picture} of $f$ is the combinatorial data consisting of the collection of all MacLane clusters for $f$ together with their radii. We will denote by $\Sigma_f^M$ the set of all MacLane clusters for $f$.
\end{defn}

\begin{defn}
We say that a MacLane pseudo-valuation $v\in\pM$ \textit{defines a MacLane cluster} $(\s,w)\in\Sigma_f^M$, if $w=v$ (and $\s=D_v\cap\roots$).
\end{defn}

\begin{defn}\label{defn:child,minimal}
We write $(\t,w)\subseteq(\s,v)$ if $w\geq v$. 
If $(\t,w)\subsetneq(\s,v)$ is maximal, we write $(\t,w)<(\s,v)$ and $v=P(w)$, and refer to $(\t,w)$ as \textit{a child} of $(\s,v)$, and to $(\s,v)$ as \textit{the parent} of $(\t,w)$. A proper MacLane cluster $(\s,v)$ with no proper child of degree $\deg v$ is said \textit{degree-minimal}.
\end{defn}

\begin{lem}\label{lem:ClusterInclusion}
Let $(\s,v), (\t,w)\in\Sigma_f^M$ such that $\s\subsetneq\t$. Then $(\s,v)\subsetneq(\t,w)$.
\proof
Since $\s\subseteq D_{v}\cap D_w$ either $D_v\subsetneq D_{w}$ or $D_{w}\subseteq D_v$. But
\[D_v\cap\roots=\s\subsetneq\t=D_{w}\cap\roots,\]
so $D_v\subsetneq D_{w}$. Thus $w> v$.
\endproof
\end{lem}



\cite[Proposition 2.26]{KW} shows that the meet of any two MacLane pseudo-valuations $v$ and $w$ exists; it will be denoted by $v\wedge w$. Hence $v\wedge w$ is the maximal MacLane pseudo-valuation $\leq v$ and $\leq w$. In other words, $\pM$ with $\leq$ forms a meet-semilattice. 

\begin{lem}\label{lem:wedgecluster}
Let $(\s,v)$, $(\t,w)\in\Sigma_f^M$, and $\s\wedge\t=D_{v\wedge w}\cap \roots$. Then $(\s\wedge\t,v\wedge w)$ is the smallest MacLane cluster containing $(\s,v)$ and $(\t,w)$.
\proof
We only need to show that $(\s\wedge\t,v\wedge w)$ is a MacLane cluster. Suppose not. Then there exists a MacLane valuation $v'>v\wedge w$, with $\s\wedge \t=D_{v'}\cap\roots$ and $\deg v'\leq\deg(v\wedge w)$. Then $v'\not\leq v$ or $v'\not\leq w$, from the definition of $v\wedge w$. Without loss of generality we can assume that $v'\not\leq v$.

If $v\not< v'$, then $D_{v'}\not\subset D_v$ and $D_v\not\subseteq D_{v'}$ so $D_{v'}\cap D_v=\varnothing$ by Remark \ref{rem:Galois1disjoint2discoids}(2). But this contradicts
\[D_v\cap\roots=\s\subseteq\s\wedge\t=D_{v'}\cap \roots.\]

If $v< v'$, then
\[\s\subseteq\s\cup\t\subseteq\s\wedge\t=D_{v'}\cap \roots\subseteq D_v\cap\roots=\s.\]
But then $\s=D_{v'}\cap \roots$, $v'>v$ and $\deg v'\leq \deg (v\wedge w)\leq \deg v$ by Lemma \ref{lem:centreparent}, which contradicts the definition of MacLane cluster for $(\s,v)$.
\endproof
\end{lem}

Let $F\in K[x]$ be a monic irreducible factor of $f$. Let $v_F$ be the MacLane pseudo-valuation with $D_{v_F}=D(F,\infty)$ (Theorem \ref{thm:valdiscoid}). We also denote $v_F$ by $v_r$ where $r\in\roots$ is any root of $F$. For any non-empty $G_K$-invariant subset $\s\subseteq\roots$, define $g_\s=\prod_{r\in\s}(x-r)\in K[x]$. Then $g_\s\mid f$. Let $F_1,\dots,F_m$ be the irreducible monic factors of $g_\s$. Define $v_\s\in\pM$ by
\[v_\s=v_{F_1}\wedge\dots\wedge v_{F_m}.\]

\begin{lem}\label{lem:vs}
Let $v\in\pM$ and let $\s=D_v\cap\roots\neq\varnothing$.
Then $v\leq v_{\s}$. In particular, $(\s,v_{\s})$ is a MacLane cluster.
\proof
The set $\s$ is $G_K$-invariant, as so are $D$ and $\roots$. Let $F_1,\dots,F_m$ be the irreducible factors of $g_\s$ as above. Let $\s_i$ be the set of roots of $F_i$. Note that $D_{v_{F_i}}=\s_i$ for all $i$. Then $D_{v_{\s}}\supseteq\bigcup_{i=1}^m D_{v_{F_i}}=\s$. Suppose $w\in\pM$ with $\s= D_w\cap\roots$. Then $D_{v_{F_i}}\subseteq D_w$, so $w\leq v_{F_i}$ for all $i$. By definition of $v_{\s}$ we have $w\leq v_\s$. Since $w\leq v_{\s}$ for any $w$ with $\s=D_w\cap\roots$, it only remains to show that $\s=D_{v_\s}\cap\roots$. Since $v\leq v_{\s}$ from above, we have
\[\s\subseteq D_{v_\s}\cap\roots\subseteq D_v\cap\roots=\s,\]
that implies $D_{v_\s}\cap\roots=\s$. Thus $(\s,v_{\s})$ is a MacLane cluster.
\endproof
\end{lem}

\begin{lem}\label{lem:uniqueclustervs}
Let $\s=D_v\cap\roots\neq\varnothing$, for some $v\in\pM$. Let
\[[v_0,v_1(\phi_1)=\lambda_1,\dots, v_n(\phi_n)=\lambda_n]\]
be a minimal MacLane chain for $v_\s$. Then there exists $i=0,\dots,n$ such that $v\leq v_i$, $\deg v=\deg v_i$ and $(\s,v_i)$ is a cluster. In particular, if $(\s,v)$ is a MacLane cluster, then $v=v_i$.
\proof
Let $w\in\pM$ such that $D_w\cap\roots=\s$. Then $w\leq v_\s$ by Lemma \ref{lem:vs}. Proposition \ref{prop:ValuationsInequality} implies that $\deg w=\deg v_m$ for some $i=0,\dots,m$. 

The argument above holds in particular when $w=v$.
It only remains to show that $\s=D_{v_i}\cap\roots$. We have
\[\s=D_{v_\s}\cap\roots\subseteq D_{v_i}\cap\roots\subseteq D_{v}\cap\roots=\s,\]
that implies $\s=D_{v_i}\cap\roots$, as required.
\endproof
\end{lem}

\begin{prop}\label{prop:SigmaTree}
The set $\Sigma_f^M$ under the partial order $\supseteq$ forms a rooted tree.
\proof
Let $V_f^M=\{v\in\pM\mid (D_v\cap\roots,v)\in\Sigma_f^M\}$. By Remark \ref{rem:clustersfirst}\ref{item:vdeterminess} there is a natural bijection from $V_f^M$ to $\Sigma_f^M$ taking $v\mapsto (D_v\cap\roots,v)$ inverting partial orders by definition. Hence it suffices to show that $V_f^M$ is a rooted tree. First note that $V_f^M\neq \varnothing$ since $v_F\in V_f^M$ for any monic irreducible factor $F$ of $f$. Then $V_f^M$ is a rooted tree by \cite[Corollary 2.8]{KW} and Lemma \ref{lem:wedgecluster}.
\endproof
\end{prop}

\begin{lem}\label{lem:nonproperclusters}
Let $(\s,v)$ be a MacLane cluster. Then $|\s|\geq \deg v$. Furthermore, $|\s|>\deg v$ if and only if $(\s,v)$ is proper.
\proof
First note that $v\leq v_\s$ by Lemma \ref{lem:vs}. Then Lemma \ref{lem:centreparent} implies 
\[\deg v\leq\deg v_\s\leq\min_{r\in\s}\deg v_r=\min_{r\in\s}|G_K\cdot r|\leq|\s|.\]

If $|\s|=\deg v$, then $\s=G_K\cdot r$ for some (any) $r\in\s$, and $\deg v=\deg v_\s$. It follows from Lemma \ref{lem:uniqueclustervs} that $v=v_\s=v_r$. Hence $(\s,v)$ is not proper. 

If $(\s,v)$ is not proper, that is $v\notin\M$, then $v=v_\s=v_r$ for some (any) $r\in\s$. In particular, $\s=\deg v_r=\deg v$.
%
\endproof
\end{lem}

\begin{rem}[Alternative definition for MacLane clusters]
Let $\Sigma$ be the set of pairs $(\s,n)$, where $n\in\Z_+$ and $\s=D_v\cap \roots\neq\varnothing$ for some MacLane pseudo-valuation $v$ of degree $n$. It follows from Remark \ref{rem:Galois1disjoint2discoids}(2) and Theorem \ref{thm:valdiscoid} that the map $\Sigma_f^M\rightarrow \Sigma$, taking $(\s,v)\mapsto(\s,\deg v)$ is bijective.
\end{rem}

\begin{lem}\label{lem:lambdav}
Let $(\s,v)$ be a MacLane cluster. Then $\lambda_v=\min_{r\in\s}v_K(\phi(r))$ for any centre $\phi$ of $v$.
\proof
Let $\lambda=\min_{r\in\s}v_K(\phi(r))$. Since $\s\subset D_v=D(\phi,\lambda_v)$, we have $\lambda\geq \lambda_v$. Suppose $\lambda> \lambda_v$. Let $w=[v,w(\phi)=\lambda]$. Then $w>v$ and $\deg w=\deg \phi=\deg v$. But $\s=\roots\cap D_w$
for our choice of $\lambda$. This contradicts the fact that $(\s,v)$ is a MacLane cluster (Definition \ref{defn:Clusters}\ref{item:Clustersdefn(2)}).
\endproof
\end{lem}

\begin{nt}
Let $\m P\subset K[x]$ to be the subset of monic irreducible polynomials. For any $d\in\Z$, denote by $\m P_{\leq d}$ the set $\{g\in\m P\mid \deg g\leq d\}$.
\end{nt}

\begin{lem}\label{lem:lambdav2}
Let $(\s,v)$ be a proper MacLane cluster. Then 
\[\lambda_v=\max_{g\in \m P_{\leq \deg v}}\min_{r\in\s}v_K(g(r)).\]
\proof
Let $d=\deg v$. By Lemma \ref{lem:lambdav} we only need to show that $\lambda_v\geq\max_{g\in\m P_{\leq d}}\min_{r\in\s}v_K(g(r))$. Suppose not. Then there exists a polynomial $g\in\m P_{\leq d}$ such that $\lambda:=\min_{r\in\s}v_K(g(r))>\lambda_v$. Let $w\in\pM$ such that $D_w=D(g,\lambda)$ (Theorem \ref{thm:valdiscoid}). Then $\s\subseteq D_w\cap\roots$. By Lemma \ref{lem:degdiscoid} we have $\deg w\leq\deg g\leq\deg v$ and $w(g)\geq \lambda$. Since $\s\subset D_w\cap D_v$, either $D_v\subseteq D_w$ or $D_w\subsetneq D_v$ by Remark \ref{rem:Galois1disjoint2discoids}(2). If $D_w\subsetneq D_v$, then $w>v$ and $\s=D_w\cap\roots$, a contradiction, since $(\s,v)$ is a MacLane cluster. So $D_v\subseteq D_w$, that is $v\geq w$. Hence $v(g)\geq w(g)\geq\lambda>\lambda_v$. This gives a contradiction since $v(g)\leq\lambda_v$ by Lemma \ref{lem:v(lowerdegree)}.
\endproof
\end{lem}

\begin{lem}\label{lem:ClusterEquivCond}
Let $v\in\pM$ and $\s=D_v\cap\roots$. Then $(\s,v)\in\Sigma_f^M$ if and only if
\[\lambda_v=\max_{g\in\m P_{\leq\deg v}}\min_{r\in\s}v_K(g(r)).\]
\proof
One implication follows from Lemma \ref{lem:lambdav2}. Suppose 
\[\textstyle\lambda_v=\max_{g\in\m P_{\leq\deg v}}\min_{r\in\s}v_K(g(r)).\]
By Lemma \ref{lem:uniqueclustervs}, there exists a MacLane pseudo-valuation $w\geq v$ with $\deg w= \deg v$ such that $(\s,w)\in\Sigma_f^M$. Let $\lambda_w$ be the radius of $w$. Then Lemma \ref{lem:lambdav2} implies $\lambda_w=\lambda_v$. But this is possible only if $w=v$, by Lemma \ref{lem:centreparent}.
\endproof
\end{lem}

\begin{lem}\label{lem:Clusterdescent}
Let $v\neq v_0$ be a MacLane valuation, $\phi$ a strong key polynomial over $v$ and $\lambda\in\hat\Q$, $\lambda>v(\phi)$. Set $w=[v,w(\phi)=\lambda]$, $\s=D_v\cap\roots$, $\t=D_w\cap\roots$. If $\t\neq \varnothing$, then $(\s,v)$ is a MacLane cluster. 
\proof
First note that $\t\subseteq\s$. Let $g\in K[x]$ be any monic irreducible polynomial of degree $\deg g\leq\deg v$. Then $\deg g<\deg \phi$ and so $w(g)=v(g)$. Hence Proposition \ref{prop:2.7OS2} implies that \[v(g)=w(g)=\min_{r\in\t}v_K(g(r))\geq\min_{r\in\s}v_K(g(r))\geq v(g).\]
As $g$ was arbitrary, $(\s,v)$ is a MacLane cluster by Lemmas \ref{lem:v(lowerdegree)} and \ref{lem:ClusterEquivCond}.
\endproof
\end{lem}

\begin{lem}\label{lem:parentchildvaluation}
Let $(\s,v)$ be a MacLane cluster and let $(\t,w)$ be its parent. Then $v=[w,v(\phi)=\lambda_v]$ for any centre $\phi$ of $v$.
\proof
The lemma follows from Proposition \ref{prop:ValuationsInequality} and Lemma \ref{lem:Clusterdescent}.
\endproof
\end{lem}

\begin{prop}\label{prop:appendix}
Let $F\in O_K[x]$ monic and irreducible. Let $v,w\in\pM$ such that $v\leq v_F$, e.g.\ when $v\in\M$ and $F\in\KP v$. Then 
\[(v\wedge w)(F)=\min\{v(F),w(F)\}\]
In particular, if $v\not< w$, then $w(F)=(v\wedge w)(F)$.
\proof
The first part of the statement follows from the proof of \cite[Proposition 2.26]{KW}, defining $w\wedge v$. Suppose $v\not< w$. If $v=w$, then $(v\wedge w)(F)=w(F)$. If $v\not\leq w$, then $v\wedge w<v\leq v_F$. This implies $v(F)>(v\wedge w)(F)$ by \cite[Lemma 2.22]{KW}. Thus $(v\wedge w)(F)=w(F)$.
\endproof
\end{prop}

\begin{lem}\label{lem:fvaluation}
Let $v\in \pM$. Then
\[v(f)=v_K(c_f)+\sum_{F\in\m P,\,F\mid f}\deg F\cdot\frac{\lambda_{v\wedge v_{F}}}{\deg(v\wedge v_{F})}=v_K(c_f)+\sum_{r\in\roots}\frac{\lambda_{v\wedge v_r}}{\deg(v\wedge v_r)}.\]
\proof
Recall $f/c_f\in O_K[x]$. Then $f=c_f\cdot\prod_{F\in\m P,\,F\mid f} F$ and the factors $F$ in the product belong to $O_K[x]$. It suffices to show that $v(F)=\deg F\cdot\frac{\lambda_{v\wedge v_{F}}}{\deg(v\wedge v_{F})}$ for all $F\in\m P\cap O_K[x]$. Let $F\in\m P\cap O_K[x]$. By Lemma \ref{lem:vminimal}, the polynomial $F$ is $w$-minimal, for any MacLane valuation $w<v_F$. In particular, $F$ is $(v\wedge v_F)$-minimal. Hence 
\[\frac{(v\wedge v_F)(F)}{\deg F}=\frac{\lambda_{v\wedge v_F}}{\deg (v\wedge v_F)}.\]
by Theorem \ref{thm:minimality}. Since $v_F\not < v$, Proposition \ref{prop:appendix} shows $v(F)=(v\wedge v_F)(F)$ and so concludes the proof.
\endproof
\end{lem}



\subsection{Newton polygons}
Let $v$ be a MacLane valuation and $\phi\in\KP v$. Recall the definition of the Newton polygon $N_{v,\phi}(f)$.

\begin{defn}
The \textit{principal Newton polygon $N_{v,\phi}^-(f)$} is formed by the edges of $N_{v,\phi}(f)$ with slope $<-v(\phi)$.
\end{defn}

For any edge $L$ of $N_{v,\phi}^-(f)$ with slope $-\lambda$, define the MacLane valuation $v_L=[v,v_L(\phi)=\lambda]$. Then $L=L_{v_L}(f)$ (Notation \ref{nt:NewtonPolygon}). Denote by $\lambda_L$ the radius of $v_L$.

The aim of this subsection is proving the following theorem.

\begin{thm}\label{thm:NewtonClusters}
Let $v\in \M$ and $\phi\in\KP v$ as above.
\begin{enumerate}[label=(\roman*),leftmargin=1cm]
    \item If $(\t,w')$ is a MacLane cluster with centre $\phi'\veq_v\phi$ satisfying $w'(\phi)<\infty$, then $N_{v,\phi}^-(f)$ has an edge $L$ of slope $-w'(\phi)$ and $\i_{v_L}=|\t|/\deg \phi$. \label{item:ClustersToNewton}
    \item \label{item:NewtonToClusters}
 Conversely, for every edge $L$ of $N_{v,\phi}^-(f)$ there is a MacLane cluster $(\t,w_L)$ with $w_L\geq v_L$, $\deg w_L=\deg \phi$, $w_L(\phi)=\lambda_L$ and $|\t|=\i_{v_L}\deg \phi$.
\end{enumerate}
    In case \ref{item:NewtonToClusters}, if there exists a proper $(\s,w)\in\Sigma_f^M$ with $w=[v,w(\phi)=\lambda]$, $\lambda\geq \lambda_L$, then $w_L=v_L$. 
\end{thm}

We first recall the following result from \cite{FGMN}.

\begin{thm}[{\cite[Theorem 6.2]{FGMN}}]\label{thm:monicirredpoly}
Let $F\in O_K[x]$ be a monic irreducible polynomial and $r\in\bar K$ a root of $F$. Then $\phi\vdiv_v F$ if and only if $v_K(\phi(r))>v(\phi)$.
Moreover, if this condition holds, one also has:
\begin{enumerate}
    \item Either $F=\phi$, or $N_{v,\phi}(F)$ consists of one edge of slope $-v_K(\phi(r))$.\label{item:monicirredpoly1}
    \item $d:=\deg F/\deg \phi\in\Z_+$ and $F\veq_v\phi^d$. \label{item:monicirredpoly2}
\end{enumerate}
\end{thm}

\begin{lem}\label{lem:EdgesRoots}
Let $w=[v,w(\phi)=\lambda]$ be an augmentation of $v$. Let $\s_\lambda$ be the set of roots $r$ of $f$ satisfying $v_K(\phi(r))=\lambda$. Then $|\s_\lambda|/\deg \phi=\i_w-\io_w$.
\proof
Without loss of generality we can suppose $f$ monic. If $\lambda=\infty$, then $|\s_\lambda|=\ord_\phi(f)$ and the equality $|\s_\lambda|/\deg \phi=\i_w-\io_w$ follows from the definition of $\io_w,\i_w$. Hence suppose $\lambda<\infty$. 

We first show the statement for $f=F$ irreducible. In this case either $\s_\lambda=\varnothing$ or $\s_\lambda=\roots$. Suppose $\s_\lambda=\roots$, which means $v_K(\phi(r))=\lambda>v(\phi)$ for any (some) $r\in\roots$. Since $F\neq \phi$ (otherwise $\phi(r)=0$), Theorem \ref{thm:monicirredpoly} implies that $L_w(F)=N_{v,\phi}(F)$, $\io_w=0$ and $\i_w=\deg F/\deg \phi=|\s_\lambda|/\deg \phi$. Now suppose that $L_w(F)$ is an edge of $N_{v,\phi}(F)$. So $\i_w\geq 1$. We want to show $\s_\lambda\neq\varnothing$. Let $\i=\i_w$. Expand
\[F=\sum_{j=0}^da_j\phi^j,\qquad a_j\in K[x],\,\deg a_j<\deg \phi,\, a_d\neq 0.\]
By definition of $L_w(f)$ we have $w(a_j\phi^j)\geq w(a_\i\phi^\i)$ for all $j$. Therefore 
\[v(a_\i\phi^\i)=w(a_\i\phi^\i)-\i(\lambda-v(\phi))<w(a_j\phi^j)-j(\lambda-v(\phi))=v(a_j\phi^j)\]
for all $j<\i$. In particular, $v(a_0)>v(F)$, so $\phi\vdiv_v F$. Theorem \ref{thm:monicirredpoly} then implies that $-\lambda=-v_K(\phi(r))$ for any $r\in\roots$. Therefore $\s_\lambda\neq\varnothing$.

Let $f\in O_K[x]$ be any monic separable polynomial. Write $f=F_0\cdots F_t$, with $F_j\in O_K[x]$ monic irreducible. Denote by $\roots_j$ the set of roots of $F_j$ and by $\s_{\lambda,j}$ the elements $r\in\roots_j$ satisfying $v_K(\phi(r))=\lambda$. Clearly 
$\s_\lambda=\bigsqcup_j \s_{\lambda,j}$. Moreover, from \cite[Corollary 2.7]{FGMN}, we have \[L_w(f)=L_w(F_0)+\dots+L_w(F_t)\]
(see before \cite[Corollary 2.7]{FGMN} for a definition of $+$).
The lemma then follows from the first part of the proof.
\endproof
\end{lem}

\begin{prop}\label{prop:EdgesRoots}
Let $w=[v,w(\phi)=\lambda]$ be an augmentation of $v$ and let $\s=D_w\cap\roots$. Then $\i_w=|\s|/\deg \phi$.
\proof
By definition $\i_w=\sum_{\lambda'\geq \lambda}(\i_{w'}-\io_{w'})$, where $w'=[v,w'(\phi)=\lambda']$. Lemma \ref{lem:EdgesRoots} implies
\[\i_w\deg \phi=\sum_{\lambda'\geq\lambda}\left|\s_{\lambda'}\right|=\bigg|\bigcup_{\lambda'\geq\lambda}\s_{\lambda'}\bigg|=|\s|,\]
where $\s_{\lambda'}\subseteq\roots$ is the set of roots $r$ of $f$ satisfying $v_K(\phi(r))=\lambda'$.
\endproof
\end{prop}

Now we are ready to prove Theorem \ref{thm:NewtonClusters}.
\begin{proof}[Proof of Theorem \ref{thm:NewtonClusters}]


\ref{item:ClustersToNewton}. Let $(\t,w)$ be a cluster with centre $\phi'\veq_v\phi$ and $w(\phi)<\infty$. In particular, $\deg\phi=\deg\phi'$. 
Let $\lambda_\t=\min_{r\in\t}v_K(\phi(r))\geq w(\phi)$. Consider the MacLane valuation $w_\t=[v,w_\t(\phi)=\lambda_\t]$. Then $\t\subseteq D_{w_\t}\cap\roots$. By Remark \ref{rem:Galois1disjoint2discoids}(2) and Theorem \ref{thm:valdiscoid}, either $w_\t>w$ or $w_\t\leq w$. By definition of MacLane cluster we have $w_\t\leq w$. But then $\lambda_\t\leq w(\phi)$. Thus $\lambda_\t=w(\phi)$. Furthermore,
\[\t\subseteq D_{w_\t}\cap\roots\subseteq D_{w}\cap\roots=\t,\]
and so $\t=D_{w_\t}\cap\roots$.
Then Lemma \ref{lem:EdgesRoots} implies that $L_{w_\t}(f)$ is an edge of $N_{v,\phi}(f)$. The equality $\i_{w_\t}\deg \phi=|\t|$ follows from Proposition \ref{prop:EdgesRoots}.


\ref{item:NewtonToClusters}. Let $L$ be an edge of $N_{v,\phi}^-(f)$. Let $\t=D_{v_L}\cap\roots$. From Lemma \ref{lem:EdgesRoots} and Proposition \ref{prop:EdgesRoots} it follows that 
\[|\t|=\i_{v_L}\cdot\deg \phi\quad\text{ and }\quad\min_{r\in\t}v_K(\phi(r))=\lambda_L.\]
By Lemma \ref{lem:uniqueclustervs} there exists a unique MacLane pseudo-valuation $w_L\geq v_L$ such that $\deg w_L=\deg v_L=\deg\phi$ and $(\t,w_L)$ is a cluster. In particular, $w_L(\phi)=\lambda_L$ as
\[\lambda_L=v_L(\phi)\leq w_L(\phi)\leq\min_{r\in\t}v_K(\phi(r))=\lambda_L.\]

there exists a proper MacLane cluster $(\s,w)$ with $w=[v,w(\phi)=\lambda]$, $\lambda\geq \lambda_L$. Then $w\geq v_L$ and so $\s\subseteq\t$. Furthermore, $\deg w_L=\deg v_L=\deg w$; hence, by definition of cluster, if $\s=\t$ then $w=v_L=w_L$. So suppose $\s\subsetneq\t$. It follows from Lemma \ref{lem:ClusterInclusion} that $(\s,w)\subsetneq(\t,w_L)$.
Since $\phi$ is centre of $w$, Lemma \ref{lem:centreparent} implies that $\phi$ is also a centre of $w_L$. But we have already showed $w_L(\phi)=\lambda_L$, so $w_L=v_L$ as required.
\endproof
\end{proof}

\subsection{Residual polynomials}
In this subsection we will see that there is a close relationship between certain children $(\t,w)<(\s,v)$ and multiple irreducible factors of $f|_v$.
We will need the following result.

\begin{thm}[{\cite[Theorem 6.4]{FGMN}}]\label{thm:Newtondecomposition}
Let $v\in \M$ and let $\phi\in K[x]$ be a proper key polynomial over $v$. Every monic $g\in O_K[x]$ factorises into a product of monic polynomials in $O_K[x]$
\[g=g_0\cdot \phi^{\ord_\phi(g)}\prod_{\lambda,h} g_{\lambda,h},\]
where $-\lambda$ runs on the slopes of $N_{v,\phi}^-(g)$ and $h\in k_{w_\lambda}[X]$ runs on the monic irreducible factors of $g|_{w_\lambda}$, where $w_\lambda=[v,w_\lambda(\phi)=\lambda]$. 
Let $g=F_1,\dots,F_s$ be the factorisation of $g$ in monic irreducible polynomials $F_j\in O_K[x]$. Then $g_0$ is the product of all $F_j$ such that $\phi\notvdiv_v F_j$, while $g_{\lambda,h}$ is the product of all $F_j$ with 
$N_{v,\phi}(F_j)$ one-sided of slope $-\lambda$ and $F_j|_{w_\lambda}=h^l$ for some $l$. In particular,
\[\deg g_0=\deg g-l(N_{v,\phi}^-(g))\deg \phi, \qquad\deg g_{\lambda,h}=b_{w_\lambda}\cdot \ord_h(g|_{w_\lambda})\cdot \deg h\cdot \deg \phi,\]
where $b_{w_\lambda}$ (Notation \ref{nt:bvhvlv}) equals the denominator of $e_v\lambda$.
\end{thm}

Consider a MacLane valuation $v$. Assume $v\neq v_0$. Let $\phi_v$ be a centre of $v$. By Proposition \ref{prop:uniqueMacLanevalStrongPoly} there exists a unique MacLane valuation $v'$ over which $\phi_v$ is a strong key polynomial. Then $v=[v',v(\phi_v)=\lambda_v]$. Let $\s=D_v\cap\roots$.
We decompose
\begin{equation}\label{eqn:ffactorisation}
    f/c_f=f_0\phi_v^{\ord_{\phi_v}(f)}\prod_{\lambda,h} f_{\lambda,h},
\end{equation}
as in Theorem \ref{thm:Newtondecomposition} with respect to the principal Newton polygon $N_{v',\phi_v}^-(f)$. Recall $\epsilon_v=e_{v'}$ and $b_v$ equals the denominator of $\epsilon_v\lambda_v$.


\begin{lem}\label{lem:reductionNewtonpolygon}
If $\phi\in\KP v$ such that $\phi|_v$ is a multiple irreducible factor of $f|_v$, then $N_{v,\phi}^-(f)$ has an edge.
\proof
By Theorem \ref{thm:Newtondecomposition} it suffices to show that $f$ has a monic irreducible factor $F\neq\phi$ that $v$-divisible by $\phi$. Let $h=\phi|_v$. Since $f_{\lambda_v,h}|_v=h^{\ord_h(f|_v)}$, one has $f_{\lambda_v,h}\neq \phi$. As $f$ is separable, there exists a monic irreducible factor $F$ of $f_{\lambda_v,h}$ different from $\phi$. Thus $\phi\vdiv_v F$ by \cite[Theorem 5.3]{FGMN}.
\endproof
\end{lem}

\begin{lem}\label{lem:AugmentationMultipleFactor}
Let $w=[v,w(\phi)=\lambda]$ be an augmentation of $v$. Suppose $(\t,w)$ is a proper MacLane cluster. If $\phi|_v$ is irreducible\footnote{Note that $\phi|_v$ is irreducible if and only if $\phi$ is not a centre of $v$, by Lemma \ref{lem:keypolyreduction}.}, then $\ord_{\phi|_v}(f|_v)>1$.
\proof
Let $h=\phi|_v$. Lemma \ref{lem:keypolyreduction} implies $\phi\not\veq_{v}\phi_v$ and
\begin{equation}
\label{eqn:degvdegreduction}
    \deg \phi=b_v\deg h\deg \phi_v.
\end{equation}
Then by Theorem \ref{thm:Newtondecomposition} it suffices to show that 
$\deg f_{\lambda_v,h}>\deg \phi$. Since $\phi\not\veq_{v}\phi_v$ one has
$w(\phi_v)=\lambda_v$ by \cite[Lemmas 4.13,4.14]{Rut}.
Let $r\in\t$ and let $F\in O_K[x]$ be the minimal polynomial of $r$. Then 
\[v_K(\phi_v(r))=w(\phi_v)=v(\phi_v)=\lambda_v>v'(\phi_v),\]
where the first equality follows from Proposition \ref{prop:2.7OS2}.
Then either $F=\phi_v$ or $N_{v',\phi_v}(F)$ consists of one edge of slope $-\lambda_v$ by Theorem \ref{thm:monicirredpoly}. On the other hand
\[v_K(\phi(r))\geq w(\phi)=\lambda>v(\phi).\]
Again by Theorem \ref{thm:monicirredpoly} we have $F\veq_v\phi^l$, for some $l\in\Z_+$. In particular, $F\neq\phi_v$ and $F|_v=h^l$ by Propositions \ref{prop:v-equivalentResidual} and \ref{prop:ReductionDegreeProduct}\ref{item:ReductionProduct}. It follows from Theorem \ref{thm:Newtondecomposition} that $F\mid f_{\lambda_v,h}$.
Thus $|\t|\leq\deg f_{\lambda_v,h}$. Then Lemma \ref{lem:nonproperclusters} concludes the proof.
\endproof
\end{lem}

\begin{prop}\label{prop:degreeminimalreduction}
Suppose $-\lambda_v$ is the minimum slope of $N_{v',\phi_v}^-(f)$. Then $(\s,v)$ is not a degree-minimal MacLane cluster if and only if $b_v=1$ and
$f|_v$ has a multiple factor $h\in k_v[X]$ of degree $1$. 
\proof
Suppose $(\s,v)$ is a degree-minimal MacLane cluster. Suppose that $b_v=1$ and that $f|_v$ has a multiple irreducible factor $h\in k_v[X]$. Theorem \ref{thm:ChildrenReduction} implies that there exists a proper child $(\t,w)<(\s,v)$ with centre $\phi$ such that $\phi|_v=h$. Then $\deg\phi>\deg v$. Hence $\deg h>1$ by Lemma \ref{lem:keypolyreduction}. 

Now suppose $(\s,v)$ is not a degree-minimal MacLane cluster. Then there exists $w>v$ with $\deg v=\deg w$ such that $(\t,w)$ is a proper MacLane cluster, for some $\t\subseteq\roots$. Proposition \ref{prop:ValuationsInequality} implies that $w=[v,w(\phi)=\lambda]$ for some $\phi\in\KP v$ and $\lambda>v(\phi)$. In particular, $w$ is also an augmentation of $v'$. If $\phi\veq_v\phi_v$, then \[w(\phi_v)=\min\{\lambda,v(\phi_v-\phi)\}>\lambda_v.\] 
Hence $N_{v',\phi_v}^-(f)$ would have a slope $-w(\phi_v)$ smaller than $-\lambda_v$ by Theorem \ref{thm:NewtonClusters}\ref{item:ClustersToNewton}, contradicting our assumptions. Hence $\phi\not\veq_{v}\phi_v$. It follows that
\[\lambda_v=v(\phi_v)=v(\phi-\phi_v)=v'(\phi-\phi_v)\in\Gamma_{v'},\]
and so $b_v=1$. By Lemma \ref{lem:keypolyreduction} the polynomial $\phi|_v$ is irreducible and
$\deg\phi|_v=1$.
Then $\ord_{\phi|_v}(f|_v)>1$ by Lemma \ref{lem:AugmentationMultipleFactor}. 
\endproof
\end{prop}

\begin{thm}\label{thm:ChildrenReduction}
Suppose $(\s,v)$ is a proper MacLane cluster with $v\neq v_0$. 
\begin{enumerate}[label=(\roman*)]
    \item Let $h\in k_v[X]$ monic and irreducible such that $\ord_{h}(f|_v)>1$. There exists a proper child $(\t,w)<(\s,v)$ with centre $\phi$ such that $\phi|_v=h$. \label{item:MultipleFactor=>Child}
    \item Conversely, for any proper child $(\t,w)<(\s,v)$ with centre $\phi$ such that $\phi|_v$ is irreducible, one has $\ord_{\phi|_v}(f|_v)>1$. \label{item:Child=>MultipleFactor}
\end{enumerate}
In either case, $f_{\lambda_v,\phi|_v}=\prod_{r\in\t}(x-r)$ and $\ord_{\phi|_v}(f|_v)=|\t|/\deg w$.
\proof
Without loss of generality assume $f$ monic. 
Let $v'\in\M$ and $\phi_v\in\KP v$ as above and consider the factorisation (\ref{eqn:ffactorisation}) of $f$.

\ref{item:MultipleFactor=>Child}. Suppose that the monic irreducible polynomial $h\in k_v[X]$ is a multiple factor of $f|_v$. Then
\[f_{\lambda_v,h}|_v=h^{\ord_h(f|_v)}\qquad\text{where }\ord_h(f|_v)>1.\]
By \cite[Theorem 5.7]{FGMN} there exists $\phi\in \KP v$ such that $\phi|_v=h$. 
Let $\roots_h$ be the set of roots of $f_{\lambda_v,h}$ and set \[\lambda=\min_{r\in\roots_h}v_K(\phi(r)).\] 
Now $\phi$ is a proper key polynomial over $v$ since $\phi|_v$ is irreducible. 
Then \cite[Theorem 5.13]{FGMN} implies that $\phi\vdiv_v F$ for any irreducible monic factor $F$ of $f_{\lambda_v,h}$.
Hence $\lambda>v(\phi)$ by Theorem \ref{thm:monicirredpoly}. 
Therefore $w=[v,w(\phi)=\lambda]$ is an augmentation of $v$. 
Let $\t=D_w\cap \roots$. From the definition of $\lambda$ we have $\roots_h\subseteq\t$.
The pair $(\t,w)$ may not be a MacLane cluster. However, by Lemma \ref{lem:uniqueclustervs}, we can find a MacLane pseudo-valuation $w'\geq w$ with $\deg w'=\deg w$ such that $(\t,w')$ is an MacLane cluster. 
Let $\psi$ be a centre of $w'$. Then $\psi$ is a centre of $w$ by Lemma \ref{lem:centreparent}. It follows from Lemma \ref{lem:centre} that $\psi\in\KP v$ and $\psi\veq_v \phi$. Hence $\psi|_v=\phi|_v=h$ by Proposition \ref{prop:v-equivalentResidual}.
Therefore, by replacing $\phi$ with $\psi$ and $w$ with $w'$ if necessary, we can assume $(\t,w)$ is a MacLane cluster. 
Furthermore, 
\begin{equation*}
|\t|\geq|\roots_h|=\deg f_{\lambda_v,h}>b_v\deg h\deg v=\deg \phi
\end{equation*}
by Theorem \ref{thm:Newtondecomposition} and Lemma \ref{lem:keypolyreduction}. Lemma \ref{lem:nonproperclusters} implies that $(\t,w)$ is proper.

The MacLane cluster $(\t,w)$ may not be a child of $(\s,v)$.
Suppose there exists a (proper) MacLane cluster $(\t',w')$ such that $(\t,w)\subsetneq(\t',w')\subsetneq(\s,v)$. 
We want to show that $\phi$ is a centre of $w'$. Suppose $\deg w>\deg w'$. Then for any centre $\phi'$ of $w'$, $\deg \phi'<\deg \phi$ and so $w(\phi')=v(\phi')$. On the other hand, $w'>w$ and $w'(\phi')>v(\phi')$, so $w(\phi')\geq w'(\phi')>v(\phi')$, which gives a contradiction. Hence
Lemma \ref{lem:centreparent} implies that $\phi$ is also a centre of $(\t',w')$. 

\ref{item:Child=>MultipleFactor}. Let $(\t,w)<(\s,v)$ proper with centre $\phi$ such that $\phi|_v$ is irreducible. Then $w>v$. Proposition \ref{prop:ValuationsInequality} and Lemma \ref{lem:Clusterdescent} implies that $w=[v,w(\phi)=\lambda]$ for some $\lambda>v(\phi)$, since $(\t,w)$ is a child of $(\s,v)$. 
Lemma \ref{lem:AugmentationMultipleFactor} concludes the proof of \ref{item:Child=>MultipleFactor}.

In the proof of Lemma \ref{lem:AugmentationMultipleFactor} we showed that $|\t|\leq \deg f_{\lambda_v,\phi|_v}$. Then $\t=\roots_h$ from above. Finally, $\ord_{\phi|_v}(f|_v)=|\t|/\deg w$ by Theorem \ref{thm:Newtondecomposition} and (\ref{eqn:degvdegreduction}).
\endproof
\end{thm}

\begin{rem}
In \S\ref{sec:Residualpoly} we showed how to compute the reduction $f|_v$ algorithmically for any $v\in\M$, knowing a MacLane chain for $v$ (see also \cite[\S3]{FGMN}). Assume $v_K(r)>0$ for any $r\in\roots$ (in the next section we will see that we can always require this condition for our purpose). Suppose we know how to factorise polynomials in $k[X]$, e.g.\ $k$ is finite. Then we can algorithmically find MacLane chains for all MacLane valuations defining MacLane clusters, starting from the Newton polygon $N_{v_0,x}(f)$ and using the results \ref{thm:NewtonClusters}, \ref{lem:reductionNewtonpolygon}, \ref{lem:AugmentationMultipleFactor}, \ref{prop:degreeminimalreduction}, \ref{thm:ChildrenReduction}.
\end{rem}

\section{Model construction}\label{sec:ModConstr}
Suppose $\mathrm{char}(k)\neq 2$. Let $C/K$ be a hyperelliptic curve of genus $g\geq 1$. We can find a separable polynomial $f=c_f\prod_{r\in\roots}(x-r)\in K[x]$, where $v_K(r)>0$ for any $r\in\roots$, such that $C/K:y^2=f(x)$. 
Let $(\s_1,\mu_1),\dots,(\s_n,\mu_n)$ be all degree-minimal MacLane clusters for $f$. Note that if $r\in\s_i$ has minimal polynomial $F\in K[x]$ of degree $\deg\mu_i$, then $F$ is a centre of $\mu_i$ by Lemma \ref{lem:centreparent}, as $v_F\geq \mu_i$. Choose centres $\psi_1,\dots,\psi_n$ of $\mu_1,\dots,\mu_n$ respectively, with the following property:
\begin{equation}\label{eqn:centrechoice}
\text{\begin{tabular}{c}
If possible, choose $\psi_i$ equal to the minimal polynomial\\ 
of some root $r\in \s_i$ of $K$-degree $\deg \mu_i$.
\end{tabular}}
\end{equation}

Thanks to Lemma \ref{lem:centreparent}, for any proper MacLane cluster $(\s,v)\in\Sigma_f^M$ we inductively choose a centre $\phi_v$ as follows:
\begin{enumerate}[label=(\roman*)]
    \item If $(\s,v)$ is degree-minimal, that is $(\s,v)=(\s_i,\mu_i)$ for some $1\leq i\leq n$, fix $\phi_v=\psi_i$.
    \item If $(\s,v)$ has children of degree $\deg v$, choose one of them, say $(\t,w)$, and fix $\phi_v=\phi_w$.
\end{enumerate}

\begin{defn}
Let $(\s,v)$ be a proper MacLane cluster. A \textit{cluster chain for $v$} is MacLane chain
\[[v_0, v_1(\phi_1)=\lambda_1,\dots,v_m(\phi_m)=\lambda_m]\]
for $v$, where $\{\phi_w \mid (\t,w)\supseteq(\s,v)\}=\{\phi_1,\dots,\phi_m\}$.
\end{defn}

\begin{lem}\label{lem:clusterchainuniqueness}
Let $(\s,v)\in\Sigma_f^M$ proper and let $[v_0,\dots,v_m(\phi_m)=\lambda_m]$
be a cluster chain for $v$. Consider the chain of proper MacLane clusters
\[(\t_1,w_1)\supsetneq(\t_2,w_2)\supsetneq\dots\supsetneq(\t_s,w_s)=(\s,v)\]
satisfying:
\begin{enumerate}[label=(\alph*)]
    \item $(\t_1,w_1)\supseteq(\t,w)$ for any proper MacLane cluster $(\t,w)\supseteq(\s,v)$. \label{item:clusterchaina}
    \item $\phi_{w_{i}}\neq \phi_{w_{i+1}}$ for all $1\leq i< s$.\label{item:clusterchainb}
    \item For any $1\leq i<s$, the MacLane cluster $(\t_i,w_i)$ is the smallest MacLane cluster containing $(\t_{i+1},w_{i+1})$ and satisfying \ref{item:clusterchainb}. \label{item:clusterchainc}
\end{enumerate}
Then $m=s$, $\phi_i=\phi_{w_i}$ and $v_i=w_i$.
\proof
Clearly $\{\phi_w \mid (\t,w)\supseteq(\s,v)\}=\{\phi_{w_1},\dots,\phi_{w_s}\}$, with the centres $\phi_{w_i}$ all distinct. By definition of cluster chain $m\geq s$. However, if $m>s$, then $\phi_i=\phi_j$ for some $i<j$. This is not possible, as $v(\phi_i)=\lambda_i<\lambda_j=v(\phi_j)$ by \cite[Lemmas 4.21,4.22]{Rut}. Hence $m=s$.

Clearly $v_m=w_m$. Suppose there exists $i<m$ such that $\phi_i=\phi_v$. It follows that 
\[\lambda_i=v_i(\phi_i)=v(\phi_i)=\lambda_v=\lambda_m,\]
a contradiction by \cite[Lemma 4.21]{Rut}. Therefore $\phi_m=\phi_{w_m}$.
Let $\sigma\in S_{m-1}$ be the permutation such that $\phi_{w_i}=\phi_{\sigma(i)}$.
For any $i=1,\dots,m-1$, either $(\t_i,w_i)$ is degree-minimal or there exists a child $(\s',v')<(\t_i,w_i)$ not containing $(\t_{i+1},w_{i+1})$ such that $\phi_{w_i}=\phi_{v'}$ by \ref{item:clusterchainc}. 

Suppose $(\t_i,w_i)$ is degree-minimal. Let 
\[j_i=\max\{j=1,\dots,m\mid \deg \phi_j=\deg\phi_{w_i}\}.\]
Lemma \ref{lem:Clusterdescent} implies that $v_{j_i}$ defines a proper MacLane cluster of degree $\deg w_i$ and so $v_{j_i}=w_i$. In fact, $\phi_{j_i}$ must equal $\phi_{w_i}$ since $(\t_i,w_i)$ is degree-minimal, for our choice of centres. Therefore $\sigma(i)=j_i$ and so $w_i=v_{\sigma(i)}$.

Suppose $(\t_i,w_i)$ is not degree-minimal and let $(\s',v')<(\t_i,w_i)$ as above. Note that $(\s',v')$ does not contain in $(\s,v)$ and $(\s'\wedge \s,v'\wedge v)=(\t_i,w_i)$. Hence $w_i(\phi_{w_i})=v(\phi_{\sigma(i)})=\lambda_{\sigma(i)}$ by Proposition \ref{prop:appendix}. It follows that \[D_{w_i}=D(\phi_{w_i},w_i(\phi_{w_i}))=D(\phi_{\sigma(i)},\lambda_{\sigma(i)})=D_{v_{\sigma(i)}},\] 
and so $w_i=v_{\sigma(i)}$ from Theorem \ref{thm:valdiscoid}. 

We showed that $w_i=v_{\sigma(i)}$ for any $i=1,\dots,m-1$. Since $v_1<\dots<v_m$ and $w_1<\dots<w_m$ the permutation $\sigma$ must be the identity. 
\endproof
\end{lem}

\begin{nt}
Let $(\roots,w_\roots)$ denote the root of $(\Sigma_f^M,\supseteq)$ (Proposition \ref{prop:SigmaTree}). 
\end{nt}

\begin{lem}\label{lem:MaximalClusterDegree}
The pseudo-valuation $w_\roots$ is a degree $1$ MacLane valuation. Furthermore, $w_\roots>v_0$.
\proof
Let $w$ be the maximal element of \[\{w'\in\pM\mid D_{w'}\cap\roots=\roots,\,\deg w'=1\}.\] 
Note that the set is non-empty as $v_0$ belongs to it. If $w$ is not a valuation, then $|\roots|\leq 1$ by Lemma \ref{lem:nonproperclusters}, a contradiction. Hence $(\roots,w)$ is a proper MacLane cluster and so $w_\roots\leq w$. But then $w=w_\roots$ by definition of MacLane cluster since $\deg w=1$. Finally, 
\[\lambda_{w_\roots}\geq\min_{r\in\roots}v_K(r)>0,\]
by Lemma \ref{lem:ClusterEquivCond}, and so $w_\roots>v_0$.
\endproof
\end{lem}

\begin{lem}
Let $(\s,v)$ be a proper MacLane cluster. There exists a unique cluster chain for $v$. Furthermore, $v>v_0$.
\proof
The uniqueness follows by Lemma \ref{lem:clusterchainuniqueness}. Moreover, $v>v_0$ by Lemma \ref{lem:MaximalClusterDegree}.
We construct a cluster chain of $v$ recursively to prove the existence. First let $(\roots,w_\roots)$ as above.
Then $w_\roots=[v_0,w_\roots(\phi_\roots)=\lambda_\roots]$ is a cluster chain for $(\roots, w_\roots)$. Now let $(\s,v)$ be any MacLane cluster different from $(\roots,w_\roots)$ and consider its parent $(\t,w)$. By recursion we can assume that $w$ is equipped with a cluster chain
\[[v_0, \dots,v_{m-1}(\phi_{m-1})=\lambda_{m-1},v_m(\phi_m)=\lambda_m]\]
So $\phi_m=\phi_{w}$ from Lemma \ref{lem:clusterchainuniqueness}. If $\phi_{w}=\phi_v$, then
\[[v_0, \dots,v_{m-1}(\phi_{m-1})=\lambda_{m-1},v(\phi_m)=\lambda_v]\]
is a cluster chain for $v$. If $\phi_{w}\neq \phi_v$, Lemma \ref{lem:parentchildvaluation} implies that
\[[v_0, \dots,v_{m-1}(\phi_{m-1})=\lambda_{m-1},v_m(\phi_m)=\lambda_m, v(\phi_v)=\lambda_v]\]
is an  the augmentation chain for $v$. Proving it is a MacLane chain would conclude the proof. Suppose by contradiction that $\phi_v\veq_w \phi_{w}$. Then $\deg \phi_v=\deg \phi_{w}$. In particular, $(\t,w)$ is not degree-minimal. As $\phi_v\neq \phi_{w}$,
there exists a child $(\s',v')<(\t,w)$ such that $\phi_{w}=\phi_{v'}$. 
Hence $(\s\wedge\s',v\wedge v')=(\t,w)$. 
Set 
\[w'=[w,w'(\phi_{w})=\min\{\lambda_{v'},\lambda_v,w(\phi_v-\phi_{w})\}].\]
Therefore $w<w'\leq v'$. Moreover $v(\phi_{w})=\min\{\lambda_v,w(\phi_{w}-\phi_v)\}$, and so $v\geq w'$. But then $w<w'\leq v\wedge v'$ which gives a contradiction.
\endproof
\end{lem}

Let $h=1,\dots,n$ and consider the MacLane valuation
$\mu_h$ of the degree-minimal cluster $(\s_h,\mu_h)$. Let
\begin{equation}\label{eqn:muhclusterchain}
    [v_0,v_1(\phi_1)=\lambda_1,\dots, v_{m-1}(\phi_{m-1})=\lambda_{m-1},v_m(\phi_m)=\lambda_m]
\end{equation}
be a cluster chain for $\mu_h$. Then $\phi_m=\psi_h$. 
Denote $\phi=\psi_h$ and $v=v_{m-1}$. Denote by $\epsilon_h$ the ramification index $e_v=\epsilon_{\mu_h}$.
Let $g(x,y)=y^2-f(x)$ and expand
\[g=\sum_{i,j}a_{ij}\phi^iy^j,\qquad a_{ij}\in K[x],\,\deg a_{i,j}<\deg \phi.\]
Define the Newton polytopes 
\begin{align*}
    \Delta_h&=\text{convex hull}\lb\{(i,j):a_{ij}\neq 0 \}\rb\subset\R^2,\\
    \tilde\Delta_h&=\text{lower convex hull}\lb\{(i,j,v(a_{ij})): a_{ij}\neq 0\}\rb\subset\R^3.
\end{align*}
Consider the homeomorphic projection $s_h:\tilde\Delta_h\rightarrow\Delta_h$. Above every point $P\in\Delta_h$ there is a unique point $(P,\tilde \mu_h(P))\in\tilde\Delta_h$. This defines a piecewise affine function $\tilde \mu_h:\Delta_h\rightarrow\R$, and the pair $(\Delta_h,\tilde \mu_h)$ determines $\tilde\Delta_h$. Let $\tilde F$ be any $2$-dimensional (open) face of $\tilde\Delta_h$ and let $F=s_h(\tilde F)$. Define $\tilde v_{F,h}:\R^2\rightarrow\R$ to be the unique affine function coinciding with $\tilde \mu_h$ on $F$. Let $\lambda_F=\tilde v_{F,h}(0,0)-\tilde v_{F,h}(1,0)$. Define $\tilde\Delta_h^-\subseteq\tilde\Delta_h$ as the sub-polytope consisting of (the closure of) all $2$-dimensional faces $\tilde F$ of $\tilde\Delta_h$ with $\lambda_F>v(\phi)$. Clearly
\[\tilde\Delta_h^-=\text{lower convex hull}\lb\{(i,0,u): (i,u)\in N_{v,\phi}^-(f)\}\cup\{(0,2,0)\}\rb\subset \R^3.\]
where $N_{v,\phi}^-(f)$ is the principal Newton polygon of $f$ with respect to $v,\phi$. 
The image of $\tilde\Delta_h^-$ under $s_h$ will be denoted by $\Delta_h^-$.
The images of the $0$-,$1$- and $2$-dimensional (open)
faces of the polytope $\tilde{\Delta}_h^-$ under $s_h$ are
called $h$-vertices, $h$-edges and $h$-faces. Finally, a $\ast$-vertex, $\ast$-edge, $\ast$-face is respectively an $h$-vertex, $h$-edge, $h$-face for some $h=1,\dots,n$.


\begin{defn}
Let $G$ be a $h$-vertex, $h$-edge or $h$-face.
\begin{enumerate}[label=(\alph*)]
    \item Denote by $\tilde G$ the inverse image of $G$ under $s_h$.
    \item Denote by $\bar G$ the closure of $G$ in $\R^2$.
    \item Denote by $G_\Z$ the set of points $P$ of $G$ with $\epsilon_h\tilde\mu_h(P)\in\Z$.
    \item Denote by $G_\Z(\Z)$ the intersection $G_\Z\cap\Z$.
\end{enumerate}
Finally, define the \textit{denominator} of $G$, denoted $\delta_G$, as the common denominator of $\epsilon_h \tilde \mu_h(P)$ for every $P\in\bar G(\Z)$.
\end{defn}


Let $(\s,w)$ be a proper MacLane cluster centre $\phi_w=\psi_h$. Lemma \ref{lem:clusterchainuniqueness} implies that the cluster chain for $w$ is
\[[v_0,v_1(\phi_1)=\lambda_1,\dots, v_{m-1}(\phi_{m-1})=\lambda_{m-1},w(\psi_h)=\lambda_w]\]
where $v_i,\phi_i,\lambda_i$ are as in (\ref{eqn:muhclusterchain}).
Theorem \ref{thm:NewtonClusters} implies that there is a $1$-to-$1$ correspondence between proper MacLane clusters and $\ast$-faces. Given a proper MacLane cluster $(\s,w)$ we will denote by $F_w$ the corresponding $\ast$-face. If $\phi_w=\psi_h$, then $F_w$ is an $h$-face. Then $F_w$ has $3$ edges:
\begin{enumerate}[label=(\arabic*)]
    \item An $h$-edge, denoted $L_w$, linking the points $(\io_w,0)$ and $(\i_w,0)$.
    \item An $h$-edge, denoted $V_w$, linking the points $(\i_w,0)$ and $(0,2)$.
    \item An $h$-edge, denoted $V_w^0$, linking the points $(\io_w,0)$ and $(0,2)$.
\end{enumerate}
\begin{defn}
For any proper MacLane cluster $(\s,w)$ and any $l=1,\dots,n$, define $\tilde w_l:\R^2\rightarrow \R$ by \[\tilde w_l(x,y)=-w(\psi_l)x-\tfrac{w(f)}{2}y+w(f)\]
and $\tilde w:\R^{n+1}\rightarrow\R$ by \[\tilde w(x_1,\dots,x_n,y)=-(w(\psi_1)x_1+\dots+w(\psi_n)x_n)-\tfrac{w(f)}{2} y+w(f).\]
Finally define $e_{\tilde w}=(\tilde\Gamma_{w}:\Z)$, where $\tilde\Gamma_{w}=\tilde w(\Z^{n+1})$.
\end{defn}

Let $(\s,w)$ be a proper MacLane cluster with centre $\phi_w=\psi_h$. Then $\tilde w_h=\tilde v_{F_w,h}$. We will denote $(\s_{F_w},v_{F_w})=(\s,w)$.



\begin{defn}
Let $E$ be an $h$-edge. We say $E$ is \textit{inner} if $E=V_w$ for some proper MacLane cluster $(\s,w)\neq(\roots,w_\roots)$. In this case we say that $E$ \textit{bounds} $F_w$ and $F_{P(w)}$.
In all other cases $E$ is said \textit{outer} and \textit{bounds} only the $h$-face whose it is an edge.
\end{defn}

\subsection{Matrices}\label{subsec:Matrices}
Let $(\s,v)$ be a proper cluster with centre $\phi_v=\psi_h$. 
Let
\begin{equation} \label{eqn:clusterchainforv}
    [v_0,v_1(\phi_1)=\lambda_1,\dots,v_{m-1}(\phi_{m-1})=\lambda_{m-1},v_m(\phi_m)=\lambda_m]
\end{equation}
be the unique cluster chain for $v$. Construct the invariants and the rational functions attached to (\ref{eqn:clusterchainforv}) in \S\ref{sec:Residualpoly}.
Denote $v_-=v_{m-1}$. Recall $e_{v_-}=\epsilon_h$.


Let $E$ be either $L_v$ or $V_v$ or $V_v^0$ if $(\s,v)$ is degree-minimal. 
Let $v_E=[v_-,v_E(\psi_h)=\infty]$ if $E=V_v^0$, and $v_E=v$ otherwise.


\begin{defn}\label{defn:gammaoE}
Let $o=1,\dots,n$, $o\neq h$. Define $\gamma_{o,E}=\gamma_j$ if $\psi_o=\phi_j$, while
\[\gamma_{o,E}=\begin{cases}\psi_o\cdot \psi_h^{-\deg\psi_o/\deg\psi_h}&\mbox{if }\mu_o\geq v_E,\\
\psi_o\cdot \pi_{m-1}^{-e_{v_{m-1}}v(\psi_o)}&\mbox{otherwise},\end{cases}\]
if $\psi_o\neq\phi_j$, for all $j=1,\dots,m$.
\end{defn}

\begin{lem}\label{lem:gammaoEval0}
Let $o=1,\dots,n$, $o\neq h$. Then $\gamma_{o,E}$ is a well-defined element of $K(x)$ satisfying $v_F(\gamma_{o,E})=0$ for any $\ast$-face $F$ bounded by $E$.
\proof
Let $F$ be any $\ast$-face bounded by $E$.
Then $(\s,v)\leq (\s_F,v_F)$.
Lemma \ref{lem:clusterchainuniqueness} implies that $v_F(\phi_{j})=v(\phi_j)$ for all $j<m$. So the statement is trivial if $\psi_o=\phi_j$, for some $j<m$. Suppose $\psi_o\neq\phi_j$, for all $j<m$. Then $\mu_o\not\leq v$.
In particular, $\mu_o\not<v_F$ and so $v_F(\phi_o)=(v_F\wedge \mu_o)(\phi_o)$ by Proposition \ref{prop:appendix}. 

Suppose $v_E\leq\mu_o$. 
Then $v_F\leq\mu_o$ and so $\psi_o$ is $v_F$-minimal by Lemma \ref{lem:vminimal}. It follows that $\deg v\mid\deg\psi_o$ by Lemma \ref{lem:minimalityequivalence}. Theorem \ref{thm:minimality} implies that 
\[\frac{v(\psi_o)}{\deg \psi_o}=\frac{\lambda_v}{\deg v}\qquad\text{and}\qquad\frac{v_F(\psi_o)}{\deg \psi_o}=\frac{v_F(\psi_h)}{\deg v},\]
since $v_F=v$ when $F=F_v$. Therefore $v_F(\gamma_{o,L})=0$.

Suppose $v_E\not\leq\mu_o$. First we want to show that \begin{equation}\label{eqn:vEwedge}
    v_E\wedge \mu_o=v_F\wedge \mu_o.
\end{equation} 
Note that either $v_E\wedge \mu_0\leq v_F$ or $v_E\wedge \mu_0>v_F$ since $v_E\geq v_F$. 
If $E=V_v^0$ (and so $v$ is degree-minimal), then $v_F=v$. If $v_E\wedge \mu_o\leq v$, then (\ref{eqn:vEwedge}) follows. Suppose $v_E\wedge \mu_o>v$. Then $\mu_o>v$ and so $\mu_o(\psi_h)=v(\psi_h)$ by Lemma \ref{lem:clusterchainuniqueness}. 
Furthermore, $\psi_h$ is a centre of $v_E\wedge \mu_o\leq v_E$. 
But then Lemma \ref{lem:centreparent} and Proposition \ref{prop:appendix} imply that
\[\mu_o(\psi_h)=(v_E\wedge \mu_o)(\psi_h)=\lambda_{v_E\wedge\mu_o}>\lambda_{v}=v(\psi_h)=\mu_o(\psi_h),\]
a contradiction. If $E\neq V_v^0$, then $v_E=v$. Since $v\wedge\mu_o<v$ defines a MacLane cluster by Lemma \ref{lem:wedgecluster}, we have $v\wedge\mu_o\leq v_F$. Hence (\ref{eqn:vEwedge}). 

It follows from (\ref{eqn:vEwedge}) and Proposition \ref{prop:appendix} that
\begin{equation}\label{eqn:vEpsio}
    v_E(\psi_o)=(v_E\wedge \mu_o)(\psi_o)=(v_F\wedge \mu_o)(\psi_o)=v_F(\psi_o).
\end{equation}
Hence it suffices to show that $v(\psi_o)\in\Gamma_{v_{m-1}}$. 
By Proposition \ref{prop:ValuationsInequality} write
\[v\wedge\mu_o=[v_{a-1},(v\wedge\mu_o)(\phi_a)=\lambda_a'],\]
for some $a\leq m$ and $\lambda_a'\leq\lambda_a$.

If $v\leq \mu_o$, then $v$ is degree-minimal. It follows that $v\wedge\mu_o$ appears in the cluster chain for $\mu_o$ by Lemma \ref{lem:clusterchainuniqueness}. Therefore $v(\psi_o)\in\Gamma_{v_{a-1}}\subseteq\Gamma_{v_{m-1}}$ by Remark \ref{rem:degreeramificationindexChain}.

If $v\not\leq\mu_o$, then $v\wedge\mu_o<v$. By Lemma \ref{lem:wedgecluster}, the valuation $v\wedge \mu_o$ defines a proper MacLane cluster $(\s',v\wedge \mu_o)\supsetneq(\s,v)$. Let $(\t,w)\in\Sigma_f^M$ such that 
\[(\s,v)\subseteq(\t,w)<(\s',v\wedge \mu_o).\]
Since $\mu_o\not\geq w$, if $\psi_w=\psi_{v\wedge \mu_o}$, then $v\wedge\mu_o$ appears in the cluster chain for $\mu_o$ by Lemma \ref{lem:clusterchainuniqueness}. Therefore $v(\psi_o)\in\Gamma_{v_{m-1}}$ as above. Finally, if $\psi_w\neq\psi_{v\wedge \mu_o}$, then $v\wedge\mu_o$ appears in the cluster chain for $v$ again by Lemma \ref{lem:clusterchainuniqueness}. Since $v\wedge\mu_o<v$, one has $(v\wedge\mu_o)(g)\in\Gamma_{v_{m-1}}$ for any $g\in K[x]$. In particular, $v(\psi_o)\in\Gamma_{v_{m-1}}$ from (\ref{eqn:vEpsio}). 
\endproof
\end{lem}

Let $E_v^\ast$ be the unique affine function $\Z^2\rightarrow\Z$ with $E_v^\ast|_E=0$ and $E_v^\ast|_{F_v}\geq 0$. Choose $P_0,P_1\in\Z^2$ such that $E_v^\ast(P_0)=0$ and $E_v^\ast(P_1)=1$.

\begin{defn}
Define the slopes $[s_1^E,s_2^E],$ at $E$ to be
\[s_1^E= \delta_E \epsilon_h\lb \tilde v_{h}(P_1)-\tilde v_{h}(P_0)\rb,\]
\[s_2^E=\begin{cases}\delta_E \epsilon_h\lb \tilde w_{h}(P_1)-\tilde w_{h}(P_0)\rb&\mbox{if }$E$\mbox{ inner, with $(\s,v)<(\t,w)$,}\\\lfloor s_1^E-1\rfloor&\mbox{if }$E$\mbox{ outer.}\end{cases}\]\normalsize
\end{defn}

Let $\delta=\delta_E$. Pick fractions $\frac{n_i}{d_i}\in\Q$ such that
\[s_1^E=\frac{n_0}{d_0}>\dots>\frac{n_{r_E+1}}{d_{r_E+1}}=s_2^E,\quad\text{with\scalebox{0.9}{$\quad\begin{vmatrix}n_i\!\!\!&n_{i+1}\cr d_i\!\!\!&d_{i+1}\cr\end{vmatrix}=1$}}.\]
Let $r=r_E$. Redefine $n_{r+1}=-1$, $d_{r+1}=0$ if $E$ is outer.

Write $\tilde E=\tilde P_0+\nu\R$, with $\delta\nu=(\delta a_x, \delta a_y, \delta a_z)\in\Z^2\times \frac{1}{\epsilon_h}\Z$ primitive and such that $(a_x,a_y)$ goes counterclockwise along $\partial F_v$.
Let $o\neq h$. By Definition \ref{defn:gammaoE} and Lemma \ref{lem:gammadecomposition}, we can uniquely write 
\[\gamma_{o,E}=\psi_1^{m_{1o}}\cdots \psi_n^{m_{no}}\cdot\pi^{m_{(n+2)o}}\] 
Define $\nu_o\in\R^{n+2}$ by $\nu_o=(m_{1o},\dots,m_{no},0,m_{(n+2)o})$.

Now consider the embedding $\iota_h:\R^3\hookrightarrow\R^{n+2}$ given by \[(x_h,y,z)\mapsto(0,\dots,0,x_h,0,\dots,0,y,z),\]
where $x_h$ is the $h$-th coordinate in $\R^{n+2}$. Define $\nu_h^\R=\iota_h(\delta\nu)$.
Write $P_1-P_0=(b_x,b_y)$ and define $\omega_i^\R=\iota_h\big(d_ib_x,d_ib_y,\frac{n_i}{\delta \epsilon_h}\big)\in\R^{n+2}$ for any $i=0,\dots,r+1$. 
The vectors above define hyperplanes in $\R^{n+2}$,
\[\mathcal{P}_{E,i}=\nu_1\R+\dots+\nu_n\R+\omega_i\R\qquad i=0,\dots,r+1.\]

Let $M_{E,i}^\R\in M_{n+2}(\R)$ be the matrix given by \[M_{E,i}^\R=(\nu_1,\dots,\nu_{h-1},\nu_h^\R,\nu_{h+1},\dots,\nu_n,\omega_i^\R,-\omega_{i+1}^\R)\] 
where the vectors represent the columns of $M_{E,i}^\R$.
Then\footnote{See Appendix \ref{appendix:Matrices} for more details.}
\[\det M_{E,i}^\R=\textstyle\prod_{o=1}^{m-1}e_o\cdot \frac{1}{e_{v_{m-1}}}=1.\] 
Moreover, all entries of $M_{E,i}^\R$ are integers except possibly $\delta a_z\in\tfrac{1}{\epsilon_h}\Z$, and $\tfrac{n_i}{\delta \epsilon_h}$, $-\tfrac{n_{i+1}}{\delta \epsilon_h}$, rational numbers in $\tfrac{1}{\delta\epsilon_h}\Z$.
Pick $k_i$ with
\[k_i\equiv -n_i(\delta \epsilon_ha_z)^{-1}\mod\delta.\]
This is possible as $\delta\nu$ is primitive in $\Z^2\times\frac{1}{\epsilon_h}\Z$. 
Let $\tau\in S_{n+2}$ be a permutation such that $\phi_o=\psi_{\tau(o)}$ for all $o=1,\dots,m$ and $\tau(n+1)=n+1$, $\tau(n+2)=n+2$.
Define the vectors 
\[
\nu_h=\nu_h^\R+\sum_{o=1}^{m-1}c_o\nu_{\tau(o)}\delta a_z,\quad
\omega_i=\omega_i^\R+k_i\frac{\nu_h^\R}{\delta}+\sum_{o=1}^{m-1}c_o\nu_{\tau(o)}(\tfrac{n_i}{\delta \epsilon_h}+k_ia_z),
\]
where $c_o=e_{v_{o-1}}\ell_o$. The next lemma shows that they belong to $\Z^{n+2}$.

\begin{lem}\label{lem:integralisation}
Write $\sum_{o=1}^{m-1}c_o\nu_{\tau(o)}=(a_1,\dots,a_{n+2})$. Then
\[a_{\tau(j)}=\begin{cases}
\epsilon_h\ell_j\ell_{j+1}'\cdots\ell_{m-1}'&\text{if }j<m,\\
0&\text{if }m\leq j\leq n+1,\\
\epsilon_h\ell_1'\cdots\ell_{m-1}'-1&\text{if }j=n+2.
\end{cases}\]
In particular, $\nu_h,\omega_i\in\Z^{n+2}$.
\proof
Recall $\gamma_{E,\tau(o)}=\gamma_o$ for any $o=1,\dots,m$. If $j<m$ Lemma \ref{lem:gammadecomposition} implies
\small\begin{align*}
    a_{\tau(j)}&=c_je_j-\textstyle\sum_{o=j+1}^{m-1}c_oh_o\ell_j\ell_{j+1}'\cdots\ell_{o-1}'\\
    &=e_{v_j}\ell_j-\textstyle\sum_{o=j+1}^{m-1}e_{v_{o-1}}(\ell_oh_o)\ell_j\ell_{j+1}'\cdots\ell_{o-1}'\\
    &=\ell_j\big( e_{v_j}+\textstyle\sum_{o=j+1}^{m-1}e_{v_{o}}\ell_{j+1}'\cdots\ell_{o}'-\textstyle\sum_{o=j+1}^{m-1}e_{v_{o-1}}\ell_{j+1}'\cdots\ell_{o-1}'\big)\\
    &=\ell_j\lb e_{v_j} +e_{v_{m-1}}\ell_{j+1}'\cdots\ell_{m-1}'-e_{v_{j}}\rb=e_{v_{m-1}}\ell_j\ell_{j+1}'\cdots\ell_{m-1}',
\end{align*}\normalsize
where we used $\ell_oh_o+\ell_o'e_o=1$. If $m\leq j\leq n+1$, then the $\tau(j)$-th coordinate of $\nu_{\tau(o)}$ is $0$ for all $o=1,\dots,m-1$; so $a_{\tau(j)}=0$. Finally
\begin{align*}a_{n+2}&=-\textstyle\sum_{o=1}^{m-1}c_oh_o\ell_1'\cdots\ell_{o-1}'=\textstyle\sum_{o=1}^{m-1}e_{v_{o}}\ell_{1}'\cdots\ell_{o}'-\textstyle\sum_{o=1}^{m-1}e_{v_{o-1}}\ell_{1}'\cdots\ell_{o-1}'\\
&=e_{v_{m-1}}\ell_1'\cdots\ell_{m-1}'-1,
\end{align*}
as required.
\endproof
\end{lem}
Define $M_{E,i}=(\nu_1,\dots,\nu_n,\omega_i,-\omega_{i+1})\in M_{(n+2)}(\Z)$, where the vectors represent the columns of $M_{E,i}$.
Note that $\det M_{E,i}=\det M_{E,i}^\R=1$.
Let us describe $M_{E,i}$ as product of simpler matrices. Let $\varepsilon_1,\dots,\varepsilon_{n+2}\in \R^{n+2}$ be the standard basis of $\R^{n+2}$. Define $\kappa_i=\tfrac{k_i}{\delta}\varepsilon_{h}$ and $\xi=\sum_{o=1}^{m-1}c_o\varepsilon_{\tau(o)}$.
Define 
\[
\begin{array}{l}
T_h=(\varepsilon_1,\dots,\varepsilon_{h-1},\varepsilon_h+\delta a_z\cdot \xi,\varepsilon_{h+1},\dots,\varepsilon_n,\varepsilon_{n+1}+\tfrac{n_i}{\delta \epsilon_h}\xi,\varepsilon_{n+2}-\tfrac{n_{i+1}}{\delta \epsilon_h}\xi),\\
T=(\varepsilon_1,\dots,\varepsilon_n,\varepsilon_{n+1}+\kappa_i,\varepsilon_{n+2}-\kappa_{i+1}).
\end{array}
\]
Then $M_{E,i}=M_{E,i}^\R\cdot T_h\cdot T$. Now we want to describe $M_{E,i}^{-1}$. It follows from before that $M_{E,i}^{-1}=T^{-1}\cdot T_h^{-1} \cdot (M_{E,i}^\R)^{-1}$, where
\[
\begin{array}{l}
T_h^{-1}=(\varepsilon_1,\dots,\varepsilon_{h-1},\varepsilon_h-\delta a_z\cdot \xi,\varepsilon_{h+1},\dots,\varepsilon_n,\varepsilon_{n+1}-\tfrac{n_i}{\delta \epsilon_h}\xi,\varepsilon_{n+2}+\tfrac{n_{i+1}}{\delta \epsilon_h}\xi),\\
T^{-1}=(\varepsilon_1,\dots,\varepsilon_n,\varepsilon_{n+1}-\kappa_i,\varepsilon_{n+2}+\kappa_{i+1}),
\end{array}
\]
It remains to describe $(M_{E,i}^\R)^{-1}$. First note that the $h$-th, $(n+1)$-th and $(n+2)$-th columns of $(M_{E,i}^\R)^{-1}$ are respectively 
\begin{align*}
\iota_h&\big((b_y/\delta,n_{i+1}a_y-\delta \epsilon_hd_{i+1}a_zb_y, n_ia_y-\delta \epsilon_hd_ia_zb_y)\big),\\
\iota_h&\big((-b_x/\delta,-n_{i+1}a_x+\delta \epsilon_hd_{i+1}a_zb_x,-n_ia_x+\delta \epsilon_hd_ia_zb_x)\big),\\
\iota_h&\big((0,\delta \epsilon_hd_{i+1},\delta \epsilon_hd_i)\big).
\end{align*}
Let $o=1,\dots,n$. Lemma \ref{lem:gammadecomposition} and Definition \ref{defn:gammaoE} imply that we can write
\begin{equation}\label{eqn:psioingammas}
\psi_o^{\epsilon_h}=\gamma_{1,E}^{\alpha_{1o}}\cdots\gamma_{h-1,E}^{\alpha_{(h-1)o}}\cdot\psi_h^{\alpha_{ho}}\cdot\gamma_{h+1,E}^{\alpha_{(h+1)o}}\cdots\gamma_{n,E}^{\alpha_{no}}\cdot\pi^{\alpha_{\pi o}},
\end{equation}
for some unique $\alpha_{1o},\dots,\alpha_{no},\alpha_{\pi o}\in\Z$. Let $\tilde\alpha_{oj}=\alpha_{oj}/\epsilon_h$. Define 
\[
\tilde\nu_o=\begin{cases}(\tilde\alpha_{o1},\dots,\tilde\alpha_{on},0,0),&\mbox{if }o\neq h\\
\tfrac{1}{\delta}(\tilde\alpha_{h1}b_y,\dots,\tilde\alpha_{hn}b_y, b_x,0)&\mbox{if }o=h
\end{cases}
\]
Finally, 
define 
\begin{align}
\tilde\omega_i=\delta \epsilon_hd_i\big(&\big(\tfrac{n_i}{\delta \epsilon_hd_i}a_y-a_zb_y\big)\tilde\alpha_{h1}+\tilde\alpha_{\pi 1},\dots\label{eqn:tildeomega}\\
&\dots,\big(\tfrac{n_i}{\delta \epsilon_hd_i}a_y-a_zb_y\big)\tilde\alpha_{hn}+\tilde\alpha_{\pi n},-\tfrac{n_i}{\delta \epsilon_hd_i}a_x+a_zb_x,1\big).\notag
\end{align}
From the definition of $M_{E,i}^\R$ it follows that
\[(M_{E,i}^\R)^{-1}=\begin{psmallmatrix}\tilde\nu_1\vspace{-5pt}\cr\vdots\cr\tilde\nu_n\cr\tilde\omega_{i+1}\cr\tilde\omega_i\end{psmallmatrix},\]
where the vectors are the rows of the matrix.
Lemma \ref{lem:gammadecomposition} gives an explicit of $(M_{E,i}^\R)^{-1}$. Note also that for the structure of $T^{-1}$ and $T_h^{-1}$ the $\tau(o)$-th row of $M_{E,i}^{-1}$ coincides with the $\tau(o)$-th row of $(M_{E,i}^\R)^{-1}$, when $o>m$. Define
\[\mathcal{P}_{E,i}^{\perp +}=\tilde\omega_i\R_+,\] 
ray perpendicular to the hyperplane $\mathcal{P}_{E,i}$.

\begin{rem}
Note that $\tilde\nu_{\tau(o)}=\varepsilon_{\tau(o)}$ for $m<o\leq n$.
\end{rem}

\begin{lem}\label{lem:inneredge}
Suppose $E$ is inner, with $(\s,v)<(\t,w)$. Then $\tilde v_h|_E=\tilde w_h|_E$.
\proof
Recall that $E=V_v$. If $F_w$ is an $h$-face, the result trivially follows, as $E=V_{w}^0$. 

Suppose $F_w$ is not an $h$-face. By definition of cluster chain we have $w=v_-$. The polynomial $\psi_h$ is $w$-minimal, hence $\tfrac{\lambda_{w}}{\deg w}=\tfrac{w(\psi_h)}{\deg v}$ by Theorem \ref{thm:minimality}. From Lemma \ref{lem:fvaluation} and Proposition \ref{prop:EdgesRoots} it follows that
\small\[\tilde v_h(\i_v,0)=v(f)-\tfrac{|\s|}{\deg v}\cdot\lambda_v=w(f)-\tfrac{|\s|}{\deg w}\cdot \lambda_{w}=w(f)-\tfrac{|\s|}{\deg v}\cdot w(\psi_h)=\tilde w_h(\i_v,0).\]\normalsize
This concludes the proof since $\tilde v_h(0,2)=0=\tilde w_h(0,2)$. 
\endproof
\end{lem}

\begin{lem}\label{lem:fan}
We have
\[\tilde\omega_0=e_{\tilde v}(v(\psi_1),\dots,v(\psi_n),\tfrac{v(f)}{2},1).\]
Let $r=r_E$. Then
\[\tilde\omega_{r+1}=\begin{cases}
e_{\tilde w}(w(\psi_1),\dots,w(\psi_n),\tfrac{w(f)}{2},1)&\text{ if $E$ inner, with $(\s,v)<(\t,w)$,}\\
(-a_y\tilde\alpha_{h1},\dots,-a_y\tilde\alpha_{hn},a_x,0)&\text{ if $E$ outer.}
\end{cases}\]
\proof
Note that $\delta_{F_v}=\delta_Ed_0$ and $\delta_{F_v}\epsilon_h=e_{\tilde v}$. Recall $\tilde v_{F_v,h}=\tilde v_h$ and
\[\tilde v_h(x,y)=-\lambda_vx-\tfrac{v(f)}{2}y+v(f).\]
Then since $\nu$ and $(b_x,b_y,\tfrac{n_0}{\delta \epsilon_hd_0})$ generate $\tilde F_v$ (face of $\tilde\Delta_h$), we have
\begin{equation}\label{eqn:lambdaepsilon}
    \tfrac{n_0}{\delta \epsilon_hd_0}a_y-a_zb_y=\lambda_v\quad\text{and}\quad-\tfrac{n_0}{\delta \epsilon_hd_0}a_x+a_zb_x=\tfrac{v(f)}{2}.
\end{equation}
By (\ref{eqn:psioingammas}) and Lemmas \ref{lem:gammapival} and \ref{lem:gammaoEval0}, we have 
$v(\psi_o)=\lambda_v\tilde\alpha_{ho}+\tilde\alpha_{\pi o}$ for any $o=1,\dots,n$.
Hence the description of $\tilde\omega_0$ follows from (\ref{eqn:lambdaepsilon}).

Suppose that $E$ is inner, with $(\s,v)<(\t,w)$. Then either $w=v_-$ or $w=[v_-,w(\psi_h)=\lambda_w]$. In either case, $\delta_Ed_{r+1}\epsilon_h=e_{\tilde w}$. We have
\[\tilde w_h(x,y)=-w(\psi_h)x-\tfrac{w(f)}{2}y+w(f).\]
Since $\tilde w_h|_E=\tilde v_h|_E$ by Lemma \ref{lem:inneredge} and $\frac{n_{r+1}}{\delta \epsilon_hd_{r+1}}=\tilde w_h(P_1)-\tilde w_h(P_0)$, the vectors $\nu$ and $(b_x,b_y,\tfrac{n_{r+1}}{\delta \epsilon_hd_{r+1}})$ generate the plane $z=\tilde w_h(x,y)$ in $\R^3$. Hence
\[
    \tfrac{n_{r+1}}{\delta \epsilon_hd_{r+1}}a_y-a_zb_y=v_{F}(\psi_h)\quad\text{and}\quad\tfrac{n_{r+1}}{\delta \epsilon_hd_{r+1}}a_x-a_zb_x=\tfrac{v_F(f)}{2}.
\]
Similarly to before, by (\ref{eqn:psioingammas}) and Lemmas \ref{lem:gammapival} and \ref{lem:gammaoEval0}, we have 
$v_F(\psi_o)=v_F(\psi_h)\tilde\alpha_{ho}+\tilde\alpha_{\pi o}$ for any $o=1,\dots,n$. The description of $\tilde \omega_{r+1}$ follows, for $E$ inner.

Finally, suppose that $E$ is outer. Then $n_{r+1}=-1$ and $d_{r+1}=0$. The description of $\tilde\omega_{r+1}$ follows directly from the definition.
\endproof
\end{lem}

\subsection{Toroidal embedding}\label{subsec:TorEmbedding}
Let us start this subsection with the following notation.
\begin{nt}
Let $A$ be a ring and let $a_1,\dots,a_n\in A^\times$, for some $n\in\Z_+$. For any matrix $M=(m_{ij})\in\SL_n(\Z)$ denote by $(a_1,\dots,a_n)\bullet M$ the vector
\[(a_1^{m_{11}}\cdots a_n^{m_{n1}},\dots,a_1^{m_{1n}}\cdots a_n^{m_{nn}}).\]
\end{nt}

Denote by $m_{\ast\ast}$ and $\tilde m_{\ast\ast}$ the entries of $M_{E,i}$ and $M_{E,i}^{-1}$ respectively. Note that $\tilde m_{(n+1)(n+2)},\tilde m_{(n+2)(n+2)}\geq 0$. Then the coordinate transformation
\begin{align*}
(X_1,\dots,X_n,Y,Z)&= (x_1,\dots,x_n,y,\pi)\bullet M_{E,i},\\
(x_1,\dots,x_n,y,\pi)&=(X_1,\dots,X_n,Y,Z)\bullet M_{E,i}^{-1}
\end{align*}
gives the ring isomorphism
\small\[K[x_1^{\pm 1},\dots,x_n^{\pm 1},y^{\pm 1}]\stackrel{M_{E,i}}{\simeq}\frac{O_K[X_1^{\pm 1},\dots, X_n^{\pm 1},Y^{\pm 1}, Z^{\pm 1}]}{(\pi-X_1^{\tilde m_{1(n+2)}}\cdots X_n^{\tilde m_{n(n+2)}}Y^{\tilde m_{(n+1)(n+2)}}Z^{\tilde m_{(n+2)(n+2)}})}.\]\normalsize
Define
\[R=\frac{O_K[X_1^{\pm 1},\dots, X_n^{\pm 1},Y, Z]}{(\pi-X_1^{\tilde m_{1(n+2)}}\cdots X_n^{\tilde m_{n(n+2)}}Y^{\tilde m_{(n+1)(n+2)}}Z^{\tilde m_{(n+2)(n+2)}})}.\]

For any $h$-edge $E$, define cones in $\R^{n+1}\times\R_+$
\begin{center}
\begin{tabular}{lll}
    $0$-dimensional cone & $\sigma_0=\{0\},$&  \\
    $1$-dimensional cones & $\sigma_{E,i}=\mathcal{P}_{E,i}^{\perp +}$ & $(0\leq i\leq r+1)$,\\
    $2$-dimensional cones & $\sigma_{E,i,i+1}=\mathcal{P}_{E,i}^{\perp +}+\mathcal{P}_{E,i+1}^{\perp +}$ & $(0\leq i\leq r)$.
\end{tabular}
\end{center}
The set of all such cones from all $E$ is a fan $\Sigma$ from Appendix \ref{sec:appendix:fan}.
Recall
\begin{align*}
    &\mathcal{P}_{E,i}=\nu_1\R+\dots+\nu_n\R+\omega_i\R&&=m_{\ast 1}\R+\dots+m_{\ast n}\R+m_{\ast (n+1)}\R,\\
    &\mathcal{P}_{E,i+1}=\nu_1\R+\dots+\nu_n\R+\omega_{i+1}\R&&=m_{\ast 1}\R+\dots+m_{\ast n}\R+m_{\ast (n+2)}\R,\\
    &\mathcal{P}_{E,i}\cap\mathcal{P}_{E,i+1}=\nu_1\R+\dots+\nu_n\R&&=m_{\ast 1}\R+\dots+m_{\ast n}\R,
\end{align*}
\[\sigma_{E,i}=\tilde m_{(n+2)\ast}\R_+,\quad\sigma_{E,i+1}=\tilde m_{(n+1)\ast}\R_+,\]
\[\sigma_{E,i,i+1}=\tilde m_{(n+1)\ast}\R_++\tilde m_{(n+2)\ast}\R_+.\]
The monomial exponents from the dual cone are 
\begin{align*}
    \sigma_{E,i}^\vee\cap\Z^{n+2}&=m_{\ast 1}\Z+\dots+m_{\ast n}\Z+m_{\ast (n+1)}\Z+m_{\ast (n+2)}\Z_+,\\
    \sigma_{E,i+1}^\vee\cap\Z^{n+2}&=m_{\ast 1}\Z+\dots+m_{\ast n}\Z+m_{\ast (n+1)}\Z_++m_{\ast (n+2)}\Z,\\
    \sigma_{E,i,i+1}^\vee\cap\Z^{n+2}&=m_{\ast 1}\Z+\dots+m_{\ast n}\Z+m_{\ast (n+1)}\Z_++m_{\ast (n+2)}\Z_+.
\end{align*}
The toric scheme
\[T_\Sigma=\bigcup_{\sigma\in\Sigma}T_\sigma,\qquad T_\sigma=\Spec O_K[\sigma^\vee\cap\Z^{n+2}],\]
associated with $\Sigma$ (\cite{K2MS}) is then obtained by glueing $T_{\sigma_{E,i,i+1}}=\Spec R$ for varying $E$ and $i$, along their common opens. Note that
\[T_{\sigma_0}=\Spec R[Y^{-1},Z^{-1}],\quad T_{\sigma_{E,i}}=\Spec R[Y^{-1}], \quad T_{\sigma_{E,i+1}}=\Spec R[Z^{-1}].\]

Note that $\deg \psi_{\tau(1)}=1$ by Lemmas \ref{lem:MaximalClusterDegree} and \ref{lem:clusterchainuniqueness}. Let 
\[C_0=\Spec\frac{K[x][x_1^{\pm 1},\dots,x_n^{\pm 1},y^{\pm 1}]}{(y^2-f(x),x_1-\psi_1(x),\dots,x_n-\psi_n(x))}.\]
Then $C_0\subseteq C$. Furthermore it canonically embeds in $T_{\sigma_0}$ via the isomorphism given by $M_{E,i}$ and the isomorphism given by \[\frac{K[x]\big[x_{\tau(1)}^{\pm1}\big]}{\big(x_{\tau(1)}-\psi_{\tau(1)}(x)\big)}\simeq K\big[x_{\tau(1)}^{\pm 1}\big].\]
We define $\m C$ as the closure of $C_0$ in $T_\Sigma$. Then $\m C$ is integral and also separated since so is $T_\Sigma$. Furthermore, $\m C$ is flat by \cite[Corollary 3.10]{Liu}. We will explicitly describe $\m C$ and show it is a proper regular model of $C$ with strict normal crossing.



\subsection{Charts}\label{subsec:Charts}
Keep the notation of \S\ref{subsec:Matrices}. From now on we suppose without loss of generality that the permutation $\tau$ is the identity.

Let $1\leq o\leq h$. By \cite[Theorem 16.1]{Mac} every polynomial $g\in K[x]$ can be uniquely written as a sum
\[g=\sum_s a_s\cdot \psi_1^{n_{1s}}\cdots\psi_o^{n_{os}},\]
where $a_s\in K$ and $n_{js}< \deg \psi_{j+1}/ \deg \psi_{j}$ for any $j<o$. Let $u_s\in O_K^\times$ such that $a_s=u_s\cdot\pi^{v_K(a_s)}$. Then we denote by $g^{(o)}$, the polynomial
\[g^{(o)}=\sum_s u_s\cdot \pi^{v_K(a_s)}\cdot x_1^{n_{1s}}\cdots x_o^{n_{os}}\in K[x_1,\dots,x_o].\]

Consider $M_{E,i}$. 
Recall $\tilde m_{(n+1)(n+2)},\tilde m_{(n+2)(n+2)}\geq 0$. Define 
\[\Pi(X_1,\dots,X_n,Y,Z)=\pi-X_1^{\tilde m_{1(n+2)}}\dots X_n^{\tilde m_{n(n+2)}}Y^{\tilde m_{(n+1)(n+2)}}Z^{\tilde m_{(n+2)(n+2)}},\]
Via $M_{E,i}$ we have the following isomorphism
\[\frac{K[x][x_1^{\pm 1},\dots,x_n^{\pm 1},y^{\pm 1}]}{(y^2-f(x),x_1-\psi_1(x),\dots,x_n-\psi_n(x))}\stackrel{M_{E,i}}{\simeq}\frac{O_K[X_1^{\pm 1},\dots, X_n^{\pm 1},Y^{\pm 1}, Z^{\pm 1}]}{(\Pi,\mathcal{F}_1,\dots,\mathcal{F}_n)},\]
where $\mathcal{F}_1,\dots,\mathcal{F}_n\in O_K^\times[X_1^{\pm 1},\dots,X_n^{\pm 1},Y,Z]$ satisfying $Y\nmid \mathcal{F}_j$, $Z\nmid \mathcal{F}_j$, and
\footnotesize\begin{align*}
    y^2-f^{(h)}(x_1,\dots,x_h)&\stackrel{M_{E,i}}{=}Y^{n_{Y,1}}Z^{n_{Z,1}}\mathcal{F}_1(X_1,\dots,X_n,Y,Z),\\
    x_j-\psi_j^{(j-1)}(x_1,\dots,x_{j-1})&\stackrel{M_{E,i}}{=}Y^{n_{Y,j}}Z^{n_{Z,j}}\mathcal{F}_j(X_1,\dots,X_n,Y,Z)\quad\text{for }2\leq j\leq h,\\
    x_j-\psi_j^{(h)}(x_1,\dots,x_h)&\stackrel{M_{E,i}}{=}Y^{n_{Y,j}}Z^{n_{Z,j}}\mathcal{F}_j(X_1,\dots,X_n,Y,Z)\quad\text{for }h< j\leq n,
\end{align*}\normalsize
for some $n_{Y,j},n_{Z,j}\in\Z$.
Then we define the affine $O_K$-scheme
\[U_{E,i}=\Spec \frac{O_K[X_1^{\pm 1},\dots, X_n^{\pm 1},Y, Z]}{(\Pi,\mathcal{F}_1,\dots,\mathcal{F}_n)}.\]
In the next lemma we will describe the special fibre of $U_{E,i}$. In particular, we will show that it has dimension $1$. Then the next lemma implies that $U_{E,i}=\mathcal{C}\cap T_{\sigma_{E,i,i+1}}$.

\begin{lem}
If the special fibre of $U_{E,i}$ is of dimension $\leq 1$, then $U_{E,i}=\mathcal{C}\cap T_{\sigma_{E,i,i+1}}$.
\proof
By construction the generic fibre of $U_{E,i}$ is isomorphic to $C_\eta\cap T_{\sigma_{E,i,i+1}}$. Then it suffices to show that $U_{E,i}$ is the closure of its generic fibre in $T_{\sigma_{E,i,i+1}}$. Suppose not. Then $U_{E,i}$ has an irreducible component $U$ entirely contained in its special fibre. Since $O_K[X_1^{\pm 1},\dots, X_n^{\pm 1},Y,Z]$ is regular, $\dim U\geq 2$ by Krull's height theorem.
\endproof
\end{lem}

\subsection{Special Fibre}
In this section we want to study the special fibre of $U_{E,i}$. Now, $U_{E,i}\subset T_{\sigma_{E,i,i+1}}$ and the special fibre of the latter has underlying reduced subscheme $Z=0$ if $E$ is outer and $i=r$, or $YZ=0$ otherwise.
\begin{nt}
Let $g\in K[x_1^{\pm 1},\dots,x_n^{\pm 1},y^{\pm 1}]$. Consider the Laurent polynomial $\m G\in O_K^\times[X_1^{\pm 1},\dots, X_n^{\pm 1},Y^{\pm 1},Z^{\pm 1}]$ given by
\[g\big((X_1,\dots,X_n,Y,Z)\bullet M_{E,i}^{-1}\big)=\m G(X_1,\dots,X_n,Y,Z).\] 
Denote by $\ord_Z(g)$ [resp.\ $\ord_Y(g)$] the integer $\ord_Z(\m G)$ [resp.\ $\ord_Y(\m G)$].
\end{nt}

We want to study $U_{E,i}\cap\{Z=0\}$. Let $w_{E,i}:K[x]\rightarrow \hat\Q$ be the valuation given in (\ref{eqn:wEi}). Then $\ord_Z(x_j)=w_{E,i}(\psi_j)\ord_Z(\pi)$ for all $1\leq j\leq n$. Let $w_j=v_j$ for all $j<h$ and $w_h=w_{E,i}$.
\begin{lem}\label{lem:ordZg(j)}
Let $g\in K[x]$. For all $1\leq j\leq h$,
\begin{equation*}
\ord_Z\big(g^{(j)}\big)=w_j(g)\ord_Z(\pi)
\end{equation*}
\proof
If $w_{E,i}$ is MacLane then the equality follows from \cite[Theorem 16.1]{Mac}. Suppose $w_{E,i}$ is not MacLane. Then $(\s,v)$ is maximal, $E=V_v$ and $1\leq i\leq r$. But then $h=1$ and $\deg \psi_1=1$. Expand $g=\sum_{\i}a_\i\psi_1^\i$, where $a_\i\in K$. Then $g^{(1)}=\sum_{\i}a_\i x_1^\i$. It follows that
\[\ord_Z\big(g^{(1)}\big)=\min_{\i}\big(v_K(a_\i)\ord_Z(\pi)+\i\cdot \ord_Z(x_1)\big)=
w_{E,i}(g)\ord_Z(\pi)\]
as $\ord_Z(x_1)=w_{E,i}(\psi_1)\ord_Z(\pi)$.
\endproof
\end{lem}



\begin{nt}
For any $\m G\in O_K[X_1^{\pm 1},\dots, X_n^{\pm 1},Y,Z]$ denote
\[\bar{\m G}_Y=\m G(X_1,\dots,X_n,0,Z),\qquad\bar{\m G}_Z=\m G(X_1,\dots,X_n,Y,0),\]
and $\bar{\m G}=\m G(X_1,\dots,X_n,0,0)$.
\end{nt}


\begin{defn}
Define $p_0=\pi\in O_K$. Let $1\leq j\leq h$ and recursively define $p_j\in K[x_1^{\pm 1},\dots,x_j^{\pm 1}]$ by
$p_j=x_j^{\ell_j}p_{j-1}^{\ell_j'}$.
Then $p_j(\psi_1,\dots,\psi_j)=\pi_j$.



Define $\Pi_j\in O_K[X_1^{\pm 1},\dots, X_n^{\pm 1}]$ by
\[Y^{\ast}Z^{\ast}\cdot \Pi_{j}\stackrel{M_{E,i}}{=}p_j.\]
\end{defn}




Note that
\begin{equation}\label{eqn:MEiChangepsigammaGeneral}
(\psi_1,\dots,\psi_n,y,\pi)\bullet M_{E,i}=(\gamma_1,\dots,\gamma_{h-1},\psi_h^{\delta a_x}y^{\delta a_y}\pi_{h-1}^{\delta a_z},\dots),
\end{equation}
and that $\tilde \alpha_{hj}=0$, $\tilde\alpha_{\pi j}=\lambda_j$ for any $j<h$.

\begin{lem}\label{lem:Xj}
Let $1\leq j\leq h$. 
Then $X_j\stackrel{M_{E,i}}{=}x_j^{e_j}p_{j-1}^{-h_j}$ if $j<h$ or $E=L_v$.
\proof
When $j<h$, then $X_j=x_j^{e_j}p_{j-1}^{-h_j}$ from (\ref{eqn:MEiChangepsigammaGeneral}). If $j=h$, then $X_j=x_j^{\delta a_x}y^{\delta a_y}p_{j-1}^{\delta a_z}$. If $E=L_v$, then $w_{E,i}=v$. Since $L_v$ corresponds to the edge $L_v(f)$ of $N_{v_-,\psi_h}^-(f)$, one has $\delta=e_h$, $a_x=1$, $a_y=0$, $a_z=-\lambda_h$. It follows that $X_j=x_j^{e_j}p_{j-1}^{-h_j}$, as required.
\endproof
\end{lem}

\begin{lem}\label{lem:Pi}
Let $1\leq j\leq h$. Then $\Pi_j\in O_K[X_1^{\pm 1},\dots,X_h^{\pm 1}]$.
\proof
If $o>h$ then $\tilde m_{jo}=0$ for $j\neq o$. The lemma follows.
\endproof
\end{lem}


Recall the definition of the fields $k_j$, $j=1,\dots,h$, given in \S\ref{sec:Residualpoly}. Note that $k_1=k_0$ since $\deg\psi_1=1$ (Remark \ref{rem:ki}). The ring homomorphisms $k_o[X_o^{\pm 1}]\rightarrow k_{o+1}$, $1\leq o<j$, taking $X_o$ to the generator of $k_{o+1}$ over $k_o$, induce a surjective homomorphism
\begin{equation*}
\m R_{j}:O_K[X_1^{\pm 1},\dots, X_{n}^{\pm 1},Y,Z]\rightarrow k_{j}[X_j^{\pm 1},\dots, X_n^{\pm 1},Y,Z].
\end{equation*}

\begin{lem}\label{lem:mGreduction}
Let $1\leq j \leq h$
       and let $g\in K[x]$. Fix a polynomial $\mathcal{G}\in O_K^\times[X_1^{\pm 1},\dots,X_n^{\pm 1},Y,Z]$, $Y\nmid \mathcal{G}$, $Z\nmid \mathcal{G}$, such that
    \[g^{(j)}(x_1,\dots,x_{j})\stackrel{M_{E,i}}{=}Y^{n_Y}Z^{n_Z}\mathcal{G}(X_1,\dots,X_n,Y,Z),\]
    for some $n_Y,n_Z\in\Z$. If either $E=L_v$ or $j<h$, then 
    \[\m R_{j}(\bar{\m G}_Z)=Y^\ast\cdot \m R_j(\Pi_j)^{e_{v_j}\alpha}\cdot H_{j, \alpha}(g)(X_{j})\]
    where $\alpha=v_j(g)$.
    
\proof
We prove the lemma by induction on $j$. Suppose either $E=L_v$ or $j<h$, so that $w_j=v_j$. Let $j=1$. 
Expand $g=\sum_{s}a_s\phi_1^j$, where $a_s\in K$.
Then $g^{(1)}(x_1)=\sum_sa_sx_1^s$. Lemma \ref{lem:ordZg(j)} implies that $\ord_Z(a_sx_1^s)=n_{Z}$ if and only if $(s,v_K(a_s))$ is a point of the edge $L_{v_1}(g)$ of the Newton polygon $N_{v_0,\psi_1}(g)$. 
Therefore we can assume $g^{(1)}(x_1)=\sum_{s\geq 0}a_{\i_1+se_1}x_1^{\i_1+se_1}$, where $\i_1=\i_1(\alpha_1)$ and $\alpha_1=v_1(g)$. Then
\[\frac{g^{(1)}}{p_1^{e_{v_1}\alpha_1}}=\sum_{s\geq 0}\lb\frac{a_{\i_1+se_1}}{\pi^{u_1-sh_1}}\rb (x_1^{e_1}\pi^{-h_1})^{s+c_1(\alpha_1)},\]
where $u_1=u_1(\alpha_1)$.
Then we obtain the required equality by Lemma \ref{lem:Xj}.



Now suppose $j>1$. 
Expand 
\[g=\sum_{s\geq 0}a_s \psi_j^s,\qquad\text{where }\deg a_s<\deg \psi_j.\]
Note that $g^{(j)}=\sum_{s}a_s^{(j-1)} x_j^s$ by definition.
Similarly to before, Lemma \ref{lem:ordZg(j)} implies that $\ord_Z\big(a_s^{(j-1)}x_j^s\big)=n_{Z}$ if and only if $(s,v_{j-1}(a_s))$ is a point of the edge $L_{v_j}(g)$ of the Newton polygon $N_{v_{j-1},\psi_j}(g)$. 
Therefore we can assume \[g^{(j)}=\textstyle\sum_{s}a_{\i_{j,s}}^{(j-1)}x_j^{\i_{j,s}},\] 
where $\i_{j,s}=\i_j(\alpha_j)+se_j$ and $\alpha_j=v_j(g)$. Then
\[\frac{g^{(j)}}{p_j^{e_{v_j}\alpha_j}}=\sum_{s}\frac{a_{\i_{j,s}}^{(j-1)}}{p_{j-1}^{u_{j,s}}}\lb x_j^{e_j}p_{j-1}^{-h_j}\rb^{s+c_j(\alpha_j)},\]
where $u_{j,s}=u_j(\alpha_j)-sh_j$. Lemma \ref{lem:Xj} and the inductive hypothesis conclude the proof.
\endproof
\end{lem}

\begin{lem}\label{lem:mRkernel}
Let $1\leq j\leq h$. Then $\ker(\m R_j)=(\bar{\m F}_{2,Z},\dots,\bar{\m F}_{j,Z},\pi)$.
\proof
We prove the lemma by induction on $j$. Suppose $j=1$. Since 
$\deg\psi_1=1$, we have $k_1=k$, and so $\ker(\m R_1)=(\pi)$. Let $j>1$. It follows from Lemma \ref{lem:ordZg(j)} that \[\ord_Z(x_j)=\ord_Z\big(\psi_j^{(j)}\big)>\ord_Z\big(\psi_j^{(j-1)}\big).\] 
Then Lemma \ref{lem:mGreduction} implies that 
\[\m R_{j-1}(\bar{\m F}_{j,Z})=\m R_{j-1}(\Pi_{j-1})^{e_{v_{j-1}}\alpha}\cdot H_{j-1,\alpha}(\psi_j),\]
where $\alpha=v_{j-1}(\psi_j)$. Since $k_j\simeq k_{j-1}[X_{j-1}]/(H_{j-1,\alpha}(\psi_j))$ by Remark \ref{rem:ki} and $\m R_{j-1}(\Pi_{j-1})$ is invertible by Lemma \ref{lem:Pi}, we have 
\[\ker(\m R_j)=\ker(\m R_{j-1})+(\bar{\m F}_{j,Z}).\] 
The inductive hypothesis concludes the proof.
\endproof
\end{lem}


Let $h<j\leq n$. Then \[\ord_Z(x_j)=\ord_Z\big(\psi_j^{(h)}\big)\] by Lemma \ref{lem:ordZg(j)}. Since $\tilde m_{jj}=1$ and $\tilde m_{oj}=0$ for all $1\leq o\leq n$, $o\neq j$, there exists $\m T_j\in O_K[X_1^{\pm 1},\dots,X_h^{\pm 1},Y,Z]$ such that $\m F_j$ equals $X_j-\m T_j$ up to some unit.
Let $\m R=\m R_h$ and $\m T=\prod_{h<j\leq n}\m T_j$. Denote $\m F=\m F_1$.
Lemma \ref{lem:mRkernel} implies that $U_{E,i}\cap \{Z=0\}$ is isomorphic to
\[\Spec \frac{k_v[X_h^{\pm 1},Y,\m R(\bar{\m T}_Z)^{-1}]}{(\m R(\bar{\m F}_Z))}.\]
Similar computations (using $w_{E,i+1}$ instead of $w_{E,i}$) show that if $i<r$ or $E$ is inner, then $U_{E,i}\cap \{Y=0\}$ is isomorphic to
\[\Spec \frac{k_v[X_h^{\pm 1},Z,\m R(\bar{\m T}_Y)^{-1}]}{(\m R(\bar{\m F}_Y))}.\]
Let $g(x,y)=y^2-f(x)$ and expand
\[g=\sum_{j,o}a_{jo}\psi_h^jy^o,\qquad a_{jo}\in K[x],\,\deg a_{jo}<\deg \psi_h.\]
Then $y^2-f^{(h)}=\sum_{j,o}a_{jo}^{(h-1)}x_h^jy^o$. Recall the notation $w_{E,i}(y)$ from Appendix \ref{sec:appendix:fan}. Let $\xi_i$ be the plane with normal vector $(w_{E,i}(\psi_h),w_{E,i}(y),1)$ 
and on which $\tilde E$ lies. We have 
\[\ord_Z(a_{jo}^{(h-1)}x_h^jy^o)=\ord_Z(y^2-f^{(h)})\quad\text{if and only if}\quad (j,o,v_-(a_{jo}))\in\xi_i.\]
More precisely, $(X_h,Y,Z)=(x_h,y,p_{h-1})\bullet M$ with
\begin{equation}\label{eqn:3dimMatrix}
M=\begin{psmallmatrix}
\!\delta a_x&d_ib_x+k_ia_x&-d_{i+1}b_x-k_{i+1}a_x\!\\
\!\delta a_y&d_ib_y+k_ia_y&-d_{i+1}b_y-k_{i+1}a_y\!\\
\!\delta \epsilon_ha_z&\frac{n_i}{\delta}+\epsilon_hk_ia_z&-\frac{n_{i+1}}{\delta}-\epsilon_hk_{i+1}a_z\!
\end{psmallmatrix}\in\SL_3(\Z).
\end{equation}
Let $\phi:\Z^2\rightarrow \Z^2$ given by 
\[\phi(s,t)=P_0+(\delta a_x,\delta a_y)s+(d_ib_x+k_ia_x,d_ib_y+k_ia_y)t.\]
Lemma \ref{lem:mGreduction} implies that, up to units, $\m R(\bar{\m F}_Z)$ equals
\[\sum_{(s,t)\in\Z^2}H_{h-1,\alpha_{s,t}}\big(a_{\phi(s,t)}\big)X_h^sY^t,\]
where $\alpha_{s,t}=\delta a_z s+\frac{n_i}{\delta \epsilon_h}t+k_ia_zt$. In particular, $U_{E,i}\cap \{Z=0\}$ is of dimension $1$, and similarly $U_{E,i}\cap \{Y=0\}$ when $i<r$ or $E$ is inner. It follows that the special fibre of $U_{E,i}$ is $1$-dimensional.

\subsection{Components}\label{subsec:Components}

We want to describe $U_{E,i}\cap \{Z=0\}$ and $U_{E,i}\cap \{Y=0\}$ explicitly.

\begin{rem}\label{rem:resydualcentres}
Let $h<j\leq n$ such that $v=\mu_j\wedge\mu_h$. Let $(\s_j,\mu_j)\subseteq(\t,w)<(\s,v)$. Lemma \ref{lem:minimalityequivalence} implies that $\psi_j$ is $v$-equivalent to $\phi_w^d$, where $d=\deg\mu_j/\deg w$. Thus $\psi_j|_v$ is a power of $\phi_w|_v$ by Proposition \ref{prop:v-equivalentResidual}.
\end{rem}

\begin{lem}\label{lem:psijunit}
Let $h<j\leq n$. If $v\neq \mu_h\wedge\mu_j$ then $\psi_j|_v$ is a unit. 
\proof
The lemma follows from Lemma \ref{lem:phiunit}.
\endproof
\end{lem}

\begin{lem}\label{lem:mTj}
Let $h<j\leq n$.
\begin{enumerate}
\item Suppose $E=L_v$. 
Then, up to units, 
$\m R(\bar{\m T}_{j,Z})$ equals $\psi_j|_v(X_h)$,
and, similarly, $\m R(\bar{\m T}_{j,Y})$ equals $\psi_j|_v(X_h)$ when $i<r$.
\item Suppose $E=V_v$ or $E=V_v^0$. Then $\bar{\m T}_{j}$ is a unit. Furthermore, $\bar{\m T}_{j}=\bar{\m T}_{j,Z}$ if $i>0$ and $\bar{\m T}_{j}=\bar{\m T}_{j,Y}$ if $i<r$.
\end{enumerate}
\proof

Suppose $E=L_v$. Then Lemma \ref{lem:mGreduction} implies that $\m R(\bar{\m T}_{j,Z})$ equals $\psi_j|_v(X_h)$ up to units. Similarly for $\m R(\bar{\m T}_{j,Y})$ when $i<r$.

Expand 
\[\psi_j=\sum_{\i=0}^d a_\i\psi_h^\i,\qquad a_\i\in K[x],\, a_d\neq 0,\,\deg a_\i<\deg \psi_h.\]
Then $\psi_j^{(h)}=\sum_\i a_\i^{(h-1)}x_h^\i$. 

Suppose $(\s,v)$ maximal and $E=V_v$. 
Then $h=1$ and $\deg \psi_h=1$. Lemma \ref{lem:vhc} implies that $w_{E,i}(\psi_j)=d\cdot w_{E,i}(\psi_1)$ for any $i=0,\dots,r$. In fact, for all $i=1,\dots,r$ we have
\begin{equation}\label{eqn:wEipsij}
w_{E,i}(\psi_j-\psi_1^d)>w_{E,i}(\psi_j)
\end{equation}
since $w_{E,i}(\psi_h)<w_{E,0}(\psi_h)$. Recall \[\ord_Z(\psi_j^{(1)})=w_{E,i}(\psi_j)\ord_Z(\pi)\]
from Lemma \ref{lem:ordZg(j)}, and similarly, $\ord_Y(\psi_j^{(1)})=w_{E,i+1}(\psi_j)\ord_Y(\pi)$ when $i<r$. The inequality (\ref{eqn:wEipsij} implies that $\bar{\m T}_{j}$ is a unit, $\bar{\m T}_{j}=\bar{\m T}_{j,Z}$ when $i>0$ and $\bar{\m T}_{j}=\bar{\m T}_{j,Y}$ when $i<r$.

Suppose $E=V_v$ inner. Then $w_{E,i}$ is a MacLane valuation with centre $\psi_h$ and satisfying $v\geq w_{E,i}\geq w$. In particular, $w_{E,i}(\psi_j)=w_{E,i}(a_{\i_v}\psi_h^{\i_v})$. Lemma \ref{lem:ordZg(j)} implies that $\bar{\m T}_{j,Z}$ is a unit if and only if $\psi_j|_{w_{E,i}}$ is a unit. But then $\bar{\m T}_{j,Z}=\bar{\m T}_{j}$ is a unit when $i>0$ by Lemma \ref{lem:phiunit}. Similarly $\bar{\m T}_{j,Y}=\bar{\m T}_{j}$ is a unit when $i<r$.

Suppose $(\s,v)$ degree-minimal and $E=V_v^0$. Then $w_{E,i}$ is a MacLane valuation with centre $\psi_h$ and satisfying $v\leq w_{E,i}$. In particular, $w_{E,i}(\psi_j)=w_{E,i}(a_{\io_v}\psi_h^{\io_v})$.
Similarly to the previous case, Lemmas \ref{lem:ordZg(j)}, \ref{lem:phiunit} conclude the proof.
\endproof
\end{lem}

Suppose $E=L_v$. Fix $P_0=(\i_v,0)$, $P_1=\big(\left\lfloor\tfrac{\i_v-1}{2}\right\rfloor,1\big)$. Then
\begin{equation}\label{eqn:s1L}
    s_1^E=e_v\big(\lambda_v(\lfloor\i_v/2\rfloor+1)-\tfrac{v(f)}{2}\big),
\end{equation}
and $s_2^E=\lfloor s_1^E-1\rfloor$.
The $h$-edge $L_v$ corresponds to the edge $L_v(f)$ of $N_{v_-,\psi_h}(f)$. In particular, $\delta_E=e_h$ and $\nu=(1,0,-\lambda_h)$. It follows that $\m X_h=X_h$. 
Therefore, up to units,
\begin{align*}
\m R(\bar{\m F}_Z)&=f|_v(X_h)\quad\text{for }0<i\leq r,\\
\m R(\bar{\m F}_Y)&=f|_v(X_h)\quad\text{for }0\leq i<r.
\end{align*}
Fix 
\[k_j=\ell_hn_j+\ell_h'e_h d_j(\lfloor\i_v/2\rfloor+1),\quad \text{for }j=0,\dots,r+1.\]
Then $k_j\equiv n_j(\delta_E \epsilon_ha_z)^{-1}\mod\delta_E$, as required. 
Let $i=0$ and let $M$ be the matrix of (\ref{eqn:3dimMatrix}). Then 
\[
M^{-1}=\begin{psmallmatrix}
\ell_h' & 0 & -\ell_h\\
d_1e_v\lambda_v & d_1e_v\frac{v(f)}{2}+\frac{1}{d_0} & d_1e_h\\
d_0e_v\lambda_v & d_0e_v\frac{v(f)}{2} & d_0e_h
\end{psmallmatrix}.
\]
Hence $y^2p_h^{-e_vv(f)}=Y^{2/d_0}$. Lemma \ref{lem:mGreduction} then implies that $R_h(\bar{\m F}_Z)$ equals 
\[Y^{2/d_0}-H_{h,v(f)}(f)(X_h)\] 
up to units. The quantity $n_v:=2/d_0$ equals $1$ if $e_vv(f)$ is odd and $2$ if $e_vv(f)$ is even. Recall 
\[H_{h,v(f)}(f)(X)=X^{\io_v/e_h-\ell_h \epsilon_hv(f)}f|_v,\] 
from (\ref{eqn:Hvto|v}). Note that $\io_v=\i_w$ if $(\s,v)$ has a child $(\t,w)$ with centre $\psi_h$ and $\io_v=0$ otherwise.

Suppose $E=V_v$. We can choose $P_0=(\i_v,0)$, $P_1=\big(\left\lfloor\tfrac{\i_v-1}{2}\right\rfloor,1\big)$ so that 
\begin{equation*}
    s_1^E=\delta_E\epsilon_h\big(\lambda_v(\lfloor\i_v/2\rfloor+1)-\tfrac{v(f)}{2}\big),
\end{equation*}
If $(\s,v)\neq(\roots,w_\roots)$ and $(\s,v)<(\t,w)$, then 
\begin{equation*}
    s_2^E=\delta_E\epsilon_h\big(\lambda_{w}(\lfloor\i_v/2\rfloor+1)-\tfrac{w(f)}{2}\big),
\end{equation*}
while $s_2^E=\lfloor s_1^E-1\rfloor$ otherwise.
Up to units
\begin{align*}
\m R(\bar{\m F}_Z)&=X_h^b-H_{h-1,\alpha}(a_{\i_v})\quad\text{for }0<i\leq r,\\
\m R(\bar{\m F}_Y)&=X_h^b-H_{h-1,\alpha}(a_{\i_v})\quad\text{for }0\leq i<r,
\end{align*}
where $b=E_{\Z}(\Z)+1$ and $\alpha=v_-(a_{\i_v})$. Let $u=c_f\prod_{r'\notin\s}(x-r')\in K[x]$ and let $u_h=u-\psi_hq$ for some $q\in K[x]$ such that $\deg u_h<\deg\psi_h$. From Theorem \ref{thm:Newtondecomposition}, one has $H_{h-1,\alpha}(a_{\i_v})=H_{h-1,v_-(u_h)}(u_h)$.

Suppose $v=\mu_h$ and $E=V_{v}^0$. Fix $P_0=(0,2)$, $P_1=(1,1)$, so
\begin{equation*}
    s_1^E=-\delta_E\epsilon_h\big(\lambda_v-\tfrac{v(f)}{2}\big),
\end{equation*}
and $s_2^E=\lfloor s_1^E-1\rfloor$.
Then up to units
\begin{align*}
\m R(\bar{\m F}_Z)&=X_h^{-b}-H_{h-1,\alpha}(a_{\io_v})\quad\text{for }0<i\leq r,\\
\m R(\bar{\m F}_Y)&=X_h^{-b}-H_{h-1,\alpha}(a_{\io_v})\quad\text{for }0\leq i<r,
\end{align*}
where $b=E_{\Z}(\Z)+1$ and $\alpha=v_-(a_{\io_v})$. Let $\roots_h$ be the set of roots of $\psi_h$. Let $u^0=c_f\prod_{r'\in\roots\setminus\roots_h}(x-r')\in K[x]$ and let $u_h^0=u^0-\psi_hq$ for some $q\in K[x]$ such that $\deg u_h^0<\deg\psi_h$. One has $H_{h-1,\alpha}(a_{\io_v})=H_{h-1,v_-(u_h^0)}(u_h^0)$.


\subsection{Regularity}
If $(\s,v)$ has a proper child with centre $\phi_w\neq\phi_v$, then $\phi_w|_v$ is irreducible by Lemmas \ref{lem:clusterchainuniqueness} and \ref{lem:keypolyreduction}. Let $E=L_v$. By Remark \ref{rem:resydualcentres} and Lemmas \ref{lem:psijunit}, \ref{lem:mTj}, the subscheme $U_{E,i}\cap\{Z=0\}$ is isomorphic to
\begin{equation}\label{eqn:mathringGammav}
\Spec \frac{k_v\big[X_h^{\pm 1},Y,\prod_{(\t,w)<(\s,v)}(\phi_w|_v(X_h))^{-1}\big]}{(\m R(\bar{\m F}_Z))},
\end{equation}
where the product runs through all proper children of $(\s,v)$.
Similarly for $U_{E,i}\cap\{Y=0\}$ when $i<r$. 

\begin{nt}\label{nt:mathringGammav}
We denote by $\mathring{\Gamma}_v$ the scheme $U_{L_v,0}\cap\{Z=0\}$.
\end{nt}

\begin{thm}\label{thm:regularmodel}
The model $\mathcal{C}/O_K$ is regular.
\proof
We want to prove that $U_{E,i}$ is regular, for any $h$-edge $E$, $h=1,\dots,n$, and any $i=0,\dots,r_E$. In fact, for the definition of $\Pi$, it suffices to show that the subschemes $U_{E,i}\cap\{Z=0\}$ and $U_{E,i}\cap\{Y=0\}$ are regular, where the latter is considered only if $i<r$. From the description given in \S\ref{subsec:Components} we only need to consider the case $E=L_v$, for some proper MacLane cluster $(\s,v)$ for $f$. Let $r=r_E$. For the explicit description of $\m R(\bar{\m F}_Z)$ and $\m R(\bar{\m F}_Y)$ it suffices to prove that all multiple irreducible factors of $f|_v$ are of the form $\phi_w|_v$ for some proper child $(\t,w)$ of $(\s,v)$. But this follows from Theorem \ref{thm:ChildrenReduction}.
\endproof
\end{thm}


\subsection{Properness}
Let $\m C_s^\mathrm{red}$ be the underlying reduced subscheme of the special fibre of $\m C$. In the previous subsections we showed that $\m C_s^\mathrm{red}$ consist of $1$-dimensional subschemes $\Gamma_v$ for each proper MacLane cluster $(\s,v)$, closures of $\mathring{\Gamma}_v$ (Notation \ref{nt:mathringGammav}) in $\m C$, and chains of $\P_1$. In this subsection we will show that $\Gamma_v$ is projective for any proper $(\s,v)\in\Sigma_f^M$.
By \cite[Remark 3.28]{Liu} the properness of $\m C$ will follow.

Let $(\s,v)$ be a proper MacLane cluster and recall the notation introduced in previous subsections. Let $C_v$ be the regular projective curve with ring of rational functions
\[k_v(X)[Y]/\big(Y^{n_v}-X^{\io_v/e_h-\ell_h \epsilon_hv(f)}f|_v\big).\]
From (\ref{eqn:mathringGammav}) we have a natural birational map $\Gamma_v\--> C_v$ defined on the dense open $\mathring{\Gamma}_v$.
It extends to a morphism $\iota:\Gamma_v\rightarrow C_v$ by \cite[II.7.4.9]{EGA}. Zariski's Main Theorem implies that $\iota$ is an open immersion, since $\Gamma_v$ is separated and regular. By point counting we can prove that $\iota$ is an isomorphism.

Let $\mathring C_v=\iota(\mathring \Gamma_v)$. By Theorem \ref{thm:ChildrenReduction} we have
\[\ord_{\phi_w|_v}(f|_v)=|\t|/\deg w,\]
for any proper child $(\t,w)<(\s,v)$ with $\phi_w\neq\phi_v$.
The set $C_v(\bar k)\setminus\mathring C_v(\bar k)$ is finite and consists of:
\begin{enumerate}[label=(\arabic*)]
    \item $\gcd(n_v, \io_v/e_h-\ell_h \epsilon_hv(f)+\deg(f|_v))$ points at infinity;\label{item:(1)pointsatinfinity}
    \item $\gcd(n_v, \io_v/e_h-\ell_h \epsilon_hv(f))$ points on $X=0$;\label{item:(2)pointsatX=0}
    \item $\gcd(n_v,|\t|/\deg w)$ points on $Y=0$ ($X\neq 0$) for each proper child $(\t,w)<(\s,v)$ with $\phi_w\neq\phi_v$.\label{item:(3)pointsatY=0}
\end{enumerate}

\ref{item:(1)pointsatinfinity} Let $E=V_v$. The scheme $\Gamma_v$ has $(|E_{\Z}(\Z)|+1)$ $\bar k$-points in \[U_{E,0}\cap\{Y=Z=0\}\] not contained in $\mathring{\Gamma}_v$. Note that $|E_{\Z}(\Z)|$ equals $1$ if $\i_v$ and $\epsilon_h(v(f)-\i_v\lambda_v)$ are both even, while it equals $0$ otherwise. 
In fact,
\[\big(\epsilon_h(v(f)-\i_v\lambda_v),\i_v\big)=\big( e_vv(f), \i_v/e_h-\ell_h \epsilon_hv(f)\big) \cdot
\begin{psmallmatrix}
\ell'_h & \ell_h\\
-h_h& e_h
\end{psmallmatrix},\]
and so
\[|E_{\Z}(\Z)|+1=\gcd(n_v, \io_v/e_h-\ell_h \epsilon_hv(f)+\deg(f|_v)),\]
since $\deg(f|_v)=(\i_v-\io_v)/e_h$.

\ref{item:(2)pointsatX=0} Let $E=V_v^0$. Let $(\t,w)<(\s,v)$ such that $E=V_w$ if $(\s,v)$ is not degree-minimal. Let $U=U_{V_v^0,0}$ if $(\s,v)$ is degree-minimal and $U=U_{V_w,r_E+1}$ otherwise. The scheme $\Gamma_v$ has $(|E_{\Z}(\Z)|+1)$ $\bar k$-points in 
\[U\cap\{Y=Z=0\}\] 
not visible on $\mathring{\Gamma}_v$. Note that $|E_{\Z}(\Z)|$ equals $1$ if $\io_v$ and $\epsilon_h(v(f)-\io_v\lambda_v)$ are both even, while equals $0$ otherwise. Similarly to the case above, 
\[\big(\epsilon_h(v(f)-\io_v\lambda_v),\io_v\big)=\big( e_vv(f), \io_v/e_h-\ell_h \epsilon_hv(f)\big) \cdot
\begin{psmallmatrix}
\ell'_h & \ell_h\\
-h_h& e_h
\end{psmallmatrix},\]
and so
\[|E_{\Z}(\Z)|+1=\gcd(n_v, \io_v/e_h-\ell_h \epsilon_hv(f)).\]

\ref{item:(3)pointsatY=0} Let $(\t,w)<(\s,v)$ be a proper child such that $\phi_w\neq\phi_v$. Let $E=V_w$. The scheme $\Gamma_v$ has $(|E_{\Z}(\Z)|+1)$ $\bar k$-points in 
\[U_{E,0}\cap\{Y=Z=0\}\] 
not visible on $\mathring{\Gamma}_v$. Note that $|E_{\Z}(\Z)|$ equals $1$ if $\i_w$ and $e_v(v(f)-\i_w\lambda_v)$ are both even, while it equals $0$ otherwise. Since $\i_w=|\t|/\deg w$ by Proposition \ref{prop:EdgesRoots}, we can compute
\[|E_{\Z}(\Z)|+1=\gcd(n_v,|\t|/\deg w).\]

Thus $|\Gamma_v(\bar k)\setminus\mathring{\Gamma}_v(\bar k)|=|C_v(\bar k)\setminus\mathring{C}_v(\bar k)|$, and so $\Gamma_v\simeq C_v$.

\begin{rem}
If $k_v$ is perfect, $\Gamma_v$ is a generalised Baker's model of the curve $\mathring{\Gamma}_v\cap\G_{m,k_v}^2$ according to \cite{Mus2}.
\end{rem}

\section{Main result}\label{sec:MainResult}
Let $C/K$ be a hyperelliptic curve of genus $g\geq 1$. Choose a separable polynomial $f\in K[x]$ as in the previous section so that $C/K:y^2=f(x)$. Then $v_K(r)>0$ for every root $r\in\bar K$ of $f$. Denote by $\roots$ the set of roots of $f$ as before. Consider the MacLane cluster picture of $f$ and fix a centre $\phi_v$ for all proper MacLane clusters $(\s,v)\in\Sigma_f^M$ as we did at the beginning of \S\ref{sec:ModConstr}. Denote by $\Sigma$ the set of proper MacLane clusters for $f$.

\begin{defn}\label{defn:QuantitiesForTheoremsOnModelsDefinition}
Let $(\s,v)\in\Sigma$. 
Consider its cluster chain
\[[v_0,v_1(\phi_1)=\lambda_1,\dots,v_{m-1}(\phi_{m-1})=\lambda_{m-1},v_{m}(\phi_{m})=\lambda_{m}].\]
Define the following quantities:
\begin{center}
\begin{tabular}{|l@{$=\>\,$}l|}
\hline
    $\epsilon_v$          & $e_{v_{m-1}}$\cr
    $b_v$       & $e_v/\epsilon_v$\cr
    $\ell_v$    & $\ell_m$\cr
    $k_v$       & $k_m$\cr
    $f_v$       & $[k_v:k]$\cr
    $\nu_v$     & $v(f)$\cr
    $n_v$            & $1$ if $e_v\nu_v$ odd, $2$ if $e_v\nu_v$ even\cr
    $m_v$            & $2e_v/n_v$\cr
    $p_v$            & $1$ if $\i_v$ is odd, $2$ if $\i_v$ is even\cr
    $s_v$            & $\frac 12(\i_v\lambda_v+p_v\lambda_v-\nu_v)$\cr
    $\gamma_v$       & $2$ if $\i_v$ is even and $\epsilon_v(\nu_v\!-\!\i_v\lambda_v)$ is odd, $1$ otherwise\cr
    $\delta_v$          & $1$ if $(\s,v)$ is degree-minimal, $0$ otherwise\cr
    $p_v^0$          & $1$ if $\delta_v=1$ and $\deg v=\min_{r\in\s}[K(r):K]$, $2$ otherwise\cr
    $s_v^0$          & $-\nu_v/2+\lambda_v$\cr
    $\gamma_v^0$     & $2$ if $p_v^0=2$ and $\epsilon_v\nu_v$ is an odd integer, $1$ otherwise\cr
\hline
\end{tabular} 
\end{center}

Define
\[\tilde v=\big\{(\t,w)\in\Sigma\mid (\t,w)<(\s,v)\text{ and }\tfrac{f_v|\t|}{f_wb_v\deg v}-\ell_v\nu_v\epsilon_w\notin 2\Z\big\}.\]
Let $c_v^0=1$ if 
$\tfrac{2-p_v^0}{b_v}-\ell_v\nu_v\epsilon_v\notin 2\Z$, and $c_v^0=0$ otherwise. Define
\[u_v=\smaller{\dfrac{|\s|-\sum_{(\t,w)<(\s,v)}|\t|-(2-p_v^0)\deg v}{e_v}+\sum_{(\t,w)\in\tilde v}\frac{f_w}{f_v}}+ \delta_vc_v^0.\]
The \textit{genus $g(v)$} of $(\s,v)$ is defined as follows:
\begin{itemize}
    \item if $n_v=1$, then $g(v)=0$;
    \item if $n_v=2$, then $g(v)=\max\{\lfloor(u_v-1)/2\rfloor,0\}$.

\end{itemize}
We say that $(\s,v)$ is \textit{\"{u}bereven} if $u_v=0$.

Recall the definition of $H_{m-1,\alpha}$, for $\alpha\in\Gamma_{v_{m-1}}$, from Definition \ref{defn:Loc(i)H(ii)}(ii). Define $\overline{g_v}\in k_v[y]$, and $\overline{g_v^0}\in k_v[y]$ if $\delta_v=1$, by
\begin{align*}
  \overline{g_v}(y) &=  y^{p_v/\gamma_v} - H_{m-1,v_{m-1}(u)}(u), &  u&=c_f\tprod_{r\in\roots\setminus\s} (x-r) \mod \phi_v,\\
  \overline{g_v^0}(y) &=  y^{p_v^0/\gamma_v^0} - H_{m-1,v_{m-1}(u)}(u), & u&=c_f\tprod_{r\in\roots\setminus\roots_{v}} (x-r) \mod \phi_v,
\end{align*}
where $\roots_{v}$ is the set of roots of $\phi_v$. 

Define $f'_v\in K[x]$ by
\[
  \phi_v^{2-p_v^0}f'_v(x) = \prod_{r\in\s\setminus\bigcup_{(\t,w)<(\s,v)}\t} (x-r),
\]
where the union runs through all proper children $(\t,w)<(\s,v)$.

Define $\overline{f_v}, \tilde f_v\in k_v[x]$ by
\begin{align*}
\overline{f_v}(x)&=H_{m-1,v_{m-1}(u)}(u)\cdot f'_v|_v(x),&u&=c_f\tprod_{r\in\roots\setminus\s} (x-r)\mod \phi_v,\\
\tilde{f}_v(x)&=\ch{f_v}(x)\cdot x^{\delta_vc_v^0}\cdot\tprod_{(\t,w)\in\tilde v}\phi_w|_v(x).&&
\end{align*}

Finally, define the $k_v$-schemes
\begin{itemize}
    \item $X_v:\{\ch{f_v}=0\}\subset\G_{m,k_v}$;
    \item $Y_v:\{\ch{g_v}=0\}\subset\G_{m,k_v}$;
    \item $Y_v^0:\{\ch{g_v^0}=0\}\subset\G_{m,k_v}$ if $(\s,v)$ is degree-minimal.
\end{itemize}
\end{defn}


\begin{nt}
If $C=C_1\cup\dots C_r$ is a chain of $\P^1_{k}$s (meeting transversely) of length $r$, then $\infty\in C_i$ is identified with $0\in C_{i+1}$, and $0,\infty\in C$ are respectively $0\in C_{1}$ and $\infty\in C_{r}$. If $r=0$, then $C=\Spec k$ and $0=\infty$.
\end{nt}

\begin{nt}
Let $\alpha,a,b\in\Z$, with $\alpha>0$, $a>b$, and fix $\frac{n_i}{d_i}\in\Q$ so that
    \[\alpha a=\frac{n_0}{d_0}>\frac{n_1}{d_1}>\ldots>\frac{n_r}{d_r}>\frac{n_{r+1}}{d_{r+1}}=\alpha b,\quad\text{with\scalebox{0.9}{$\quad\begin{vmatrix}n_i\!\!\!&n_{i+1}\cr d_i\!\!\!&d_{i+1}\cr\end{vmatrix}=1$}},\]
    and $r$ minimal. We write $\P^1(\alpha,a,b)$ for a chain of $\P^1_{k}$s of length $r$ and multiplicities $\alpha d_1,\dots,\alpha d_r$. Furthermore, we denote by $\P^1(\alpha,a)$ the chain $\P^1(\alpha,a,\lfloor\alpha a-1\rfloor/\alpha)$.
\end{nt}

\begin{nt}
Let $a,b\in K[x]$, $b\neq 0$. We denote by $a\mod b$ the remainder of the division of $a$ by $b$.
\end{nt}


In the next theorem we describe the special fibre of the scheme $\m C$ constructed in \S\ref{sec:ModConstr}. 

\begin{thm}[Regular SNC model]\label{thm:SNCModel}
The scheme $\m C\rightarrow O_K$ constructed in \S\ref{sec:ModConstr} is a regular model of $C$ with strict normal crossings; its special fibre $\mathcal{C}_s/k$ is described as follows:
\begin{enumerate}[label=(\arabic*)]
\item 
  Every $(\s,v)\in\Sigma$ gives a $1$-dimensional closed subscheme $\Gamma_v$ of multiplicity 
  $m_v$. The ring of rational functions of $\Gamma_v$ is isomorphic to $k_v(x)[y]/(y^{n_v}-\tilde{f_v}(x))$. If $n_v=2$, $u_v=0$, and $\tilde{f_v}\in k_v^2$, then $\Gamma_v\simeq\P_{k_v}^1\sqcup\P_{k_v}^1$, otherwise 
  $\Gamma_v$ is irreducible of genus $g(v)$.\label{regSNC1}
  \item Every $(\s,v)\in\Sigma$ with $n_v=1$ gives the closed subscheme $X_v\times_k\P^1_k$, of multiplicity $e_v$, where $X_v\times_k\{0\}\subset\Gamma_v$ (the $\P_k^1$s are open-ended).\label{regSNC2}
    \item Every non-maximal $(\s,v)\in\Sigma$, with $(\s,v)<(\t,w)$,
    gives the closed subscheme $Y_v\times_k \P^1\big(\epsilon_v\gamma_v,s_v,s_v-\tfrac{p_v}{2}(\lambda_v-\tfrac{\deg v}{\deg w}\lambda_{w})\big)$ where $Y_v\times_k\{0\}\subset\Gamma_v$ and $Y_v\times_k\{\infty\}\subset\Gamma_w$.\label{regSNC3}
    \item \label{regSNC4} Every degree-minimal $(\s,v)\in\Sigma$ 
    gives the closed subscheme $Y_v^0\times_k \P^1(\epsilon_v\gamma_v^0,-s_v^0)$ where $Y_v^0\times_k\{0\}\subset\Gamma_v$ (the chains are open-ended). 
    \item \label{regSNC5} Finally, 
    the maximal element $(\s,v)\in\Sigma$ gives the closed subscheme $Y_v\times_k \P^1(\epsilon_v\gamma_v,s_v)$ where $Y_v\times_k\{0\}\subset\Gamma_v$ (the chains are open-ended). 
\end{enumerate}
If $\Gamma_v$ is reducible, the two points in $Y_v\times_k\{0\}$ (and $Y_v^0\times_k\{0\}$ if $(\s,v)$ is degree-minimal) belong to different irreducible components of $\Gamma_v$. Similarly, if $(\s,v)$ is not maximal with $(\s,v)<(\t,w)$, and $\Gamma_w$ is reducible, then the two points of $Y_v\times_k\{\infty\}$ belong to different irreducible components of $\Gamma_w$.
\proof
The description of the special fibre of $\m C$ follows from its explicit construction developed in \S\ref{sec:ModConstr} (see especially \S\ref{subsec:Components}). We highlight the key points.

\ref{regSNC1} Each proper MacLane cluster $(\s,v)$ gives the $1$-dimensional closed subscheme $\Gamma_v$ of $\m C_s$, coming from the $\ast$-face $F_v$. The open subscheme $\mathring\Gamma_v$ (Notation \ref{nt:mathringGammav}) of $\Gamma_v$ is isomorphic to
\[\Spec \frac{k_v\big[X^{\pm 1},Y,\prod_{(\t,w)<(\s,v)}(\phi_w|_v)^{-1}\big]}{(Y^{n_v}-X^{\io_v/b_v-\ell_v \epsilon_v\nu_v}f|_v)},\]
where the product runs through all proper children of $(\s,v)$. The multiplicity of $\Gamma_v$ in $\m C_s$ is given by $e_v d_0$, where $d_0$ is the denominator of the slope $s_1^{L_v}$. We noticed in \S\ref{subsec:Components} that $n_v=2/d_0$. 
If $n_v=1$, then $\Gamma_v\simeq\P^1_{k_v}$. Suppose $n_v=2$. We want to show that the ring of rational functions of $\Gamma_v$ is 
\begin{equation}\label{eqn:Gammavfunctionring}
    k_v(X)[Y]/(Y^{n_v}-\tilde f_v(X)).
\end{equation}

If $(\s,v)$ is degree-minimal, then $\io_v=2-p_v^0$ from (\ref{eqn:centrechoice}). If $(\s,v)$ is not degree-minimal, then there exists a child $(\t,w)<(\s,v)$ with $\phi_w=\phi_v$; in particular, $\io_v=\i_w$, $f_w=f_v$ and so
\[\io_v/b_v-\ell_v \epsilon_v\nu_v=\tfrac{f_v|\t|}{f_wb_v\deg v}-\ell_v\nu_v\epsilon_w.\]

Now let $(\t,w)<(\s,v)$ with $\phi_w\neq\phi_v$. Theorem \ref{thm:ChildrenReduction} implies that \[\ord_{\phi_w|_v}(f|_v)=|\t|/\deg w.\]
Note that $\epsilon_w=e_v$ and $f_v\deg w=f_wb_v\deg v$ by Lemma \ref{lem:keypolyreduction}. Then 
\[\tfrac{|\t|}{\deg w}\notin 2\Z\quad\text{ if and only if }\quad\tfrac{f_v|\t|}{f_wb_v\deg v}-\ell_v\nu_v\epsilon_w\notin2\Z.\]

Let $[v_0,\dots, v_{h}(\phi_h)=\lambda_h]$ be the cluster chain for $v$. Let $f_\s=\prod_{r\in\roots\setminus\s}(x-r)$. The Newton polygon $N_{v_{h-1},\phi_v}(f_\s)$ has only slopes $>-\lambda_v$. Then $f_\s|_v=u|_v$, where $u=f_\s\mod\phi_v$.

The observations above, together with Proposition \ref{prop:ReductionDegreeProduct}\ref{item:ReductionProduct}, imply that (\ref{eqn:Gammavfunctionring}) is the ring of rational functions of $\Gamma_v$.

The subscheme given by $(\s,v)\in\Sigma$ in \ref{regSNC2} is the closure of 
\begin{equation}\label{eqn:Chains}
\textstyle\bigcup_{i=1}^{r_E}\big(U_{E,i}\cap\{Z=0\}\big),
\end{equation}
when $E=L_v$.
The subscheme given by $(\s,v)\in\Sigma$ in \ref{regSNC3} or \ref{regSNC5} is the closure of 
(\ref{eqn:Chains}) when $E=V_v$. Note that $(V_v)_\Z(\Z)+1=p_v/\gamma_v$. The subscheme given by a degree-minimal $(\s,v)\in\Sigma$ in \ref{regSNC4} is the closure of
(\ref{eqn:Chains}) when $E=V_v^0$. Note that $(V_v^0)_\Z(\Z)+1=p_v^0/\gamma_v^0$.
\endproof
\end{thm}

\begin{rem}
Let $(\s,v)\in\Sigma$. Note that 
\begin{enumerate}[label=(\roman*)]
    \item if $\Gamma_v$ is reducible then $p_v/\gamma_v=2$.
    \item if $(\s,v)<(\t,w)$ and $\Gamma_{w}$ is reducible, then $p_v/\gamma_v=2$.
    \item if $(\s,v)$ is degree-minimal and $\Gamma_v$ is reducible then $p_v^0/\gamma_v^0=2$.
\end{enumerate}
\end{rem}

\appendix
\section{Pseudo-valuations}\label{appendix:pseudo-valuations}

Let $A$ be an integral domain (with identity). Let $\hat\Q=\Q\cup\{\infty\}$. The ordering and the group law on $\Q$ are canonically extended to the set $\hat\Q$.

\begin{defn}
A map $v:A\rightarrow \hat\Q$ is called \textit{pseudo-valuation} (\textit{of $A$}) if
\begin{enumerate}[label=(\alph*)]
    \item $v(ab)=v(a)+v(b)$,
    \item $v(a+b)\geq\min\{v(a),v(b)\}$,
\end{enumerate}
for any $a,b\in A$.
A pseudo-valuation $v$ is said \textit{valuation} if it also satisfies
\begin{enumerate}
    \item[(c)] $v(a)=\infty$ if and only if $a=0$;
\end{enumerate}
we call it \textit{infinite pseudo-valuation} otherwise. 
\end{defn}

\begin{defn}
Let $v:A\rightarrow \hat\Q$ be a pseudo-valuation.
\begin{itemize}
    \item The \textit{valuation group} of a pseudo-valuation $v:A\rightarrow \hat\Q$, denoted $\Gamma_v$, is the subgroup generated by the subset $v(A)\cap \Q$ of $\Q$. Note that if $\Z\subseteq v(A)$, then $\Gamma_v=v(A)\cap \Q$.
    \item $v$ is \textit{discrete} if there exists $e\in\Z_+$ such that $e\Gamma_v=\Z$. If that happens, then $e_v=e$ is said ramification index of $v$. 
    \item The \textit{valuation ring} $O_v$ of a pseudo-valuation $v:A\rightarrow \hat\Q$ is the set of $a\in A$ with $v(a)\geq 0$. 
    \item The \textit{residue ring} of $v$ is the quotient of $O_v$ by the prime ideal $O_v^+$ consisting of the elements $a\in A$ with $v(a)>0$. 
    \item If $v$ is a valuation, the \textit{residue field} of $v$ is the residue ring of the valuation of $\mathrm{Frac}(A)$ that $v$ induces. 
\end{itemize}
\end{defn}
\section{Explicit matrices}\label{appendix:Matrices}
In this section we explicitly describe the matrices introduced in \S\ref{subsec:Matrices}.
Recall the notation of \S\ref{subsec:Matrices}.
Suppose the permutation $\tau$ equals the identity. 
Let $m_j'$, for $j=0,\dots,h$, be the quantities defined in Lemma \ref{lem:gammadecomposition}.
Then 
\[M_{E,i}^\R=\scriptsize\lb\!\!\!\!\!\begin{array}{ccccccccccc}
     e_1 \!\!\!\!\!\!& -h_2m_1' \!\!\!\!\!\!& \dots \!\!\!\!\!\!& -h_{h-1}m_1' \!\!\!\!\!\!& 0 \!\!\!\!\!\!& -\beta_{h+1}m'_1 \!\!\!\!\!\!& \dots \!\!\!\!\!\!& \dots \!\!\!\!\!\!& -\beta_{n}m'_1 \!\!\!\!\!\!& 0 \!\!\!\!\!\!& 0\cr
     0 \!\!\!\!\!\!& e_2 \!\!\!\!\!\!& \ddots \!\!\!\!\!\!& -h_{h-1}m_2' \!\!\!\!\!\!& 0 \!\!\!\!\!\!& -\beta_{h+1}m'_2 \!\!\!\!\!\!&\dots \!\!\!\!\!\!& \dots \!\!\!\!\!\!& -\beta_{n}m'_2 \!\!\!\!\!\!& 0 \!\!\!\!\!\!& 0 \cr
     \vdots \!\!\!\!\!\!& \ddots \!\!\!\!\!\!& \ddots \!\!\!\!\!\!& \vdots \!\!\!\!\!\!& \vdots \!\!\!\!\!\!& \vdots \!\!\!\!\!\!&\ddots \!\!\!\!\!\!&\ddots \!\!\!\!\!\!& \vdots \!\!\!\!\!\!& \vdots \!\!\!\!\!\!& \vdots \cr
        \vdots \!\!\!\!\!\!& \vdots \!\!\!\!\!\!& \ddots \!\!\!\!\!\!& e_{h-1} \!\!\!\!\!\!& 0 \!\!\!\!\!\!& -\beta_{h+1}m'_{h-1} \!\!\!\!\!\!& \dots \!\!\!\!\!\!& \dots \!\!\!\!\!\!& -\beta_{n}m'_{h-1} \!\!\!\!\!\!& 0 \!\!\!\!\!\!& 0  \cr
        \vdots \!\!\!\!\!\!& \vdots \!\!\!\!\!\!& \ddots \!\!\!\!\!\!& 0 \!\!\!\!\!\!& \delta a_x \!\!\!\!\!\!& -\beta_{h+1}' \!\!\!\!\!\!& \dots \!\!\!\!\!\!& \dots \!\!\!\!\!\!& -\beta_{n}' \!\!\!\!\!\!& d_ib_x \!\!\!\!\!\!& -d_{i+1}b_x\cr
        \vdots \!\!\!\!\!\!& \vdots \!\!\!\!\!\!& \ddots \!\!\!\!\!\!& \vdots \!\!\!\!\!\!& 0 \!\!\!\!\!\!& 1 \!\!\!\!\!\!& 0 \!\!\!\!\!\!& \dots \!\!\!\!\!\!& 0\!\!\!\!\!\!& 0 \!\!\!\!\!\!& 0\cr
        \vdots \!\!\!\!\!\!& \vdots \!\!\!\!\!\!& \ddots \!\!\!\!\!\!& \vdots \!\!\!\!\!\!& \vdots \!\!\!\!\!\!& 0 \!\!\!\!\!\!& 1 \!\!\!\!\!\!& \ddots \!\!\!\!\!\!& 0 \!\!\!\!\!\!& 0 \!\!\!\!\!\!& 0\cr
        \vdots \!\!\!\!\!\!& \vdots \!\!\!\!\!\!& \ddots \!\!\!\!\!\!& \vdots \!\!\!\!\!\!& \vdots \!\!\!\!\!\!& \vdots \!\!\!\!\!\!& \ddots \!\!\!\!\!\!& \ddots \!\!\!\!\!\!& \ddots \!\!\!\!\!\!& \vdots \!\!\!\!\!\!& \vdots\cr
        0 \!\!\!\!\!\!& 0 \!\!\!\!\!\!& \dots \!\!\!\!\!\!& 0 \!\!\!\!\!\!& 0 \!\!\!\!\!\!& 0 \!\!\!\!\!\!& 0 \!\!\!\!\!\!& \ddots \!\!\!\!\!\!& 1 \!\!\!\!\!\!& 0 \!\!\!\!\!\!& 0 \cr
        0 \!\!\!\!\!\!& 0 \!\!\!\!\!\!& \dots \!\!\!\!\!\!& 0 \!\!\!\!\!\!& \delta a_y \!\!\!\!\!\!& 0 \!\!\!\!\!\!& 0 \!\!\!\!\!\!& \ddots \!\!\!\!\!\!& 0 \!\!\!\!\!\!& d_ib_y \!\!\!\!\!\!& -d_{i+1}b_y \cr
        -h_1m_0' \!\!\!\!\!\!&  -h_2m_0' \!\!\!\!\!\!& \dots \!\!\!\!\!\!& -h_{h-1}m_0' \!\!\!\!\!\!& \delta a_z \!\!\!\!\!\!& -\beta_{h+1}m'_0 \!\!\!\!\!\!& \dots \!\!\!\!\!\!& \dots \!\!\!\!\!\!& -\beta_{n}m'_0 \!\!\!\!\!\!& \tfrac{n_i}{\delta \epsilon_v} \!\!\!\!\!\!& -\tfrac{n_{i+1}}{\delta \epsilon_v}
\end{array}\!\!\!\!\!\rb\]\normalsize
where for any $o=h+1,\dots,n$ we have
\[\beta_o=\begin{cases}
0&\text{if }\mu_o>v_E,\\
\epsilon_vv(\psi_o)&\text{otherwise,}
\end{cases}\qquad
\beta_o'=\begin{cases}
v(\psi_{o})/\lambda_v& \text{if }\mu_o>v_E,\\ 
0&\text{otherwise.}
\end{cases}\]
Therefore
\[\det M_{E,i}^\R=\footnotesize
\left|\!\!\!\begin{array}{ccccccc}
     e_1 \!\!&  \!\!\!&  \!\!\!&  \!\!\!&  \!\!\!&  \!\!\!&    \cr
      \!\!& e_2 \!\!\!&  \!\!\!&  \!\!\!&  \text{\Huge $\ast$\footnotesize}\!\!\!& \!\!\!&    \cr
      \!\!&  \!\!\!& \ddots \!\!\!&  \!\!\!&  \!\!\!& \!\!\!&   \cr
         \!\!&  \!\!\!&  \!\!\!& e_{h-1} \!\!&  \!\!\!&  \!\!\!&    \\[0.8em]
         \!\!&  \!\!\!& \!\!\!&  \!\!\!& 1 \!\!\!&  \!\!\!&  \\[-0.3em]
         \!\!&  \text{\Huge $0$\footnotesize}\!\!\!&  \!\!\!&  \!\!\!& \!\!\!&  \ddots \!\!\!& \\[-0.2em]
         \!\!&  \!\!\!&  \!\!\!&  \!\!\!&  \!\!\!&  \!\!\!&  1
\end{array}\!\!\!\right|
\cdot\left|\!\!\!\begin{array}{ccc}
    \delta v_x \!\!\!\!\!& d_i w_x \!\!\!\!\!& -d_{i+1}w_x  \cr
    \delta v_y \!\!\!\!\!& d_i w_y \!\!\!\!\!& -d_{i+1}w_y \cr
    \delta v_z \!\!\!\!\!& \tfrac{n_i}{\delta \epsilon_v} \!\!\!\!\!& -\tfrac{n_{i+1}}{\delta \epsilon_v} 
\end{array}\!\!\!\right|\normalsize=1\]

Furthermore, $T_h$ and $T$ equal respectively
\scriptsize\[
\begingroup
\begin{pmatrix}
     1 \!\!\!& \!\!\!& \!\!\!& \delta a_z c_1 \!\!\!& \!\!\!\!\!\!\! \!\!\!\!\!\!\!\!\!\!\!\!\!\!& \tfrac{n_ic_1}{\delta \epsilon_v} \!\!\!& -\tfrac{n_{i+1}c_1}{\delta \epsilon_v}\\[-1em]
      \!\!\!& \ddots \!\!\!& \!\!\!& \vdots \!\!\!& & \vdots \!\!\!& \vdots \\
     \!\!\!& \!\!\!& 1 \!\!\!& \delta a_z c_{h-1} \!\!\!& \!\!\!\!\!\!\! \!\!\!\!\!\!\!\!\!\!\!\!\!\!& \tfrac{n_ic_{h-1}}{\delta \epsilon_v} \!\!\!& -\tfrac{n_{i+1}c_{h-1}}{\delta \epsilon_v}  \cr
     \!\!\!& \!\!\!& \!\!\!& 1 \!\!\!&\!\!\!\!\!\!\! \!\!\!\!\!\!\!\!\!\!\!\!\!\! & \!\!\!& \cr
     \!\!\!& \!\!\!& \!\!\!&\!\!\! & \!\!\!\!\!\!\! \!\!\!\ddots\!\!\!\!\!\!\!\!\!\!\! & \!\!\!&\cr
    \!\!\!& \!\!\!& \!\!\!& \!\!\!& \!\!\!\!\!\!\! \!\!\!\!\!\!\!\!\!\!\!\!\!\!& 1 \!\!\!& \\[1em]
    \!\!\!& \!\!\!& \!\!\!& \!\!\!& \!\!\!\!\!\!\! \!\!\!\!\!\!\!\!\!\!\!\!\!\!& \!\!\!& 1 
\end{pmatrix},\,
\begin{pmatrix}
     1 \!\!\!& \!\!\!& \!\!\!& \!\!\!&\!\!\!& \\[-0.5em]
      \!\!\!& \ddots \!\!\!& \!\!\!& \!\!\!& \!\!\!&\\
     \!\!\!& \!\!\!&1 \!\!\!& \!\!\!&\!\!\!&\\ 
    \!\!\!&\!\!\!& \!\!\!& 1 \!\!\!& \!\!\!& \tfrac{k_i}{\delta} \!\!\!& -\tfrac{k_{i+1}}{\delta}  \\[-0.7em]
     \!\!\!&\!\!\!& \!\!\!& \!\!\!& \ddots \!\!\!& \\[0.3em]
    \!\!\!&\!\!\!& \!\!\!& \!\!\!& \!\!\!& 1 \!\!\!& \\[0.7em]
    \!\!\!&\!\!\!& \!\!\!& \!\!\!& \!\!\!& \!\!\!& 1 
\end{pmatrix},
\endgroup
\]\normalsize
and $T_h^{-1}$ and $T^{-1}$ are respectively
\scriptsize\[
\begingroup
\begin{pmatrix}
     1 \!\!\!& \!\!\!& \!\!\!& -\delta a_z c_1 \!\!\!& \!\!\!\!\!\!\!\!\!\!\!\!\!\!\!\!\!\!\!\!\!& -\tfrac{n_ic_1}{\delta \epsilon_v} \!\!\!& \tfrac{n_{i+1}c_1}{\delta \epsilon_v}\\[-1em]
      \!\!\!& \ddots \!\!\!& \!\!\!& \vdots \!\!\!& \!\!\!& \vdots \!\!\!& \vdots \cr
     \!\!\!& \!\!\!& 1 \!\!\!& -\delta a_z c_{h-1} \!\!\!& \!\!\!\!\!\!\!\!\!\!\!\!\!\!\!\!\!\!\!\!\!& -\tfrac{n_ic_{h-1}}{\delta \epsilon_v} \!\!\!& \tfrac{n_{i+1}c_{h-1}}{\delta \epsilon_v}  \cr
     \!\!\!& \!\!\!& \!\!\!& 1 \!\!\!& \!\!\!\!\!\!\! \!\!\!\!\!\!\!\!\!\!\!\!\!\!& \!\!\!& \cr
     \!\!\!& \!\!\!& \!\!\!& \!\!\!& \!\!\!\!\!\!\!\ddots \!\!\!\!\!\!\!\!& \!\!\!&\cr
    \!\!\!& \!\!\!& \!\!\!& \!\!\!& \!\!\!\!\!\!\! \!\!\!\!\!\!\!\!\!\!\!\!\!\!& 1 \!\!\!& \\[1em]
    \!\!\!& \!\!\!& \!\!\!& \!\!\!& \!\!\!\!\!\!\! \!\!\!\!\!\!\!\!\!\!\!\!\!\!& \!\!\!& 1 
\end{pmatrix},\,
\begin{pmatrix}
     1 \!\!\!& \!\!\!& \!\!\!& \!\!\!&\!\!\!& \\[-0.5em]
      \!\!\!& \ddots \!\!\!& \!\!\!& \!\!\!& \!\!\!&\\
     \!\!\!& \!\!\!&1 \!\!\!& \!\!\!&\!\!\!&\\ 
    \!\!\!&\!\!\!& \!\!\!& 1 \!\!\!& \!\!\!& -\tfrac{k_i}{\delta} \!\!\!& \tfrac{k_{i+1}}{\delta}  \\[-0.7em]
     \!\!\!&\!\!\!& \!\!\!& \!\!\!& \ddots \!\!\!& \\[0.3em]
    \!\!\!&\!\!\!& \!\!\!& \!\!\!& \!\!\!& 1 \!\!\!& \\[0.7em]
    \!\!\!&\!\!\!& \!\!\!& \!\!\!& \!\!\!& \!\!\!& 1 
\end{pmatrix},
\endgroup
\]\normalsize
where all missing entries are $0$s.

Finally, the vectors $\tilde\nu_o$, for $1\leq o\leq n$, first $n$ rows of the matrix $M_{E,i}^\R$, are
\small\[\tilde\nu_{o}=
\begin{cases}
\big(0,\dots,0,\tfrac{1}{e_o},\tfrac{h_{o+1}m_{o}}{e_{v_{o+1}}},\dots, \tfrac{h_{h-1}m_{o}}{e_{v_{h-1}}},0,\tfrac{\beta_{h+1}m_{o}}{\epsilon_v}\dots,\tfrac{\beta_{n}m_{o}}{\epsilon_v},0,0\big)&\mbox{if }1\leq o<h,\\
\tfrac{1}{\delta}(0,\dots,0,b_y,\beta_{h+1}'b_y,\dots,\beta_{n}'b_y,b_x,0)&\mbox{if }o=h,\\
(0,\dots,0,1,0,\dots,0)=\varepsilon_o&\mbox{if }h<o\leq n,\\
\end{cases}\]\normalsize
\section{MacLane clusters fan}\label{sec:appendix:fan}
We want to show that the cones constructed in \S\ref{subsec:TorEmbedding} form a fan. Let $h=1,\dots,n$. Recall the degree-minimal MacLane cluster $(\s_h,\mu_h)$. Let 
\[[v_0,v_1(\phi_1)=\lambda_1,\dots,v_{m-1}(\phi_{m-1})=\lambda_{m-1},v_m(\phi_m)=\lambda_m]\]
be the cluster chain for $\mu_h$. Let $c\in\R$, with $c>\lambda_{m-1}$ if $m>1$. Define the valuation $v_{h,c}:K(x)\rightarrow \hat\R$ given on $K[x]^\ast$ by
\[v_{h,c}\big(\textstyle\sum_j c_j\psi_h^j\big)=\min_j\big(v_{m-1}(c_j)+jc\big),\quad c_j\in K[x],\,\deg(c_j)<\deg(\psi_h).\]
Note that when $m=1$, then $\deg \psi_h=1$ and so $v_0(c_j)=v_K(c_j)$.

Let $(\s,v)$ be a proper MacLane cluster with centre $\phi_v=\psi_h$, and let $E$ be the $h$-edge $L_v$, $V_v$, or $V_v^0$ if $(\s,v)$ is degree-minimal. 
Recall the notation of \S\ref{subsec:Matrices}. 
Let $r=r_E$ and $\delta=\delta_E$.

\begin{lem}\label{lem:tildeomegalinearcomb}
For any $i=0,\dots,r+1$, there exist $\alpha,\beta\in\Q_{\geq 0}$, such that
\[\tilde\omega_i=\alpha\tilde\omega_0+\beta\tilde\omega_{r+1}.\]
\proof
If $i=0$ or $i=r+1$, the statement is trivial. Then assume $1\leq i\leq r$.
Since $n_0 d_i> n_id_0$ and $n_id_{r+1}>n_{r+1}d_{i}$, there exist $\alpha_i,\beta_i\in\Q_+$ such that
\[\alpha_i n_id_0+\beta_i n_id_{r+1}=\alpha_i n_0d_i+\beta_i n_{r+1}d_{i}.\]
Define $e=\tfrac{d_i}{d_0\alpha_i+d_{r+1}\beta_i}$, $\alpha=e\alpha_i$, $\beta=e\beta_i$. The lemma follows from (\ref{eqn:tildeomega}).
\endproof
\end{lem}

\begin{lem}\label{lem:vhc}
Let $c\in\R$ and $v_{h,c}:K[x]\rightarrow \hat\R$ as above. If
\begin{enumerate}[label=(\roman*)]
    \item \label{item:vhcinner}$(\s,v)<(\t,w)$, $E=V_v$, and $w(\psi_h)<c<\lambda_v$, or
    \item \label{item:vhcmaximal}$(\s,v)$ maximal, $E=V_v$, and $c<\lambda_v$, or
    \item \label{item:vhcminimal}$(\s,v)$ degree-minimal, $E=V_v^0$, and $c>\lambda_v$, 
\end{enumerate}
then $v_{h,c}(\gamma_{j,E})=0$
for any $j=1,\dots,n$, $j\neq h$.
\proof
Let $j=1,\dots,n$, $j\neq h$. Expand
\[\psi_j=\sum_{\i=1}^d c_\i\psi_h^\i,\quad c_\i\in K[x],\,c_d\neq 0,\,\deg c_\i<\deg\psi_h.\]
If $j=\tau(o)$ for some $o<m$, then $v_{h,c}(\psi_j)=v_{m-1}(\psi_j)$. It follows from Lemma \ref{lem:gammapival} that $v_{h,c}(\gamma_{j,E})=0$. Hence assume $j\neq\tau(o)$ for all $o<m$.

\ref{item:vhcinner} 
Assume $(\s,v)<(\t,w)$, $E=V_v$, and $w(\psi_h)<c<\lambda_v$.
Suppose $\mu_j\geq v$. Lemma \ref{lem:gammaoEval0} implies that $c_d=1$ and $v(\psi_j)=v(\psi_h^d)=d\lambda_v$. Since $c<\lambda_v$ we have $v_{h,c}(\psi_j)=dc$, by definition. Then $v_{h,c}(\gamma_{j,E})=0$. Suppose $\mu_j\not\geq v$. Therefore 
\[v(\psi_j)=w(\psi_j)\leq v_{h,c}(\psi_j)\leq v(\psi_j).\] 
where the first equality follows from  Lemma \ref{lem:gammaoEval0}. Hence $v_{h,c}(\gamma_{j,E})=0$. 

\ref{item:vhcmaximal} 
Assume $(\s,v)$ maximal, $E=V_v$, and $c<\lambda_v$. Then $\mu_j\geq v$. Lemma \ref{lem:gammaoEval0} implies that $c_d=1$ and $v(\psi_j)=v(\psi_h^d)=d\lambda_v$. It follows that $v_{h,c}(\psi_j)=dc$ as $c<\lambda_v$. Therefore $v_{h,c}(\gamma_{j,E})=0$.

\ref{item:vhcminimal}
Assume $(\s,v)$ degree-minimal, $E=V_v^0$, and $c>\lambda_v$. Recall the definition of $v_E$. Then $\mu_j\not\geq v_E$ and $v_E(\psi_j)\geq v_{h,c}(\psi_j)\geq v(\psi_j)$. It follows from (\ref{eqn:vEpsio}) that $v_{h,c}(\psi_j)= v(\psi_j)$. Thus $v_{h,c}(\gamma_{j,E})=0$. 
\endproof
\end{lem}

\begin{lem}\label{lem:tildeomegavaluation}
For any $\tilde\omega\in\sigma_{E,i,i+1}\setminus\sigma_{E,i+1}$, there exists $c\in\R$, with $c>\lambda_{m-1}$ if $m>1$, so that
\[\tilde\omega=e(v_{h,c}(\psi_1),\dots,v_{h,c}(\psi_n),C,1),\]
for some $e\in \R_+$, $C\in\R$.
In particular,
\begin{enumerate}[label=(\roman*)]
    \item \label{item:omegainner}if $(\s,v)<(\t,w)$ and $E=V_v$, then $w(\psi_h)<c<\lambda_v$;
    \item \label{item:omegamaximal}if $(\s,v)$ maximal and $E=V_v$, then $c<\lambda_v$;
    \item \label{item:omegaminimal}if $(\s,v)$ degree-minimal and $E=V_v^0$, then $c>\lambda_v$; 
    \item \label{item:omegaLv}if $E=L_v$, then $c=\lambda_v$.
\end{enumerate}
\proof
From Lemma \ref{lem:fan}, the statement is true for $\tilde\omega=\tilde\omega_0$. So suppose $\tilde \omega\neq\tilde\omega_0$. Lemma \ref{lem:tildeomegalinearcomb} implies that $\tilde\omega=\alpha\tilde\omega_0+\beta\tilde\omega_{r+1}$ for some $\alpha,\beta\in\R_{+}$. Let $e\in\R_+$, $c\in\R$ as follows
\[e=\alpha\delta\epsilon_h d_0+\beta\delta\epsilon_h d_{r+1},\qquad c=\frac{\alpha n_0a_y+\beta n_{r+1}a_y}{e}-a_zb_y.\]
From the definition of $\tilde\omega_0$ and $\tilde\omega_{r+1}$ in (\ref{eqn:tildeomega}) we have 
\[\tilde\omega=e(c\tilde\alpha_{h1}+\tilde\alpha_{\pi1},\dots,c\tilde\alpha_{hn}+\tilde\alpha_{\pi n},C,1),\]
for some $C\in\R$. Furthermore, $c$ satisfies the inequalities of cases \ref{item:omegainner}-\ref{item:omegaLv} by Lemma \ref{lem:fan}. In particular, $c>\lambda_{m-1}$ if $m>1$.
From (\ref{eqn:psioingammas}), Lemma \ref{lem:vhc} concludes the proof.
\endproof
\end{lem}

\begin{rem}
Note that the element $c\in\R$ in Lemma \ref{lem:tildeomegavaluation} is uniquely determined by the vector $\tilde\omega$. Indeed, $c$ equals the division of the $h$-th coordinate of $\tilde\omega$ by its last coordinate.
\end{rem}

Let $i=0,\dots,r+1$, with $i\leq r$ if $E$ is outer. Let $c_i=\tfrac{n_i}{\delta e_{v_-} d_i}a_y-a_zb_y$. We define the valuation $w_{E,i}:K[x]\rightarrow \hat\Q$ by $w_{E,i}(g)=v_{h,c_i}(g)$ for any $g\in K[x]^\ast$. In other words, $w_{E,i}$ is given on $K[x]^\ast$ by
\begin{equation}\label{eqn:wEi}
w_{E,i}\big(\textstyle\sum_j a_j\psi_h^j\big)=\min_j\big(v_-(a_j)+jc_i\big),
\end{equation}
where $a_j\in K[x]$, $\deg(a_j)<\deg(\psi_h)$.
In fact, $w_{E,i}$ is the MacLane valuation 
\[w_{E,i}=\big[v_-,w_{E,i}(\psi_h)=\tfrac{n_i}{\delta e_{v_-} d_i}a_y-a_zb_y\big],\]
except possibly when $(\s,v)$ is maximal, $E=V_v$ and $1\leq i \leq r$. 
Lemma \ref{lem:tildeomegavaluation} implies that
\[\tilde\omega_i=\delta e_{v_-} d_i(w_{E,i}(\psi_1),\dots, w_{E,i}(\psi_n),C,1),\]
for some $C\in\Q$. We denote $C$ by $w_{E,i}(y)$.

\begin{thm}
The set of cones $\Sigma$ defined in \S\ref{subsec:TorEmbedding} is a fan.
\proof
For any $\ast$-edge $E$ let 
\[\sigma_{E,0,r_E+1}=\bigcup_{i=0}^{r_E}\sigma_{E,i,i+1}.\]
By Lemma \ref{lem:tildeomegalinearcomb} it suffices to show that the set
\[\Sigma'=\{\sigma_0\}\cup\bigcup_{E\text{ $\ast$-edge}}(\sigma_{E,0}\cup\sigma_{E,r_E+1}\cup\sigma_{E,0,r_E+1})\]
is a fan. But this follows from Lemma \ref{lem:tildeomegavaluation}.
\endproof
\end{thm}

\end{document}